\documentclass{amsart}
\usepackage{amsmath, amssymb, amsthm, mathrsfs, amscd}
\usepackage{multicol}

\theoremstyle{plain} 
 \newtheorem{thm}{Theorem}[section]
 \newtheorem{lem}[thm]{Lemma}
 \newtheorem{cor}[thm]{Corollary}
 \newtheorem{prop}[thm]{Proposition}
 \newtheorem{claim}[thm]{Claim}

\theoremstyle{definition}
  \newtheorem{defn}[thm]{Definition}
  \newtheorem{ex}[thm]{Example}
  \newtheorem{notation}[thm]{Notation}

\theoremstyle{remark}
  \newtheorem{rem}[thm]{Remark}

\newcommand{\comm}{{\rm Comm}}
\newcommand{\stab}{{\rm Stab}}

\newcommand{\aut}{{\rm Aut}}
\newcommand{\ad}{{\rm Ad}}
\newcommand{\isom}{{\rm Isom}}
\newcommand{\cal}{\mathcal}
\newcommand{\ci}[2]{\cite[#1]{#2}}
\renewcommand{\c}{\curvearrowright}

\newcommand{\diag}{{\rm diag}}
\newcommand{\faq}[2]{\noindent#1\quad \dotfill\quad\makebox[1em][r]{#2}\par}

\begin{document}

\title[Rigidity of amalgamated free products]{Rigidity of amalgamated free products\\ in measure equivalence}
\author{Yoshikata Kida}
\address{Department of Mathematics, Kyoto University, 606-8502 Kyoto, Japan}
\email{kida@math.kyoto-u.ac.jp}
\date{February 17, 2009, revised on February 24, 2011}

\begin{abstract}
A discrete countable group $\Gamma$ is said to be ME rigid if any discrete countable group that is measure equivalent to $\Gamma$ is virtually isomorphic to $\Gamma$.
In this paper, we construct ME rigid groups by amalgamating two groups satisfying rigidity in a sense of measure equivalence.
A class of amalgamated free products is introduced, and discrete countable groups which are measure equivalent to a group in that class are investigated.
\end{abstract}

\maketitle

\tableofcontents

\section{Introduction}\label{sec-int}

The aspect of rigidity is one of the major focuses of recent advance in the study of measure equivalence.
Among other things, fascinating rigidity is discovered for higher rank lattices and for mapping class groups of compact orientable surfaces (see \cite{furman-mer} and \cite{kida-mer}, respectively).
The aim of this paper is to construct rigid groups by amalgamating two rigid groups, where rigidity is formulated in terms of measure equivalence and orbit equivalence.
All amalgamated free products in the theorems stated below are obtained by amalgamating groups over their infinite subgroups.
In particular, we do not deal with free products.
In general, free products do not satisfy such superrigidity even if their factor subgroups are rigid groups mentioned already.
We refer to \cite{al-gab} and \cite{ipp} for rigidity of strong type for free products.
We shall recall the definition of measure equivalence.
Throughout the paper, we refer to discrete countable groups as discrete groups for simplicity.

\begin{defn}[\ci{0.5.E}{gromov-as-inv}]\label{defn-me}
Two discrete groups $\Gamma$ and $\Lambda$ are said to be {\it measure equivalent (ME)} if one has a standard Borel space $(\Sigma, m)$ with a $\sigma$-finite positive measure and a measure-preserving action of $\Gamma \times \Lambda$ on $(\Sigma, m)$ such that there exist Borel subsets $X, Y\subset \Sigma$ with the equality
\[\Sigma =\bigsqcup_{\gamma \in \Gamma}\gamma Y=\bigsqcup_{\lambda \in \Lambda}\lambda X\]
up to $m$-null sets.
In this case, the space $(\Sigma, m)$ equipped with the action of $\Gamma \times \Lambda$ is called a {\it coupling} of $\Gamma$ and $\Lambda$. 
The ratio $m(X)/m(Y)$ is called the {\it coupling constant} for the coupling $(\Sigma, m)$.
When $\Gamma$ and $\Lambda$ are ME, we write $\Gamma \sim_{\rm ME}\Lambda$.
\end{defn}

Indeed, ME defines an equivalence relation between discrete groups (see Section 2 in \cite{furman-mer}).
A basic problem is to determine the class of discrete groups ME to a given group.
Any two lattices in a locally compact second countable group are ME.
Any two virtually isomorphic groups are ME, where two discrete groups are said to be {\it virtually isomorphic} if they are isomorphic up to finitely many operations taking finite index subgroups and taking the quotients by finite normal subgroups.
We recommend the reader to consult \cite{furman-survey}, \cite{gab-survey} and \cite{shalom-survey} for basic knowledge and recent achievements related to measure equivalence.

One purpose of this paper is to describe a group ME to an amalgamated free product $\Gamma =\Gamma_1\ast_A\Gamma_2$ when some conditions are imposed on the inclusions $A<\Gamma_1$ and $A<\Gamma_2$.
It is widely known that one can construct the simplicial tree $T$, called the Bass-Serre tree, associated with the decomposition of $\Gamma$, on which $\Gamma$ acts by simplicial automorphisms.
This tree $T$ gives us a geometric viewpoint in the study of algebraic structure of $\Gamma$.
Theorem \ref{thm-lqn-str} below tells us that in a certain case, if $\Lambda$ is a discrete group ME to $\Gamma$, then $\Lambda$ also acts on $T$ so that the structure of $\Lambda$ can be understood through the Bass-Serre theory \cite{serre}.

Let us introduce notation to state Theorem \ref{thm-lqn-str}.
For a group $G$ and a subgroup $H$ of $G$, we define the {\it left quasi-normalizer} of $H$ in $G$ as the subsemigroup
\[{\rm LQN}_G(H)=\{ \, g\in G \mid [H: gHg^{-1}\cap H]<\infty \, \}\]
of $G$ containing $H$.
Note that the equality ${\rm LQN}_G(H)=H$ holds if and only if each orbit for the action $H\c (G/H)\setminus \{ H\}$ defined by left multiplication consists of infinitely many points.

\begin{thm}\label{thm-lqn-str}
Let $\Gamma =\Gamma_1\ast_A\Gamma_2$ be an amalgamated free product of discrete groups such that for each $i=1, 2$,
\begin{itemize}
\item $\Gamma_i$ satisfies property (T); and
\item $|A|=\infty$, $[\Gamma_i: A]=\infty$ and ${\rm LQN}_{\Gamma_i}(A)=A$.
\end{itemize}
We denote by $T$ the Bass-Serre tree associated with the decomposition of $\Gamma$.
Then the following assertions hold:
\begin{enumerate}
\item If $\Lambda$ is a discrete group ME to $\Gamma$, then there exists a subgroup $\Lambda_+$ of $\Lambda$ of index at most two acting on $T$ without inversions so that the stabilizer of each simplex $s$ of $T$ for the action $\Lambda_+\c T$ is ME to the stabilizer of $s$ for the action $\Gamma \c T$.
\item In assertion (i), if the action $A\c \Sigma /\Lambda$ is furthermore ergodic, then $\Lambda_+$ is decomposed as an amalgamated free product $\Lambda_+=\Lambda_1\ast_B\Lambda_2$ with $\Gamma_i\sim_{\rm ME}\Lambda_i$ for each $i=1, 2$ and $A\sim_{\rm ME}B$. 
\end{enumerate}
\end{thm}

As a next step, we deduce a superrigidity result in terms of orbit equivalence on the assumption that the factor subgroups $\Gamma_1$, $\Gamma_2$ of an amalgamated free product $\Gamma_1\ast_A\Gamma_2$ satisfy rigidity in a sense of ME.
There is a close connection between orbit equivalence and measure equivalence, whose details will be reviewed in Section \ref{subsec-me-woe}.
We say that a Borel action of a discrete group $\Gamma$ on a measure space $(X, \mu)$ is {\it f.f.m.p.\ }if $\mu$ is a finite positive measure on $X$ and if the action is essentially free and preserves $\mu$.
Two ergodic f.f.m.p.\ actions $\Gamma \c (X, \mu)$ and $\Lambda \c (Y, \nu)$ are said to be {\it orbit equivalent (OE)} if there exists a Borel isomorphism $f$ between conull Borel subsets of $X$ and $Y$ such that $f_*\mu$ and $\nu$ are equivalent and the equality $f(\Gamma x)=\Lambda f(x)$ holds for a.e.\ $x\in X$.

A notable example of rigid groups in measure equivalence is a lattice in a non-compact connected simple Lie group $G$ with its center trivial and its real rank at least two. 
Furman \cite{furman-mer} proves that a discrete group is ME to such a lattice if and only if it is virtually isomorphic to a lattice in $G$.
Theorem \ref{thm-oe-rigidity} below shows superrigidity of an ergodic f.f.m.p.\ action of $\Gamma_1\ast_A\Gamma_2$ with $\Gamma_1$ and $\Gamma_2$ such lattices, where an ergodicity condition is imposed on the action of $A$.
Let us say that a measure-preserving action of a discrete group $A$ on a measure space is {\it aperiodic} if any finite index subgroup of $A$ acts on it ergodically.

\begin{thm}\label{thm-oe-rigidity}
Let $\Gamma =\Gamma_1\ast_A\Gamma_2$ be an amalgamated free product of discrete groups such that for each $i=1, 2$,
\begin{itemize}
\item $\Gamma_i$ is isomorphic to a lattice in a non-compact connected simple Lie group with its center trivial and its real rank at least two; and
\item $|A|=\infty$, $[\Gamma_i: A]=\infty$ and ${\rm LQN}_{\Gamma_i}(A)=A$.
\end{itemize}
Let $\Gamma \c (X, \mu)$ be an ergodic f.f.m.p.\ action such that the restriction $A\c (X, \mu)$ is aperiodic. If the action $\Gamma \c (X, \mu)$ is OE to an ergodic f.f.m.p.\ action $\Lambda \c (Y, \nu)$ of a discrete group $\Lambda$, then the cocycle $\alpha \colon \Gamma \times X\rightarrow \Lambda$ associated with the OE is cohomologous to the cocycle arising from an isomorphism from $\Gamma$ onto $\Lambda$.
In particular, the two actions $\Gamma \c (X, \mu)$ and $\Lambda \c (Y, \nu)$ are conjugate.
\end{thm}

We refer to Corollary \ref{cor-oer} for a more general form of this theorem.
If the action of $A$ is not assumed to be aperiodic, then the conclusion on the cocycle $\alpha$ does not hold in general.
A counterexample can be obtained by twisting an action of either $\Gamma_1$ or $\Gamma_2$.
This topic is discussed in Section \ref{sec-twist}.

Finally, we present new examples of ME rigid groups.
A discrete group $\Gamma$ is said to be {\it ME rigid} if any discrete group ME to $\Gamma$ is virtually isomorphic to $\Gamma$.
As shown in \cite{kida-mer}, mapping class groups of non-exceptional compact orientable surfaces are the first example of infinite ME rigid groups.

For a group $G$ and a subgroup $\Gamma$ of $G$, we denote by $\comm_G(\Gamma)$ the (relative) commensurator of $\Gamma$ in $G$ (see Definition \ref{defn-comm}).

\begin{thm}\label{thm-mer}
Let $G^0$ be a non-compact connected simple Lie group with its center trivial and its real rank at least two, and put $G=\aut(G^0)$.
Let $\Gamma$ be a lattice in $G^0$, and let $A$ be a subgroup of $\Gamma$ such that
\begin{itemize}
\item $|A|=\infty$, $[\Gamma : A]=\infty$ and ${\rm LQN}_{\Gamma}(A)=A$; and
\item the Dirac measure on the neutral element is the only probability measure on $\comm_G(\Gamma)$ that is invariant under conjugation by any element of $A$.
\end{itemize}
Then the amalgamated free product $\Gamma \ast_A\Gamma$ is ME rigid.
\end{thm}

It is also shown that any ergodic f.f.m.p.\ action of the group $\Gamma \ast_A\Gamma$ in Theorem \ref{thm-mer} is superrigid, where no ergodicity condition is imposed on the action of any proper subgroup.
As an application of Theorem \ref{thm-mer}, we obtain the following examples.
General examples of the same type are presented in Theorems \ref{thm-mer-tensor} and \ref{thm-mer-but}.

\begin{thm}\label{thm-mer-ex}
We define $\Gamma$ and $A$ as in either the following (a) or (b):
\begin{enumerate}
\item[(a)] Fix an isomorphism $f\colon \mathbb{R}^2\otimes \mathbb{R}^3\rightarrow \mathbb{R}^6$ between real vector spaces so that $f(\mathbb{Z}^2\otimes \mathbb{Z}^3)=\mathbb{Z}^6$.
Define $\Gamma$ and $A$ to be the groups of linear automorphisms on $\mathbb{R}^6$ given as follows:
\[\Gamma =SL(6, \mathbb{Z}),\quad A=f(SL(2, \mathbb{Z})\otimes SL(3, \mathbb{Z}))f^{-1}.\]
\item[(b)] Put $\Gamma =SL(3, \mathbb{Z})$ and define $A$ to be the group consisting of all matrices in $SL(3, \mathbb{Z})$ both of whose $(2, 1)$- and $(3, 1)$-entries are zero. 
\end{enumerate}
Then the amalgamated free product $\Gamma \ast_A\Gamma$ is ME rigid.
\end{thm}

To prove Theorem \ref{thm-mer}, we propose a useful formulation of rigidity in the setting of ME, which automatically implies not only ME rigidity but also rigidity in terms of OE and lattice embeddings.
Let $\Gamma$ be a discrete group and $G$ a standard Borel group.
Let $\pi \colon \Gamma \rightarrow G$ be a homomorphism.
We introduce a rigidity property of $\Gamma$, called {\it coupling rigidity} with respect to the pair $(G, \pi)$ (or with respect to $G$ if $\pi$ is not specified).
This rigidity forces the existence of an essentially unique, $(\Gamma \times \Gamma)$-equivariant Borel map from any self-coupling of $\Gamma$ (i.e., a coupling of $\Gamma$ and $\Gamma$) into $G$, where the action of $\Gamma \times \Gamma$ on $G$ is defined by the formula
\[(\gamma_1, \gamma_2)g=\pi(\gamma_1)g\pi(\gamma_2)^{-1},\quad \gamma_1, \gamma_2\in \Gamma,\ g\in G.\]   
If $\Gamma$ is coupling rigid with respect to $G$, then any self-coupling of $\Gamma$ can be understood through the action of $\Gamma \times \Gamma$ on $G$, which is somewhat familiar.
Thanks to Furman's representation theorem (see Theorem \ref{thm-furman-rep}), the coupling rigidity of $\Gamma$ brings an important consequence.
In fact, if $\Lambda$ is a discrete group ME to $\Gamma$, then a useful representation of $\Lambda$ into $G$ is obtained.
Furman's theorem provides us with a general principle that investigation of self-couplings of $\Gamma$ is the first step to understand the structure of $\Lambda$.
Furman \cite{furman-mer} showed the celebrated rigidity theorem for higher rank lattices mentioned already, by combining this theorem with Zimmer's cocycle superrigidity theorem.
This principle was followed by Monod-Shalom \cite{ms} and by the author \cite{kida-mer}.
We also follow this principle.
It is shown that if $\Gamma$ is the group in Theorem \ref{thm-mer}, then $\Gamma$ is coupling rigid with respect to the (abstract) commensurator $\comm(\Gamma)$ of $\Gamma$.
As a consequence, we deduce ME and OE rigidity results for $\Gamma$.
We refer to Section \ref{subsec-mecr} for an explicit formulation of coupling rigidity.

\medskip

\noindent {\bf Organization of the paper.} In Sections \ref{sec-groupoid} and \ref{sec-me-oe}, we collect basic notation and terminology related to ME, which contain discrete measured groupoids, (W)OE and the associated cocycles. 
The relationship among them is also reviewed briefly.
In Section \ref{sec-red}, we start the study of self-couplings of the amalgamated free products stated in Theorem \ref{thm-lqn-str}.
The main result of this section says that any amalgamated free product $\Gamma$ in the theorem is coupling rigid with respect to the automorphism group of the Bass-Serre tree $T$.
Any discrete group $\Lambda$ that is ME to $\Gamma$ thus acts on $T$.
In Section \ref{sec-mec-arb}, we investigate this action and prove Theorem \ref{thm-lqn-str}.
We also present sufficient conditions for the action to be locally cofinite and to be cocompact.

In Section \ref{sec-oer}, we prove Theorem \ref{thm-oe-rigidity}.
In Section \ref{sec-mer}, we give a general criterion for amalgamated free products $\Gamma$ of rigid groups to be coupling rigid with respect to $\comm(\Gamma)$ and prove Theorem \ref{thm-mer}.
In Section \ref{sec-twist}, we present amalgamated free products $\Gamma$ which are not coupling rigid with respect to $\comm(\Gamma)$.
In Sections \ref{sec-ex} and \ref{sec-mis-ex}, we provide examples of groups to which general results in Sections \ref{sec-oer}--\ref{sec-twist} are applied.
Theorem \ref{thm-mer-ex} is proved in Section \ref{sec-ex}.
Based on argument involving algebraic groups, we find many examples satisfying the assumption in Theorem \ref{thm-oe-rigidity}.
In Section \ref{sec-add}, we present a few consequences of coupling rigidity with respect to abstract commensurators and observations on groups ME to free products.

\medskip

\noindent {\bf Notation employed throughout the paper.} When $H$ is a subgroup of a group $G$, we write $H<G$.
If $H$ is a normal subgroup of $G$, then we write $H\lhd G$.
These symbols are also used in a similar way when $G$ and $H$ are groupoids.
When $H$ is a subgroup of a group $G$, the normalizer and centralizer of $H$ in $G$ are denoted by ${\rm N}_G(H)$ and ${\rm Z}_G(H)$, respectively.
If $A$ and $B$ are subsets of a set, then $A\triangle B$ stands for the symmetric difference of $A$ and $B$, i.e., the set $(A\setminus B)\cup (B\setminus A)$.
The cardinality of a set $A$ is denoted by $|A|$.
We list below Assumptions writing down conditions imposed on amalgamated free products, and list symbols employed throughout the paper.

\medskip

\begin{center}
{\sc List of Assumptions and symbols}
\end{center}

\setlength{\columnsep}{3em}
\begin{multicols}{2}
\setlength{\parskip}{0.5ex}
  \faq{Assumption $(\star)$}{\pageref{ass-star}}
  \faq{Assumption $(\circ)$}{\pageref{ass-circ}}
  \faq{Assumption $(\bullet)$}{\pageref{ass-bullet}}
  \faq{Assumption $(\dagger)$}{\pageref{ass-dagger}}
  \faq{Assumption $(\ddagger)$}{\pageref{ass-ddagger}}
  \faq{$\aut(T)$}{\pageref{aut}}
  \faq{$\aut^{*}(T)$}{\pageref{aut-star}}
  \faq{$\comm(\Gamma)$}{\pageref{abstractcomm}}
  \faq{$\comm_G(\Gamma)$}{\pageref{comm}}
  \faq{$E(T)$}{\pageref{edge}}
  \faq{{\bf i}}{\pageref{i}}
  \faq{${\rm LQN}_{\Gamma}(A)$}{\pageref{lqn}}
  \faq{$V(T)$}{\pageref{vertex}}
  \faq{$V_1(T)$, $V_2(T)$}{\pageref{vit}}
  \faq{$(\cal{G})_A$}{\pageref{restriction}}
  \faq{$\cal{G}A$}{\pageref{saturation}}
\end{multicols}

\noindent {\bf Acknowledgements.} Most of this work was done while the author was affiliated with Mathematical Institute, Tohoku University.
He expresses his sincere gratitude to his colleagues and staffs there for their kind help and support in his studies.
The manuscript of this paper was written during the stay at Institut des Hautes \'Etudes Scientifiques.
The author thanks the institute for giving nice environment and for warm hospitality.

\section{Discrete measured groupoids}\label{sec-groupoid}

We collect basic facts on discrete measured groupoids.
We recommend the reader to consult \cite{kechris} and Chapter XIII, Section 3 in \cite{take3} for basic knowledge of standard Borel spaces and discrete measured groupoids, respectively.

We refer to a standard Borel space with a finite positive measure as a {\it standard finite measure space}.
When the measure is a probability one, we refer to it as a {\it standard probability space}.
Given a discrete measured groupoid $\cal{G}$ on a standard finite measure space $(X, \mu)$ and a Borel subset $A\subset X$ of positive measure, we denote by
\[(\cal{G})_A=\{\, g\in \cal{G}\mid r(g), s(g)\in A\, \}\label{restriction}\]
the groupoid restricted to $A$, where $r, s\colon \cal{G}\rightarrow X$ are the range and source maps, respectively.
If $A$ is a Borel subset of $X$, then $\cal{G}A$ stands for the saturation
\[\cal{G}A=\{\, r(g)\in X\mid g\in \cal{G}, s(g)\in A\, \}\label{saturation},\]
which is a Borel subset of $X$.
We say that a discrete measured groupoid $\cal{G}$ on a standard finite measure space $(X, \mu)$ is {\it finite} if for a.e.\ $x\in X$, $r^{-1}(x)$ consists of at most finitely many points.
We say that $\cal{G}$ is of {\it infinite type} if for a.e.\ $x\in X$, $r^{-1}(x)$ consists of infinitely many points.

Let $\Gamma \c X$ be a Borel action of a discrete group $\Gamma$ on a standard Borel space $X$.
We equip the product space $\Gamma \times X$ with the following structure of a groupoid:
\begin{itemize}
\item The range and source maps are defined by $r(\gamma, x)=\gamma x$ and $s(\gamma, x)=x$, respectively, for $\gamma \in \Gamma$ and $x\in X$. 
\item The operation of products is defined by $(\gamma_1, \gamma_2x)(\gamma_2, x)=(\gamma_1\gamma_2, x)$ for $\gamma_1, \gamma_2\in \Gamma$ and $x\in X$.  
\item $(e, x)$ is the unit element at $x\in X$. 
\item The inverse of $(\gamma, x)\in \Gamma \times X$ is defined by $(\gamma^{-1}, \gamma x)$.
\end{itemize}
Let $\cal{G}$ denote this groupoid.
We say that a $\sigma$-finite positive measure $\mu$ on $X$ is {\it quasi-invariant} for the action $\Gamma \c X$ if the action preserves the class of $\mu$.
In this case, the action $\Gamma \c (X, \mu)$ is said to be {\it non-singular}.
We define a $\sigma$-finite positive measure $\tilde{\mu}$ on $\Gamma \times X$ by the formula
\[\tilde{\mu}(A)=\int_X\sum_{g\in \cal{G}_x}\chi_A(g)\,d\mu (x)\]
for a Borel subset $A\subset \Gamma \times X$, where $\cal{G}_x=\{\, g\in \cal{G}\mid s(g)=x\, \}$ for $x\in X$ and $\chi_A$ is the characteristic function on $A$.
The class of the measure $\tilde{\mu}$ is invariant under the map $\cal{G}\ni g\mapsto g^{-1}\in \cal{G}$.
When $\Gamma \times X$ is equipped with this structure of a groupoid and the measure $\tilde{\mu}$, we call it the {\it discrete measured groupoid} associated with the action $\Gamma \c (X, \mu)$ and denote it by $\Gamma \ltimes (X, \mu)$\label{groupoidfromgroupaction} or by $\Gamma \ltimes X$ if $\mu$ is not specified.
Note that $\Gamma \ltimes (X, \mu)$ and its quotient equivalence relation on $(X, \mu)$ are naturally isomorphic as discrete measured groupoids if the action $\Gamma \c (X, \mu)$ is essentially free.

There is a close connection between orbit equivalence and isomorphism of two discrete measured groupoids associated with group actions.
Let $\Gamma \c (X, \mu)$ and $\Lambda \c (Y, \nu)$ be ergodic f.f.m.p.\ actions on standard finite measure spaces, and let $\cal{G}$ and $\cal{H}$ be the associated groupoids, respectively.
The two actions are OE if and only if $\cal{G}$ and $\cal{H}$ are isomorphic as discrete measured groupoids.
If there are Borel subsets $A\subset X$ and $B\subset Y$ of positive measure such that $(\cal{G})_A$ and $(\cal{H})_B$ are isomorphic, then the two actions are said to be {\it weakly orbit equivalent (WOE)}.

Given an action of a groupoid $\cal{G}$ on a space $S$, i.e., a groupoid homomorphism from $\cal{G}$ into $\aut(S)$, one can formulate fixed points of it.
For simplicity, we consider only actions of $\cal{G}$ which factor through an action of a discrete group.

\begin{defn}\label{defn-inv-map}
Let $\cal{G}$ be a discrete measured groupoid on a standard finite measure space $(X, \mu)$, and let $S$ be a standard Borel space.
Suppose that we have a Borel action of a discrete group $\Gamma$ on $S$ and a groupoid homomorphism $\rho \colon \cal{G}\rightarrow \Gamma$.
A Borel map $\varphi \colon X\rightarrow S$ is said to be {\it $(\cal{G}, \rho)$-invariant} if we have the equality
\[\rho(g)\varphi(s(g))=\varphi(r(g))\quad \textrm{for\ a.e.}\ g\in \cal{G}.\]
In this case, if $\rho$ is not specified, then $\varphi$ is said to be {\it $\cal{G}$-invariant}.

More generally, if $A$ is a Borel subset of $X$ of positive measure and if a Borel map $\varphi \colon A\rightarrow S$ satisfies the above equality for a.e.\ $g\in (\cal{G})_A$, then we say for simplicity that $\varphi$ is {\it $\cal{G}$-invariant} although we should say that $\varphi$ is $(\cal{G})_A$-invariant.  
\end{defn}

We give a context in which invariant Borel maps are found.

\begin{thm}[\cite{adams-spa}]\label{thm-adams-spa}
Let $\Gamma$ be a discrete group satisfying property (T).
Suppose that we have a measure-preserving action $\Gamma \c (X, \mu)$ of $\Gamma$ on a standard finite measure space, a simplicial tree $T$ having at most countably many simplices, and a Borel cocycle $\rho \colon \Gamma \times X\rightarrow H$, where $H$ is the simplicial automorphism group of $T$ equipped with the standard Borel structure induced by the pointwise convergence topology.
Let $S(T)$ be the set of simplices of $T$, on which $H$ naturally acts.
Then there exists a Borel map $\varphi \colon X\rightarrow S(T)$ satisfying the equality
\[\rho(\gamma, x)\varphi(x)=\varphi(\gamma x)\quad \forall \gamma \in \Gamma,\ \textrm{a.e.\ }x\in X.\]
\end{thm}

Finally, we present an observation on invariant Borel maps, which will repeatedly be used in Section \ref{sec-red}.

\begin{lem}\label{lem-inv-ext}
Let $\cal{G}$ be a discrete measured groupoid on a standard finite measure space $(X, \mu)$.
Suppose that we have a Borel action of a discrete group $\Gamma$ on a standard Borel space $S$ and a groupoid homomorphism $\rho \colon \cal{G}\rightarrow \Gamma$.
If $A\subset X$ is a Borel subset of positive measure and if $\varphi \colon A\rightarrow S$ is a $(\cal{G}, \rho)$-invariant Borel map, then $\varphi$ extends to a $(\cal{G}, \rho)$-invariant Borel map from $\cal{G}A$ into $S$. 
\end{lem}

\begin{proof}
Using 18.14 in \cite{kechris}, one can find a Borel map $f\colon \cal{G}A\rightarrow \cal{G}$ such that
\begin{itemize}
\item $s(f(x))=x$ and $r(f(x))\in A$ for a.e.\ $x\in \cal{G}A$; and
\item $f(x)=e_x\in \cal{G}_x^x$ for a.e.\ $x\in A$, 
\end{itemize}
where $\cal{G}_x^x=\{ \, g\in \cal{G} \mid r(g)=s(g)=x\, \}$ is the isotropy group on $x$ and $e_x$ is the unit element of $\cal{G}_x^x$.
The map $\psi \colon \cal{G}A\rightarrow S$ defined by $\psi(x)=\rho(f(x)^{-1})\varphi(r(f(x)))$ for $x\in \cal{G}A$ is then $(\cal{G}, \rho)$-invariant. 
\end{proof}


\section{ME, WOE, and coupling rigidity}\label{sec-me-oe}

In Section \ref{subsec-me-woe}, we review a connection between ME and WOE in use of discrete measured groupoids.
This is originally discussed in \cite{furman-oer} in use of discrete measured equivalence relations.
A formulation of it in terms of groupoids makes it possible to establish one-to-one correspondence between couplings and (conjugacy classes of) isomorphisms defined between groupoids associated with two group actions.
In Section \ref{subsec-mecr}, we introduce coupling rigidity of discrete groups, which clarifies a significant role of Furman's representation theorem in deducing diverse superrigidity results.
We discuss some consequences of coupling rigidity of discrete groups with respect to their abstract commensurators.

\subsection{Relationship between ME and WOE}\label{subsec-me-woe}

Let $\Gamma$ and $\Lambda$ be discrete groups.
Let $(\Sigma, m)$ be a coupling of $\Gamma$ and $\Lambda$.
Given Borel subsets $X, Y\subset \Sigma$ with $\bigsqcup_{\gamma \in \Gamma}\gamma Y=\bigsqcup_{\lambda \in \Lambda}\lambda X=\Sigma$ up to $m$-null sets, we define actions $\Gamma \c X$ and $\Lambda \c Y$ as follows:
For $\gamma \in \Gamma$ and $x\in X$, there exists a unique $\alpha(\gamma, x)\in \Lambda$ such that $(\gamma, \alpha(\gamma, x))x\in X$ since $X$ is a fundamental domain for the action $\Lambda \c \Sigma$.
The map $(\gamma, x)\mapsto (\gamma, \alpha(\gamma, x))x$ then defines an action of $\Gamma$ on $X$ which is measure-preserving with respect to the restriction of $m$ to $X$.
To distinguish this action and the original action of $\Gamma$ on $\Sigma$, we use a dot for this new action, that is, we denote $(\gamma, \alpha(\gamma, x))x$ by $\gamma \cdot x$.
The map $\alpha \colon \Gamma \times X\rightarrow \Lambda$ is called the {\it ME cocycle} (associated with $X$).
It satisfies the cocycle identity
\[\alpha(\gamma_1\gamma_2, x)=\alpha(\gamma_1, \gamma_2\cdot x)\alpha(\gamma_2, x),\quad \forall \gamma_1, \gamma_2\in \Gamma,\ {\rm a.e.} \ x\in X.\]
We define an action of $\Lambda$ on $Y$ in a similar way and denote the ME cocycle associated with $Y$ by $\beta \colon \Lambda \times Y\rightarrow \Gamma$. 
Set $\cal{G}=\Gamma \ltimes X$ and $\cal{H}=\Lambda \ltimes Y$.
We can choose fundamental domains $X$, $Y$ satisfying the property in Lemma \ref{lem-int-max} below, which makes it much easier to construct an isomorphism between restrictions of $\cal{G}$ and $\cal{H}$. 

\begin{lem}[\ci{Lemma 2.27}{kida-survey}]\label{lem-int-max}
In the above notation, one can choose $X$ and $Y$ so that the intersection $Z=X\cap Y$ satisfies the following two conditions:
\begin{itemize}
\item $\Gamma \cdot Z=X$ up to null sets when $Z$ is regarded as a subset of $X$.
\item $\Lambda \cdot Z=Y$ up to null sets when $Z$ is regarded as a subset of $Y$.
\end{itemize}
\end{lem}

We note that replacing fundamental domains for the actions $\Gamma \c \Sigma$ and $\Lambda \c \Sigma$ corresponds to exchanging the ME cocycles into ones cohomologous to them.
It follows that this replacing does not affect the problem considering the coupling $\Gamma \times \Lambda \c \Sigma$.
Let us define the groupoid homomorphisms $f\colon (\cal{G})_Z\rightarrow (\cal{H})_Z$ and $g\colon (\cal{H})_Z\rightarrow (\cal{G})_Z$ by the formulas
\[f(\gamma, x)=(\alpha(\gamma, x), x),\quad g(\lambda, y)=(\beta(\lambda, y), y)\]
for $(\gamma, x)\in (\cal{G})_Z$ and $(\lambda, y)\in (\cal{H})_Z$.

\begin{prop}[\ci{Proposition 2.29}{kida-survey}]\label{prop-group-iso}
In the above notation, we have $g\circ f={\rm id}$ and $f\circ g={\rm id}$.
\end{prop}

Conversely, if we have an isomorphism between restrictions of two groupoids associated with measure-preserving actions of groups on standard finite measure spaces, then we can construct the corresponding coupling (see Lemma 2.30 in \cite{kida-survey}).


\subsection{Coupling rigidity and Furman's representation theorem}\label{subsec-mecr}

Let $G$ be a standard Borel group and $\Gamma$, $\Lambda$ discrete groups.
Given homomorphisms $\pi \colon \Gamma \rightarrow G$ and $\rho \colon \Lambda \rightarrow G$, we denote by $(G, \pi, \rho)$ the Borel space $G$ equipped with the action of $\Gamma \times \Lambda$ defined by
\[(\gamma, \lambda)g=\pi(\gamma)g\rho(\lambda)^{-1},\quad g\in G,\ \gamma \in \Gamma,\ \lambda \in \Lambda.\]
Let $\Sigma$ be a coupling of $\Gamma$ and $\Lambda$, and let $\Phi \colon \Sigma \rightarrow S$ be a Borel map into a standard Borel space $S$ on which $\Gamma \times \Lambda$ acts.
We say that $\Phi$ is {\it almost $(\Gamma \times \Lambda)$-equivariant} if we have the equality
\[\Phi((\gamma, \lambda)x)=(\gamma, \lambda)\Phi(x),\quad \forall \gamma \in \Gamma,\ \forall \lambda \in \Lambda,\ \textrm{a.e.}\ x\in \Sigma.\]

\begin{defn}\label{defn-coup-rigid}
Let $\Gamma$ be a discrete group, $G$ a standard Borel group and $\pi \colon \Gamma \rightarrow G$ a homomorphism. We say that $\Gamma$ is {\it coupling rigid} with respect to the pair $(G, \pi)$ if
\begin{enumerate}
\item[(a)] for any self-coupling $\Sigma$ of $\Gamma$, there exists an almost $(\Gamma \times \Gamma)$-equivariant Borel map $\Phi \colon \Sigma \rightarrow (G, \pi, \pi)$; and
\item[(b)] the Dirac measure on the neutral element of $G$ is the only probability measure on $G$ that is invariant under conjugation by any element of $\pi(\Gamma)$.
\end{enumerate}
When $\pi$ is understood from the context, $\Gamma$ is simply said to be coupling rigid with respect to $G$.
\end{defn}

\begin{lem}\label{lem-uni-quo}
Let $\Gamma$ be a discrete group, $G$ a standard Borel group and $\pi \colon \Gamma \rightarrow G$ a homomorphism satisfying condition (b) in Definition \ref{defn-coup-rigid}.
Let $\Sigma$ be a self-coupling of $\Gamma$. Then the following assertions hold:
\begin{enumerate}
\item An almost $(\Gamma \times \Gamma)$-equivariant Borel map from $\Sigma$ into $(G, \pi, \pi)$ is essentially unique if it exists.
\item Let $\Gamma'$ and $\Gamma''$ be finite index subgroups of $\Gamma$. Then any almost $(\Gamma' \times \Gamma'')$-equivariant Borel map from $\Sigma$ into $(G, \pi, \pi)$ is almost $(\Gamma \times \Gamma)$-equivariant.
\end{enumerate}
\end{lem}

\begin{proof}
Let $\Phi$ and $\Psi$ be almost $(\Gamma \times \Gamma)$-equivariant Borel maps from $\Sigma$ into $(G, \pi, \pi)$.
The map $x\mapsto \Phi(x)\Psi(x)^{-1}$ is invariant under the action of $\{ e\} \times \Gamma$ and thus induces a finite positive measure on $G$ which is invariant under conjugation by any element of $\pi(\Gamma)$.
By condition (b), we have $\Phi(x)=\Psi(x)$ for a.e.\ $x\in \Sigma$. 
Assertion (i) is proved.
Assertion (ii) follows from Lemma 5.8 in \cite{kida-mer}.
\end{proof}

The following theorem is a major consequence of coupling rigidity and is applied to getting information on an unknown group $\Lambda$ ME to a given group $\Gamma$.

\begin{thm}\label{thm-furman-rep}
Let $\Gamma$ be a discrete group, $G$ a standard Borel group and $\pi \colon \Gamma \rightarrow G$ a homomorphism.
Suppose that $\Gamma$ is coupling rigid with respect to $(G, \pi)$.
Let $\Lambda$ be a discrete group and $\Sigma$ a coupling of $\Gamma$ and $\Lambda$.
Then there exist
\begin{itemize}
\item a homomorphism $\rho \colon \Lambda \rightarrow G$; and
\item an almost $(\Gamma \times \Lambda)$-equivariant Borel map $\Phi \colon \Sigma \rightarrow (G, \pi, \rho)$.
\end{itemize}
Moreover, if $\ker \pi$ is finite and there is a Borel fundamental domain for the action of $\pi(\Gamma)$ on $G$ defined by left multiplication, then $\rho$ can be chosen so that $\ker \rho$ is finite.
\end{thm}

\begin{proof}
Since this theorem is essentially proved by Furman \cite{furman-mer}, we give only a sketch of the proof.
We consider the actions of $\Gamma \times \Gamma$ and $\Lambda \times \Lambda$ on $\Sigma \times \Lambda \times \Sigma$ defined by the formulas
\begin{align*}
(\gamma_1, \gamma_2)(x, \lambda, y)&=(\gamma_1x, \lambda, \gamma_2y),\\
(\lambda_1, \lambda_2)(x, \lambda, y)&=(\lambda_1x, \lambda_1\lambda \lambda_2^{-1}, \lambda_2y),
\end{align*}
respectively, for $\gamma_1, \gamma_2\in \Gamma$, $\lambda, \lambda_1, \lambda_2\in \Lambda$ and $x, y\in \Sigma$.
We denote by $\Omega$ the quotient space of $\Sigma \times \Lambda \times \Sigma$ by the action of $\Lambda \times \Lambda$.
The action of $\Gamma \times \Gamma$ induces an action $\Gamma \times \Gamma \c \Omega$, which defines a self-coupling of $\Gamma$.
By assumption, there exists an almost $(\Gamma \times \Gamma)$-equivariant Borel map $\Psi \colon \Omega \rightarrow (G, \pi, \pi)$.
It will be shown that this $\Psi$ satisfies the equality
\[\Psi([x, \lambda, y])=\Phi(x)\rho(\lambda)\Phi(y)^{-1},\quad \forall \lambda \in \Lambda,\ \textrm{a.e.\ }(x, y)\in \Sigma \times \Sigma,\]
where $\rho$ and $\Phi$ are maps in the theorem, and $[x, \lambda, y]$ stands for the equivalence class of $(x, \lambda, y)\in \Sigma \times \Lambda \times \Sigma$ in $\Omega$.

The above expected equality inspires us to prove that the Borel map $F\colon \Sigma^3\rightarrow G$ defined by the $F(x, y, z)=\Psi([x, e, z])\Psi([y, e, z])^{-1}$ for $(x, y, z)\in \Sigma^3$ is independent of the third variable $z$.
Indeed, we can prove it by using condition (b) in Definition \ref{defn-coup-rigid}.
It follows that if we define two Borel maps $\rho \colon \Lambda \rightarrow G$, $\Phi \colon \Sigma \rightarrow G$ by setting
\begin{align*}
\rho(\lambda)&=\Psi([x_0, \lambda, z])\Psi([x_0, e, z])^{-1},\quad \lambda \in \Lambda\\
\Phi(x)&=\Psi([x_0, e, x])^{-1}, \quad x\in \Sigma
\end{align*}
with an appropriate $x_0\in \Sigma$, then $\rho$ is independent of $z$, and $\rho$ and $\Phi$ satisfy desired properties in the theorem.
The expected equality in the previous paragraph follows from uniqueness of an almost $(\Gamma \times \Gamma)$-equivariant Borel map from a self-coupling of $\Gamma$ into $(G, \pi, \pi)$.
The latter assertion of the theorem is verified along argument in the proof of Lemma 6.1 in \cite{furman-mer}. 
\end{proof}

Lemma \ref{lem-comm-rep} below gives a necessary condition for $\Gamma$ to be coupling rigid with respect to $(G, \pi)$.
This condition is formulated in terms of the commensurator of $\Gamma$ defined as follows.

\begin{defn}\label{defn-comm}
Let $\Gamma$ be a group.
We define $F(\Gamma)$ as the set of all isomorphisms between finite index subgroups of $\Gamma$.
Let us say that two elements $\phi$, $\psi$ of $F(\Gamma)$ are equivalent if there exists a finite index subgroup of $\Gamma$ on which $\phi$ and $\psi$ are equal. 
The composition of two elements $\phi \colon \Gamma_1\rightarrow \Gamma_2$, $\psi \colon \Lambda_1\rightarrow \Lambda_2$ of $F(\Gamma)$ given by $\phi \circ \psi \colon \psi^{-1}(\Gamma_1\cap \Lambda_2)\rightarrow \phi(\Lambda_2\cap \Gamma_1)$ induces the product operation on the quotient set of $F(\Gamma)$ by this equivalence relation. 
This makes it into the group called the {\it (abstract) commensurator} of $\Gamma$ and denoted by $\comm(\Gamma)$\label{abstractcomm}. 
We denote by
\[{\bf i}\colon \Gamma \rightarrow \comm(\Gamma)\label{i}\]
the homomorphism associating to each element of $\Gamma$ the conjugation by it.

Suppose that $\Gamma$ is a subgroup of a group $G$.
The {\it (relative) commensurator} of $\Gamma$ in $G$, denoted by $\comm_G(\Gamma)$\label{comm}, is defined as the subgroup of $G$ consisting of all elements $g\in G$ such that $\Gamma \cap g\Gamma g^{-1}$ is of finite index in both $\Gamma$ and $g\Gamma g^{-1}$.
\end{defn}

To write down basic properties of $\comm(\Gamma)$ and ${\bf i}$, let us introduce the following terminology.

\begin{defn}
For a group $G$ and a subgroup $\Gamma$ of $G$, we say that $G$ is {\it ICC with respect to} $\Gamma$ if for any non-neutral element $g$ of $G$, the set
\[\{\, \gamma g\gamma^{-1}\in G \mid \gamma \in \Gamma \,\}\]
consists of infinitely many elements.
If $G$ is ICC with respect to $G$ itself, then $G$ is simply said to be {\it ICC}.
\end{defn}

It is easy to check the following:

\begin{lem}\label{lem-comm}
Let $\Gamma$ be a discrete group.
Then the following assertions hold:
\begin{enumerate}
\item The homomorphism ${\bf i}\colon \Gamma \rightarrow \comm(\Gamma)$ is injective if and only if $\Gamma$ is ICC.
\item If $\Gamma$ is finitely generated, then $\comm(\Gamma)$ is countable.
\item If $\Gamma$ is ICC, then $\comm(\Gamma)$ is ICC with respect to ${\bf i}(\Gamma)$.
\end{enumerate}
\end{lem}

\begin{lem}\label{lem-comm-rep}
Let $\Gamma$ be a discrete group, $G$ a standard Borel group and $\pi \colon \Gamma \rightarrow G$ an injective homomorphism.
Suppose that $\Gamma$ is coupling rigid with respect to $(G, \pi)$.
Then there exists an isomorphism $\bar{\pi}\colon \comm(\Gamma)\rightarrow \comm_G(\pi(\Gamma))$ such that for any isomorphism $f\colon \Gamma'\rightarrow \Gamma''$ between finite index subgroups of $\Gamma$, we have
\[\pi(f(\gamma))=\bar{\pi}([f])\pi(\gamma)\bar{\pi}([f])^{-1},\quad \forall \gamma \in \Gamma',\]
where $[f]\in \comm(\Gamma)$ is the equivalence class of $f$.
\end{lem}

\begin{proof}
For each element $g$ of $\comm_G(\pi(\Gamma))$, the isomorphism $\pi^{-1}\circ \ad g\circ \pi$ between finite index subgroups of $\Gamma$ defines an element of $\comm(\Gamma)$.
The map $g\mapsto \pi^{-1}\circ \ad g\circ \pi$ then defines the homomorphism $i\colon \comm_G(\pi(\Gamma))\rightarrow \comm(\Gamma)$.

For each isomorphism $f\colon \Gamma'\rightarrow \Gamma''$ between finite index subgroups of $\Gamma$, we define a self-coupling $\Sigma_f$ of $\Gamma$ as follows:
Let $\Sigma^0_f$ be the countable Borel space $\Gamma''$ equipped with the counting measure. We define an action $\Gamma'\times \Gamma''\c \Sigma^0_f$ by the formula
\[(\gamma_1, \gamma_2)\gamma =f(\gamma_1)\gamma \gamma_2^{-1},\quad \gamma_1\in \Gamma',\ \gamma, \gamma_2\in \Gamma''.\]
Let $\Gamma \times \Gamma \c \Sigma_f$ be the self-coupling defined by the action induced from the action $\Gamma'\times \Gamma''\c \Sigma_f^0$.

Since $\Gamma$ is coupling rigid with respect to $(G, \pi)$, there exists an essentially unique, almost $(\Gamma \times \Gamma)$-equivariant Borel map $\Phi \colon \Sigma_f\rightarrow (G, \pi, \pi)$. Let $e\in \Sigma_f^0\subset \Sigma_f$ be the neutral element and put $g=\Phi(e)^{-1}$.
For any $\gamma \in \Gamma'$, the equality $\pi(\gamma)g^{-1}=\pi(\gamma)\Phi(e)=\Phi(f(\gamma))=\Phi(e)\pi(f(\gamma))=g^{-1}\pi(f(\gamma))$ holds.
We thus have the equality
\[\pi(f(\gamma))=g\pi(\gamma)g^{-1},\quad \forall \gamma \in \Gamma'.\]
The map $\bar{\pi}\colon \comm(\Gamma)\rightarrow \comm_G(\pi(\Gamma))$ defined by $[f]\mapsto g$ in the above notation is a homomorphism, whose inverse is equal to the homomorphism $i$ constructed in the first paragraph of the proof.
\end{proof}

This lemma implies that $\comm(\Gamma)$ is the smallest group with respect to which $\Gamma$ can be coupling rigid.
When $\Gamma$ is coupling rigid with respect to $\comm(\Gamma)$, several superrigidity results for $\Gamma$ are obtained through Theorem \ref{thm-furman-rep}.
These results are essentially proved in \cite{kida-oer} and \cite{kida-mer}.
Other consequences of coupling rigidity with respect to $\comm(\Gamma)$ are presented in Section \ref{sec-add}.

\begin{prop}\label{prop-conseq-rigid}
Let $\Gamma$ be an ICC discrete group with $\comm(\Gamma)$ countable.
If $\Gamma$ is coupling rigid with respect to $\comm(\Gamma)$, then the following assertions hold:
\begin{enumerate}
\item The group $\Gamma$ is ME rigid.
\item Any ergodic f.f.m.p.\ action of $\Gamma$ is superrigid, i.e., if an ergodic f.f.m.p.\ action $\Gamma \c (X, \mu)$ is WOE to an ergodic f.f.m.p.\ action $\Lambda \c (Y, \nu)$ of a discrete group $\Lambda$, then the two actions are virtually conjugate.
\item If two aperiodic f.f.m.p.\ actions $\Gamma \c (X, \mu)$, $\Gamma \c (Y, \nu)$ are WOE, then they are conjugate.
\end{enumerate}
\end{prop}

The following lemma gives a sufficient condition for $\Gamma$ to be coupling rigid with respect to $\comm(\Gamma)$.

\begin{lem}\label{lem-red-comm}
Let $\Gamma$ be an ICC discrete group with $\comm(\Gamma)$ countable.
Let $(\Sigma, m)$ be a self-coupling of $\Gamma$.
Suppose that we have
\begin{itemize}
\item a countable Borel space $C$ on which $\Gamma \times \Gamma$ acts so that each of the subgroups $\Gamma \times \{ e\}$ and $\{ e\} \times \Gamma$ acts freely; and
\item an almost $(\Gamma \times \Gamma)$-equivariant Borel map $\Phi_0\colon \Sigma \rightarrow C$.
\end{itemize}
Then there exists an essentially unique, almost $(\Gamma \times \Gamma)$-equivariant Borel map from $\Sigma$ into $(\comm(\Gamma), {\bf i}, {\bf i})$.\end{lem}

\begin{proof}
Choose $(\Sigma, m)$, $C$ and $\Phi_0\colon \Sigma \rightarrow C$ in the assumption.
We equip the measure $(\Phi_0)_*m$ with $C$.
We may assume that each point of $C$ has positive measure.
The space $C$ is then a self-coupling of $\Gamma$.
To distinguish the two actions of $\Gamma$ on $\Sigma$, let us use the following notation.
For $\gamma \in \Gamma$, we set
\[L(\gamma)=(\gamma, e)\in \Gamma \times \Gamma,\quad R(\gamma)=(e, \gamma)\in \Gamma \times \Gamma,\] 
where $e$ is the neutral element of $\Gamma$.

We construct a $(\Gamma \times \Gamma)$-equivariant map $\Psi \colon C\rightarrow (\comm(\Gamma), {\bf i}, {\bf i})$.
Pick $c\in C$ and choose a fundamental domain $X$ for the action $R(\Gamma)\c C$ containing $c$.
Let $\alpha \colon \Gamma \times X\rightarrow \Gamma$ be the associated ME cocycle.
We denote by $\Gamma_c$ the stabilizer of $c$ for the action $\Gamma \c X$.
It follows that $\Gamma_c$ is a finite index subgroup of $\Gamma$ and that the restriction of $\alpha$ to $\Gamma_c\times \{ c\}$, denoted by $\alpha_c$, is a homomorphism from $\Gamma_c$ into $\Gamma$, which is independent of the choice of $X$ by the definition of $\alpha$.
Similarly, choose a fundamental domain $Y$ for the action $L(\Gamma)\c C$ containing $c$.
Let $\beta \colon \Gamma \times Y\rightarrow \Gamma$ be the ME cocycle and $\Gamma_c'$ the stabilizer of $c$ for the action $\Gamma \c Y$.
The restriction of $\beta$ to $\Gamma_c'\times \{ c \}$, denoted by $\beta_c$, is a homomorphism from $\Gamma_c'$ into $\Gamma$, which is independent of the choice of $Y$.
Both of the compositions $\alpha_c\circ \beta_c$ and $\beta_c\circ \alpha_c$ are equal to the inclusions.
It follows that $\alpha_c$ and $\beta_c$ define elements of $\comm(\Gamma)$ and that $\beta_c$ is the inverse of $\alpha_c$ in $\comm(\Gamma)$.

We define a map $\Psi \colon C\rightarrow \comm(\Gamma)$ by setting $\Psi(c)=(\alpha_c)^{-1}=\beta_c$ for each $c\in C$.
We prove that $\Psi$ is $(\Gamma \times \Gamma)$-equivariant.
Pick $\gamma \in \Gamma$ and $c\in C$.
Let $X, Y\subset C$ be fundamental domains for the actions $R(\Gamma)\c C$, $L(\Gamma)\c C$ containing $c$, respectively.
Let $\alpha \colon \Gamma \times X\rightarrow \Gamma$ and $\beta \colon \Gamma \times Y\rightarrow \Gamma$ be the associated ME cocycles.
Since $L(\gamma)Y$ is also a fundamental domain for the action $L(\Gamma)\c C$, we have the associated ME cocycle $\beta^{\gamma}\colon \Gamma \times L(\gamma)Y\rightarrow \Gamma$.
This cocycle $\beta^{\gamma}$ satisfies the equality
\[\beta^{\gamma}(\lambda, L(\gamma)y)=\gamma \beta(\lambda, y)\gamma^{-1},\quad \forall \lambda \in \Gamma,\ \textrm{a.e.\ }y\in Y\]  
because we have $L(\gamma \beta(\lambda, y)\gamma^{-1})R(\lambda)L(\gamma)y=L(\gamma)L(\beta(\lambda, y))R(\lambda)y$, which belongs to $L(\gamma)Y$.
The equality implies that $\Psi(L(\gamma)c)$ is equal to ${\bf i}(\gamma)\Psi(c)$.
The map $\Psi$ is hence equivariant under the action of $L(\Gamma)$.
The equivariance under the action of $R(\Gamma)$ can be proved in a similar way. 
Composing the map $\Phi_0$ with $\Psi$, we obtain an almost $(\Gamma \times \Gamma)$-equivariant Borel map $\Phi \colon \Sigma \rightarrow (\comm(\Gamma), {\bf i}, {\bf i})$.
Essential uniqueness of $\Phi$ follows from Lemma \ref{lem-uni-quo} (i) and Lemma \ref{lem-comm} (iii).
\end{proof}

Finally, we present examples of coupling rigidity.
The following theorem is due to Furman \cite{furman-mer}.
The proof relies on Zimmer's cocycle superrigidity theorem \cite{zim-book}.

\begin{thm}[\cite{furman-mer}]\label{thm-me-hrl}
Let $G$ be a non-compact connected simple Lie group with its center finite and its real rank at least two.
Let $\Gamma$ be a lattice in $G$.
We define $\imath \colon \Gamma \rightarrow \aut(\ad G)$ as the restriction of the natural homomorphism from $G$ into $\aut(\ad G)$.
Then $\Gamma$ is coupling rigid with respect to $(\aut(\ad G), \imath)$.
\end{thm}

It is known that the mapping class group of a non-exceptional compact orientable surface $S$ is coupling rigid with respect to the automorphism group of the complex of curves for $S$ (see \cite{kida-mer}).

\section{Reduction of self-couplings}\label{sec-red}

Let $\Gamma =\Gamma_1\ast_A\Gamma_2$ be an amalgamated free product of discrete groups, and let $T$ be the Bass-Serre tree associated with the decomposition of $\Gamma$.
In this section, we prove that $\Gamma$ is coupling rigid with respect to the simplicial automorphism group $\aut^*(T)$ of $T$ when several conditions are imposed on the subgroups $\Gamma_1$, $\Gamma_2$ and $A$.
This coupling rigidity gives us useful information on a group ME to $\Gamma$ thanks to Theorem \ref{thm-furman-rep}. 
A detailed study of such a group will be performed in Section \ref{sec-mec-arb}. 
The following collects conditions on the factor subgroups $\Gamma_1$, $\Gamma_2$ and the inclusions $A<\Gamma_1$ and $A<\Gamma_2$ so that coupling rigidity of $\Gamma$ holds. 

\medskip

\noindent {\bf Assumption $(\star)$}\label{ass-star}: For each $i=1, 2$, let $\Gamma_i$ be a discrete group and $A_i$ a subgroup of $\Gamma_i$ such that
\begin{itemize}
\item $\Gamma_i$ satisfies property (T); and
\item $|A_i|=\infty$, $[\Gamma_i: A_i]=\infty$ and ${\rm LQN}_{\Gamma_i}(A_i)=A_i$.
\end{itemize}
Let $\phi \colon A_1\rightarrow A_2$ be an isomorphism and set the amalgamated free product
\[\Gamma =\langle\, \Gamma_1, \Gamma_2\mid A_1\simeq_{\phi}A_2\,\rangle.\]
We denote by $A$ the subgroup of $\Gamma$ corresponding to $A_1\simeq_{\phi}A_2$.
Let $T$ denote the Bass-Serre tree associated with the decomposition of $\Gamma$ and $\imath \colon \Gamma \rightarrow \aut^*(T)$ the homomorphism arising from the action $\Gamma \c T$.

\medskip

Let us explain the notation and terminology in the above assumption.
For a group $\Gamma$ and a subgroup $A$ of $\Gamma$, we put
\[{\rm LQN}_{\Gamma}(A)=\{\, \gamma \in \Gamma \mid [A:\gamma A\gamma^{-1}\cap A]<\infty \, \}\label{lqn}\]
and call it the {\it left quasi-normalizer} of $A$ in $\Gamma$, which is a subsemigroup of $\Gamma$ and contains $A$.
We say that $A$ is {\it left-quasi-normalized by itself} in $\Gamma$ if the equality ${\rm LQN}_{\Gamma}(A)=A$ holds.
We denote by $\aut^*(T)$\label{aut-star} the simplicial automorphism group of $T$ equipped with the standard Borel structure associated with the pointwise convergence topology.

Let $\Gamma =\Gamma_1\ast_A\Gamma_2$ be an amalgamated free product.
We review basic properties of the Bass-Serre tree associated with $\Gamma$.
We define a simplicial graph $T$ as follows.
Let $V(T)=\Gamma/\Gamma_1\sqcup \Gamma/\Gamma_2$\label{vertex} be the set of vertices of $T$, and let $E(T)=\Gamma/A$\label{edge} be the set of edges of $T$.
For each $\gamma \in \Gamma$, the two end points of the edge $\gamma A\in E(T)$ are given by $\gamma \Gamma_1, \gamma \Gamma_2\in V(T)$.
It follows from I.4.1 in \cite{serre} that $T$ is in fact a connected tree, on which $\Gamma$ acts by left multiplication.

For each $v\in V(T)$, let ${\rm Lk}(v)$\label{link} denote the set of vertices in the link of $v$ in $T$.
Let $v_i\in V(T)$ be the vertex corresponding to $\Gamma_i\in \Gamma /\Gamma_i$ for $i=1, 2$.
The stabilizer of $v_i$ in $\Gamma$ is equal to $\Gamma_i$.
There is a $\Gamma_i$-equivariant one-to-one correspondence between elements of $\Gamma_i/A$ and elements of ${\rm Lk}(v_i)$.
In particular, if $A$ is of infinite index in $\Gamma_i$, then ${\rm Lk}(v_i)$ consists of infinitely many elements.
The equality ${\rm LQN}_{\Gamma_1}(A)=A$ holds if and only if any orbit for the action $A\c {\rm Lk}(v_1)\setminus \{ v_2\}$ consists of infinitely many elements.
The same thing holds for the equality ${\rm LQN}_{\Gamma_2}(A)=A$ and the action $A\c {\rm Lk}(v_2)\setminus \{ v_1\}$.

We introduce an orientation on $T$ as follows:
For each $\gamma \in \Gamma$, let $\gamma \Gamma_1, \gamma \Gamma_2\in V(T)$ be the origin and terminal of the edge $\gamma A\in E(T)$, respectively.
Let $\aut(T)$\label{aut} be the group of simplicial automorphisms of $T$ preserving this orientation.
The group $\aut(T)$ is then a subgroup of $\aut^*(T)$ with its index at most two, and it consists of automorphisms of $T$ without inversions.

The following lemma proves one of the requirement for coupling rigidity of $\Gamma$ with respect to $\aut^*(T)$.

\begin{lem}\label{lem-conj-inv}
On Assumption $(\star)$, the Dirac measure on the neutral element is the only probability measure on $\aut^*(T)$ that is invariant under conjugation by any element of $\imath(\Gamma)$.
\end{lem}

\begin{proof}
Let $\mu$ be a probability measure on $\aut^*(T)$ invariant under conjugation by any element of $\imath(\Gamma)$.
For $u, v\in V(T)$, we denote by $\aut(u, v)$ the Borel subset of $\aut^*(T)$ consisting of all $f\in \aut^*(T)$ with $f(u)=v$.
For $u\in V(T)$, we denote by $\Gamma_u$ the stabilizer of $u$ in $\Gamma$.
For any $u, v\in V(T)$ and $\gamma \in \Gamma_u$, we have the equalities
\[\aut^*(T)=\bigsqcup_{w\in V(T)}\aut(u, w),\quad \imath(\gamma)\aut(u, v)\imath(\gamma)^{-1}=\aut(u, \gamma v).\]
For any $u\in V(T)$, any orbit for the action $\Gamma_u\c V(T)$ other than $\{ u\}$ consists of infinitely many vertices because $A$ is a subgroup of infinite index in both $\Gamma_1$ and $\Gamma_2$.
For any $u\in V(T)$, the measure $\mu$ is thus supported on $\aut(u, u)$.
\end{proof}

\begin{thm}\label{thm-coup-tree}
On Assumption $(\star)$, the group $\Gamma$ is coupling rigid with respect to $(\aut^*(T), \imath)$.
\end{thm}

By Theorem \ref{thm-furman-rep}, we obtain the following corollary.
This is the first step to study groups ME to $\Gamma$.

\begin{cor}\label{cor-coup-tree}
On Assumption $(\star)$, let $\Lambda$ be a discrete group and $\Sigma$ a coupling of $\Gamma$ and $\Lambda$.
Then there exist a homomorphism $\rho \colon \Lambda \rightarrow \aut^*(T)$ and an almost $(\Gamma \times \Lambda)$-equivariant Borel map $\Phi \colon \Sigma \rightarrow (\aut^*(T), \imath, \rho)$.
\end{cor}

Theorem \ref{thm-coup-tree} is a direct consequence of Theorem \ref{thm-general-coup-tree} below.
In the rest of this section, we fix the following notation:
Let $\Gamma =\langle\, \Gamma_1, \Gamma_2\mid A_1\simeq_{\phi}A_2\,\rangle$ be the group in Assumption $(\star)$ and denote by $A$ the subgroup of $\Gamma$ corresponding to $A_1\simeq_{\phi}A_2$.
Let $\Lambda =\langle\, \Lambda_1, \Lambda_2\mid B_1\simeq_{\chi}B_2\,\rangle$ be another amalgamated free product defined by two pairs of groups $B_1<\Lambda_1$ and $B_2<\Lambda_2$ and an isomorphism $\chi \colon B_1\rightarrow B_2$ such that they satisfy the two conditions in Assumption $(\star)$ when $\Gamma_i$ and $A_i$ are replaced with $\Lambda_i$ and $B_i$, respectively, for each $i=1, 2$.
We denote by $B$ the subgroup of $\Lambda$ corresponding to $B_1\simeq_{\chi}B_2$. Let $T_{\Gamma}$ and $T_{\Lambda}$ be the Bass-Serre trees associated with the decompositions of $\Gamma$ and $\Lambda$, respectively. 
There are natural homomorphisms $\imath \colon \Gamma \rightarrow \aut^*(T_{\Gamma})$ and $\jmath \colon \Lambda \rightarrow \aut^*(T_{\Lambda})$.
Let $\isom(T_{\Lambda}, T_{\Gamma})$ (resp.\ $\isom(T_{\Gamma}, T_{\Lambda})$) denote the set of simplicial isomorphisms from $T_{\Lambda}$ onto $T_{\Gamma}$ (resp.\ from $T_{\Gamma}$ onto $T_{\Lambda}$) equipped with the standard Borel structure associated with the pointwise convergence topology.

\begin{thm}\label{thm-general-coup-tree}
In the above notation, let $\Sigma$ be a coupling of $\Gamma$ and $\Lambda$.
Then there exists an essentially unique Borel map $\Phi \colon \Sigma \rightarrow \isom(T_{\Lambda}, T_{\Gamma})$ satisfying the equality
\[\Phi((\gamma, \lambda)x)=\imath(\gamma)\Phi(x)\jmath(\lambda)^{-1},\quad \forall \gamma \in \Gamma,\ \forall \lambda \in \Lambda,\ \textrm{a.e.\ }x\in \Sigma.\]
\end{thm}

\begin{proof}
Let us fix the notation.
Choose fundamental domains $X\subset \Sigma$ for the action $\Lambda \c \Sigma$ and $Y\subset \Sigma$ for the action $\Gamma \c \Sigma$ so that the intersection $Z=X\cap Y$ satisfies the two equalities $\Gamma \cdot Z=X$ and $\Lambda \cdot Z=Y$ up to null sets (see Lemma \ref{lem-int-max}).
Let $\alpha \colon \Gamma \times X\rightarrow \Lambda$ and $\beta \colon \Lambda \times Y\rightarrow \Gamma$ be the associated ME cocycles.
We put $\cal{G}=\Gamma \ltimes X$ and $\cal{H}=\Lambda \ltimes Y$.
The identity map on $Z$ defines the groupoid isomorphism $f\colon (\cal{G})_Z\rightarrow (\cal{H})_Z$ by Proposition \ref{prop-group-iso}.
For each $s\in V(T_{\Gamma})\cup E(T_{\Gamma})$, let $\Gamma_s$ denote the stabilizer of $s$ in $\Gamma$ and put $\cal{G}_s=\Gamma_s\ltimes X$.
Define $\Lambda_t$ and $\cal{H}_t$ for each $t\in V(T_{\Lambda})\cup E(T_{\Lambda})$ in a similar way.
The following proof is in part based on the author's argument in \cite{kida-mer}.
The main task is to investigate what kinds of subgroupoids of $\cal{G}$ and $\cal{H}$ are sent to each other by $f$.
We first show that the theorem is deduced from the following three lemmas.

\begin{lem}\label{lem-pre-stab}
For each $v\in V(T_{\Gamma})$, there exist a countable Borel partition $Z=\bigsqcup_nZ_n$ and $u_n\in V(T_{\Lambda})$ with
\[f((\cal{G}_v)_{Z_n})=(\cal{H}_{u_n})_{Z_n},\quad \forall n.\]
\end{lem}

\begin{lem}\label{lem-cons-varphi} 
We define a Borel map $\varphi \colon Z\times V(T_{\Gamma})\rightarrow V(T_{\Lambda})$ by putting $\varphi(x, v)=u_n$ if $x\in Z_n$ in the notation of Lemma \ref{lem-pre-stab}.
Then for a.e.\ $x\in Z$, the map $\varphi(x, \cdot)\colon V(T_{\Gamma})\rightarrow V(T_{\Lambda})$ defines a simplicial isomorphism from $T_{\Gamma}$ onto $T_{\Lambda}$.
We denote the isomorphism $\varphi(x,\cdot)$ by $\varphi(x)$. 
\end{lem}

\begin{lem}\label{lem-eq-varphi}
The map $\varphi \colon Z\rightarrow \isom(T_{\Gamma}, T_{\Lambda})$ constructed in Lemma \ref{lem-cons-varphi} satisfies the equality
\[\varphi(\gamma \cdot x)\imath(\gamma)\varphi(x)^{-1}=\jmath \circ \alpha(\gamma, x)\]
for any $\gamma \in \Gamma$ and a.e.\ $x\in Z$ with $\gamma \cdot x\in Z$.
\end{lem}

Assuming these three lemmas, we prove the theorem.
Note that $\Sigma$ is identified with the space $X\times \Lambda$ as a measure space on which $\Gamma \times \Lambda$ acts, where the action $\Gamma \times \Lambda \c X\times \Lambda$ is defined by the formula
\[(\gamma, \lambda)(x, \lambda')=(\gamma \cdot x, \alpha(\gamma, x)\lambda'\lambda^{-1}),\quad \gamma \in \Gamma,\ \lambda, \lambda'\in \Lambda,\ x\in X.\]
We define a Borel map $\Phi$ from the subset $Z\times \{ e\}$ of $\Sigma =X\times \Lambda$ into $\isom(T_{\Lambda}, T_{\Gamma})$, which will be extended to a map from the whole space $\Sigma$ so that it is almost $(\Gamma \times \Lambda)$-equivariant.
For each $x\in Z$, we put $\Phi(x, e)=\varphi(x)^{-1}$.
Choose $\gamma_1, \gamma_2\in \Gamma$, $\lambda_1, \lambda_2\in \Lambda$ and $x_1, x_2\in Z$ ($\subset X$) such that we have $\gamma_1\cdot x_1, \gamma_2\cdot x_2\in Z$ and the equality
\[(\gamma_1, \lambda_1)(x_1, e)=(\gamma_2, \lambda_2)(x_2, e)\]
holds in $\Sigma$.
To extend $\Phi$ to a map defined on the whole space $\Sigma$, it is enough to check that the equality $\imath(\gamma_1)\varphi(x_1)^{-1}\jmath(\lambda_1)^{-1}=\imath(\gamma_2)\varphi(x_2)^{-1}\jmath(\lambda_2)^{-1}$ holds.
The above equality is equivalent to the equality
\[(\gamma_1\cdot x_1, \alpha(\gamma_1, x_1)\lambda_1^{-1})=(\gamma_2\cdot x_2, \alpha(\gamma_2, x_2)\lambda_2^{-1}).\]
We have
\begin{align*}
\imath(\gamma_1)\varphi(x_1)^{-1}\jmath(\lambda_1)^{-1}&=\imath(\gamma_1)\imath(\gamma_1)^{-1}\varphi(\gamma_1\cdot x_1)^{-1}\jmath \circ \alpha(\gamma_1, x_1)\jmath(\lambda_1)^{-1}\\
&=\varphi(\gamma_2\cdot x_2)^{-1}\jmath \circ \alpha(\gamma_2, x_2)\jmath(\lambda_2)^{-1}\\
&=\imath(\gamma_2)\imath(\gamma_2)^{-1}\varphi(\gamma_2\cdot x_2)^{-1}\jmath \circ \alpha(\gamma_2, x_2)\jmath(\lambda_2)^{-1}\\
&=\imath(\gamma_2)\varphi(x_2)^{-1}\jmath(\lambda_2)^{-1},
\end{align*}  
where the first and fourth equalities hold by Lemma \ref{lem-eq-varphi}.
We thus obtain an almost $(\Gamma \times \Lambda)$-equivariant Borel map $\Phi \colon \Sigma \rightarrow \isom(T_{\Lambda}, T_{\Gamma})$.

If $\Phi_1$ and $\Phi_2$ are almost $(\Gamma \times \Lambda)$-equivariant Borel maps from $\Sigma$ into $\isom(T_{\Lambda}, T_{\Gamma})$, then the map $\Phi_0\colon \Sigma \rightarrow \aut^*(T_{\Gamma})$ defined by the formula $\Phi_0(x)=\Phi_1(x)\Phi_2(x)^{-1}$ for $x\in \Sigma$ satisfies the equality
\[\Phi_0(\gamma x)=\imath(\gamma)\Phi_0(x)\imath(\gamma)^{-1},\quad \Phi_0(\lambda x)=\Phi_0(x)\]
for any $\gamma \in \Gamma$, $\lambda \in \Lambda$ and a.e.\ $x\in \Sigma$. 
The map $\Phi_0$ thus induces an almost $\Gamma$-equivariant Borel map from $\Sigma /\Lambda$ into $\aut^*(T_{\Gamma})$, where the action of $\Gamma$ on $\aut^*(T_{\Gamma})$ is defined by conjugation through $\imath$.
By Lemma \ref{lem-conj-inv}, $\Phi_0$ is the essentially constant map whose value is the neutral element.
It turns out that $\Phi_1$ and $\Phi_2$ are essentially equal.
\end{proof}

\begin{proof}[Proof of Lemma \ref{lem-pre-stab}]
We first prove the following:

\begin{claim}\label{claim-dis-ver}
If $v, v'\in V(T_{\Gamma})$ are distinct vertices, then for any Borel subset $W\subset Z$ of positive measure, the inclusion $(\cal{G}_v)_W\subset (\cal{G}_{v'})_W$ cannot occur.
\end{claim}

\begin{proof}
Assume that the inclusion occurs. We may assume that $v$ and $v'$ are adjacent in $T_{\Gamma}$.
We denote by $e\in E(T)$ the edge connecting $v$ and $v'$, and denote by $\Gamma_e$ the stabilizer of $e$ in $\Gamma$.
Let $\rho \colon \cal{G}\rightarrow \Gamma$ be the cocycle defined by $\rho(\gamma, x)=\gamma$ for $\gamma \in \Gamma$ and $x\in X$.
The inclusion $(\cal{G}_v)_W\subset (\cal{G}_{v'})_W$ implies that we have $\rho(\gamma, x)\in \Gamma_e$ for a.e.\ $(\gamma, x)\in (\cal{G}_v)_W$.
By Lemma \ref{lem-inv-ext}, there exists a $(\cal{G}_v, \rho)$-invariant Borel map $\psi \colon \Gamma_v\cdot W\rightarrow {\rm Lk}(v)$ with $\psi(x)=v'$ for a.e.\ $x\in W$, where ${\rm Lk}(v)$ is the set of vertices in the link of $v$ in $T_{\Gamma}$.
Note that the action $\Gamma_v\c {\rm Lk}(v)$ is transitive and that ${\rm Lk}(v)$ consists of infinitely many elements.
On the other hand, there is a finite invariant measure on $\Gamma_v\cdot W$. This is a contradiction.
\end{proof}

Pick $v\in V(T_{\Gamma})$.
By Theorem \ref{thm-adams-spa}, there exists a $(\cal{G}_v, \alpha)$-invariant Borel map from $X$ into $V(T_{\Lambda})\cup E(T_{\Lambda})$ because $\Gamma_v$ is a conjugate of either $\Gamma_1$ or $\Gamma_2$ and thus satisfies property (T).
Since $\Lambda$ acts on $T_{\Lambda}$ without inversions, we can construct a $(\cal{G}_v, \alpha)$-invariant Borel map $\varphi_v\colon X\rightarrow V(T_{\Lambda})$.

For each $v\in V(T_{\Gamma})$, let $Z=\bigsqcup_nZ_n$ be a countable Borel partition such that $\varphi_v$ is essentially constant on $Z_n$ for each $n$, and let $u_n\in V(T_{\Lambda})$ be the value of $\varphi_v$ on $Z_n$.
The $(\cal{G}_v, \alpha)$-invariance of $\varphi_v$ implies that the inclusion $f((\cal{G}_v)_{Z_n})\subset (\cal{H}_{u_n})_{Z_n}$ holds for each $n$.
Applying the same argument for $f^{-1}$ in place of $f$ and applying Claim \ref{claim-dis-ver}, we obtain the equality $f((\cal{G}_v)_{Z_n})= (\cal{H}_{u_n})_{Z_n}$ for each $n$ after taking a finer countable Borel partition $Z=\bigsqcup_nZ_n$.
\end{proof}

\begin{proof}[\it Proof of Lemma \ref{lem-cons-varphi}]
For a.e.\ $x\in Z$, the map $\varphi(x, \cdot)$ defines a bijection from $V(T_{\Gamma})$ onto $V(T_{\Lambda})$.
To prove that $\varphi(x,\cdot)$ defines a simplicial isomorphism from $T_{\Gamma}$ onto $T_{\Lambda}$, we first show the following:

\begin{claim}\label{claim-two-ver}
Let $v_1, v_2\in V(T_{\Gamma})$ be distinct and adjacent vertices and let $e\in E(T)$ be the edge connecting them. Let $\rho \colon \cal{G}\rightarrow \Gamma$ be the cocycle defined by $\rho(\gamma, x)=\gamma$ for $\gamma \in \Gamma$ and $x\in X$.
Then for any Borel subset $W\subset Z$ of positive measure, any $(\cal{G}_e, \rho)$-invariant Borel map $\psi \colon W\rightarrow V(T_{\Gamma})$ is valued in the set $\{ v_1, v_2\}$.
\end{claim}

\begin{proof}
If there were a Borel subset $W'\subset W$ of positive measure on which $\psi$ takes the value $v\in V(T_{\Gamma})$ distinct from $v_1$ and $v_2$, then there exists a $(\cal{G}_e, \rho)$-invariant Borel map $\psi'\colon \Gamma_e\cdot W'\rightarrow \Gamma_ev$.
The equality ${\rm LQN}_{\Gamma_i}(A)=A$ for each $i=1, 2$ implies that $\Gamma_ev$ consists of infinitely many elements.
This is a contradiction.
\end{proof}

Let $v_1, v_2\in V(T_{\Gamma})$ be distinct and adjacent vertices, and let $e\in E(T)$ be the edge connecting them.
Pick a Borel subset $W\subset Z$ of positive measure and $u_1, u_2\in V(T_{\Lambda})$ with $\varphi(x, v_1)=u_1$ and $\varphi(x, v_2)=u_2$ for a.e.\ $x\in W$.
Assuming that $u_1$ and $u_2$ are not adjacent, we deduce a contradiction.
We may assume that the equalities $f((\cal{G}_{v_1})_W)=(\cal{G}_{u_1})_W$ and $f((\cal{G}_{v_2})_W)=(\cal{G}_{u_2})_W$ hold.
Choose a vertex $u\in V(T_{\Lambda})$ in the geodesic between $u_1$ and $u_2$ distinct from $u_1$ and $u_2$.
Taking a Borel subset of positive measure smaller than $W$, we may assume that the map $W\ni x\mapsto \varphi(x, \cdot)^{-1}(u)\in V(T_{\Gamma})$ is constant on $W$ with its value $v\in V(T_{\Gamma})$ and may assume $f((\cal{G}_v)_W)=(\cal{H}_u)_W$. The inclusion
\[f((\cal{G}_e)_W)=f((\cal{G}_{v_1})_W\cap (\cal{G}_{v_2})_W)=(\cal{H}_{u_1})_W\cap (\cal{H}_{u_2})_W\subset (\cal{H}_u)_W\]
then holds.
We thus have $(\cal{G}_e)_W\subset (\cal{G}_v)_W$.
This contradicts Claim \ref{claim-two-ver}.
\end{proof}

\begin{proof}[Proof of Lemma \ref{lem-eq-varphi}]
Pick $v\in V(T_{\Gamma})$.
It is enough to show the equality
\[(\varphi(\gamma \cdot x)\imath(\gamma))(v)=(\jmath \circ \alpha(\gamma, x)\varphi(x))(v),\quad \textrm{a.e.\ }x\in W\]
for any $\gamma \in \Gamma$ and any Borel subset $W\subset Z$ of positive measure such that
\begin{itemize}
\item $\gamma \cdot W\subset Z$;
\item the map $x\mapsto (\varphi(x))(v)\in V(T_{\Lambda})$ is constant on $W$ with its value $u_1$;
\item the map $x\mapsto (\varphi(\gamma \cdot x)\imath(\gamma))(v)\in V(T_{\Lambda})$ is constant on $W$ with its value $u_2$;
\item we have $f((\cal{G}_v)_W)=(\cal{H}_{u_1})_W$ and $f((\cal{G}_{\gamma v})_{\gamma \cdot W})=(\cal{H}_{u_2})_{\gamma \cdot W}$; and
\item the map $x\mapsto \alpha(\gamma, x)\in \Lambda$ is constant on $W$ with its value $\lambda$.
\end{itemize}
For any $(\delta, x)\in (\cal{G}_{\gamma v})_{\gamma \cdot W}$, apply $\alpha$ to the equality
\[(\delta, x)=(\gamma, (\gamma^{-1}\delta)\cdot x)(\gamma^{-1}\delta \gamma, \gamma^{-1}\cdot x)(\gamma^{-1}, x).\]
By assumption, the equalities
\[\alpha(\gamma, (\gamma^{-1}\delta)\cdot x)=\lambda,\quad \alpha(\gamma^{-1}, x)=\alpha(\gamma, \gamma^{-1}\cdot x)^{-1}=\lambda^{-1}\]
hold.
Since $(\gamma^{-1}\delta \gamma, \gamma^{-1}\cdot x)\in (\cal{G}_v)_W$, the inclusion $(\cal{H}_{u_2})_{\gamma \cdot W}\subset (\cal{H}_{\lambda u_1})_{\gamma \cdot W}$ holds, and thus $u_2=\lambda u_1$ by Claim \ref{claim-dis-ver}.
The desired equality follows.
\end{proof}


\section{Couplings with unknown groups}\label{sec-mec-arb}

Let $\Gamma =\Gamma_1\ast_A\Gamma_2$ be the group in Assumption $(\star)$, and let $\imath \colon \Gamma \rightarrow \aut^*(T)$ be the homomorphism associated with the action of $\Gamma$ on the Bass-Serre tree $T$.
If $\Lambda$ is a discrete group and $\Sigma$ is a coupling of $\Gamma$ and $\Lambda$, then there exist a homomorphism $\rho \colon \Lambda \rightarrow \aut^*(T)$ and an almost $(\Gamma \times \Lambda)$-equivariant Borel map $\Phi \colon \Sigma \rightarrow (\aut^*(T), \imath, \rho)$ by Corollary \ref{cor-coup-tree}.
In this section, we study the structure of $\Lambda$ by using $\rho$ and $\Phi$, relying on the Bass-Serre theory.

To apply this theory to the action of $\Lambda$ on $T$ via $\rho$, each element of $\Lambda$ should act on $T$ without inversions.
In Section \ref{subsec-circ}, we explain how this difficulty is avoided.
Section \ref{subsec-pat} is devoted to fundamental observations on small couplings inside $\Sigma$ each of which is associated with a simplex of $T$ and is a coupling of the stabilizers of that simplex in $\Gamma$ and in $\Lambda$.
We obtain information on the action of $\Lambda$ on $T$ through these small couplings.
As a consequence, Theorem \ref{thm-lqn-str} is proved.
Section \ref{subsec-quo} is preliminaries for studying the case where $\Gamma_1$ and $\Gamma_2$ are rigid in a sense of measure equivalence.
If they are lattices in higher rank simple Lie groups, then the observation there is a crucial step for the proof of superrigidity results stated in Theorems \ref{thm-oe-rigidity} and \ref{thm-mer}.
In Section \ref{subsec-coco-cri}, we present sufficient conditions for the action of $\Lambda$ on $T$ to be locally cofinite and to be cocompact when $\Gamma_1$ and $\Gamma_2$ are coupling rigid with respect to some standard Borel groups.
These conditions are formulated in terms of those standard Borel groups and the image of $\Gamma_1$ and $\Gamma_2$ in them.

\subsection{Reduction from $\aut^*(T)$ to $\aut(T)$}\label{subsec-circ}

Let $\Gamma =\Gamma_1\ast_A\Gamma_2$ be the group in Assumption $(\star)$.
Let $\imath \colon \Gamma \rightarrow \aut^*(T)$ be the homomorphism associated with the action of $\Gamma$ on the Bass-Serre tree $T$.
We denote by $\aut(T)$ the subgroup of $\aut^*(T)$ consisting of orientation-preserving automorphisms of $T$, which is of index two.
The image $\imath(\Gamma)$ is contained in $\aut(T)$.
Let $\Lambda$ be a discrete group, and let $(\Sigma, m)$ be a coupling of $\Gamma$ and $\Lambda$.
There then exist a homomorphism $\rho \colon \Lambda \rightarrow \aut^*(T)$ and an almost $(\Gamma \times \Lambda)$-equivariant Borel map $\Phi \colon \Sigma \rightarrow (\aut^*(T), \imath, \rho)$ by Corollary \ref{cor-coup-tree}.

We may assume that the Borel subset $\Phi^{-1}(\aut(T))$ of $\Sigma$ has positive measure by replacing $\rho$ and $\Phi$ as follows:
If $\Phi^{-1}(\aut(T))$ is of measure zero, then choose $f\in \aut^*(T)\setminus \aut(T)$ and define a homomorphism $\rho_f\colon \Lambda \rightarrow \aut^*(T)$ and a Borel map $\Phi_f\colon \Sigma \rightarrow \aut^*(T)$ by the formulas
\[\rho_f(\lambda)=f\rho(\lambda)f^{-1},\quad \Phi_f(x)=\Phi(x)f^{-1}\]
for $\lambda \in \Lambda$ and $x\in \Sigma$.
The map $\Phi_f\colon \Sigma \rightarrow (\aut^*(T), \imath, \rho_f)$ is then almost $(\Gamma \times \Lambda)$-equivariant, and $\Phi_f^{-1}(\aut(T))$ has positive measure.

On the assumption that $\Phi^{-1}(\aut(T))$ has positive measure, we set
\[\Lambda_+=\rho^{-1}(\rho(\Lambda)\cap \aut(T)),\quad \Sigma_+=\Phi^{-1}(\aut(T)).\]
It follows that $[\Lambda :\Lambda_+]\leq 2$ and that $\Sigma_+$ is a coupling of $\Gamma$ and $\Lambda_+$.
More specifically, the action $\Gamma \times \Lambda \c \Sigma$ is induced from the action $\Gamma \times \Lambda_+\c \Sigma_+$.
It is therefore enough to study the latter action to know the structure of the original coupling $\Sigma$.
Note that both $\rho(\Lambda_+)$ and $\Phi(\Sigma_+)$ are contained in $\aut(T)$ and that the map $\Phi \colon \Sigma_+\rightarrow (\aut(T), \imath, \rho)$ is almost $(\Gamma \times \Lambda_+)$-equivariant.

Since $\Lambda_+$ acts on $T$ via $\rho$ without inversions, we can apply the Bass-Serre theory to this action to know the structure of $\Lambda_+$.
In the next subsection, we give useful information on the stabilizers of simplices of $T$ and the quotient graph for the action of $\Lambda_+$ on $T$, which will be investigated in the following general setting.

\medskip

\noindent {\bf Assumption $(\circ)$}\label{ass-circ}: We set the notation as follows:
\begin{itemize}
\item Let $\Gamma =\Gamma_1\ast_A\Gamma_2$ be an amalgamated free product of discrete groups.
\item Let $T$ denote the Bass-Serre tree associated with the decomposition of $\Gamma$.
\item Let $\imath \colon \Gamma \rightarrow \aut^*(T)$ be the homomorphism associated with the action of $\Gamma$ on $T$.
\item We orient $T$ so that $\gamma \Gamma_1\in V(T)$ is the origin of the edge $\gamma A\in E(T)$ for each $\gamma \in \Gamma$. Let $\aut(T)$ denote the group of simplicial automorphisms of $T$ preserving this orientation.
\end{itemize}
Suppose that we have
\begin{itemize}
\item a discrete group $\Lambda$ and a coupling $(\Sigma, m)$ of $\Gamma$ and $\Lambda$;
\item a homomorphism $\rho \colon \Lambda \rightarrow \aut(T)$; and
\item an almost $(\Gamma \times \Lambda)$-equivariant Borel map $\Phi \colon \Sigma \rightarrow (\aut(T), \imath, \rho)$.
\end{itemize}


\subsection{Small couplings}\label{subsec-pat}

On Assumption $(\circ)$, for each $s\in V(T)\cup E(T)$, we set 
\[\stab(s)=\{\, \varphi \in \aut(T)\mid \varphi(s)=s\, \},\]\label{stab}
which is a Borel subgroup of $\aut(T)$, and put
\[\Sigma_s=\Phi^{-1}(\stab(s)),\quad \Gamma_s=\imath^{-1}(\imath(\Gamma)\cap \stab(s)),\quad \Lambda_s=\rho^{-1}(\rho(\Lambda)\cap \stab(s)).\]
It follows that $\Sigma_s$ is a $(\Gamma_s\times \Lambda_s)$-invariant Borel subset of $\Sigma$.
The first important observation of this subsection is that $\Sigma_s$ is a coupling of $\Gamma_s$ and $\Lambda_s$.
Since $\Gamma_s$ is a conjugate of one of the subgroups $\Gamma_1$, $\Gamma_2$ and $A$ of $\Gamma$, it brings valuable information on the stabilizer $\Lambda_s$.
For each $i=1, 2$, let us denote by $V_i(T)$\label{vit} the subset $\Gamma /\Gamma_i$ of $V(T)=\Gamma /\Gamma_1\sqcup \Gamma /\Gamma_2$.

\begin{lem}\label{lem-pos-stab}
On Assumption $(\circ)$, let $S$ be one of $V_1(T)$, $V_2(T)$ and $E(T)$ and pick $s, s'\in S$.
We put
\[\aut(s, s')=\{\, f\in \aut(T)\mid f(s)=s'\, \},\quad \Sigma_s^{s'}=\Phi^{-1}(\aut(s, s')).\]
Then we have $m(\Sigma_s^{s'})>0$.
In particular, we have $m(\Sigma_s)>0$.
\end{lem}

\begin{proof}
Since the action $\Gamma \c S$ is transitive and since for each $\gamma \in \Gamma$, the equality $\imath(\gamma)\aut(s, s')=\aut(s, \gamma s')$ holds, we have $\aut(T)=\bigcup_{\gamma \in \Gamma}\imath(\gamma)\aut(s, s')$.
The lemma thus follows.
\end{proof}

\begin{lem}\label{lem-small-coup}
On Assumption $(\circ)$, we pick $s\in V(T)\cup E(T)$.
Then the following assertions hold:
\begin{enumerate}
\item If $Y\subset \Sigma_s$ is a fundamental domain for the action $\Gamma_s\c \Sigma_s$, then $Y$ is also a fundamental domain for the action $\Gamma \c \Sigma$.
In particular, $m(Y)$ is finite.
\item If $X\subset \Sigma_s$ is a fundamental domain for the action $\Lambda_s\c \Sigma_s$, then $X$ is contained in a fundamental domain for the action $\Lambda \c \Sigma$.
In particular, $m(X)$ is finite.
\item $(\Sigma_s, m|_{\Sigma_s})$ is a coupling of $\Gamma_s$ and $\Lambda_s$.
\end{enumerate}
\end{lem}

\begin{proof}
If $\gamma \in \Gamma$ satisfies $m(\gamma Y\triangle Y)>0$, then there exist $x, y\in Y$ with $\gamma x=y$, $\Phi(x)\in \stab(s)$ and $\imath(\gamma)\Phi(x)=\Phi(y)\in \stab(s)$.
We thus have $\gamma \in \Gamma_s$.
It follows that $\gamma$ is neutral because $Y$ is a fundamental domain for the action $\Gamma_s\c \Sigma_s$.
The equality $\Gamma Y=\Sigma$ holds because we have $\aut(T)=\bigcup_{\gamma \in \Gamma}\imath(\gamma)\stab(s)$.
Assertion (i) is proved.

Similarly, we can show that if $\lambda \in \Lambda$ satisfies $m(\lambda X\triangle X)>0$, then $\lambda$ is neutral.
It follows that $X$ is contained in a fundamental domain for the action $\Lambda \c \Sigma$.
Assertion (ii) is proved.
Assertion (iii) follows from assertions (i) and (ii).
\end{proof}

\begin{lem}\label{lem-trans-erg}
On Assumption $(\circ)$, let $S$ be one of $V_1(T)$, $V_2(T)$ and $E(T)$, and pick $s\in S$.
Then the action $\rho(\Lambda)\c S$ is transitive if and only if there exists a fundamental domain for the action $\Lambda \c \Sigma$ contained in $\Sigma_s$. 
This is the case if the action $\Gamma_s\c \Sigma /\Lambda$ is ergodic.
\end{lem}

\begin{proof}
Suppose that the action $\rho(\Lambda)\c S$ is transitive.
The equality $\aut(T)=\bigcup_{\lambda \in \Lambda}\stab(s)\rho(\lambda)^{-1}$ then holds.
Any fundamental domain for the action $\Lambda_s\c \Sigma_s$ is thus a fundamental domain for the action $\Lambda \c \Sigma$.
It follows that there exists a fundamental domain for the action $\Lambda \c \Sigma$ contained in $\Sigma_s$.

Conversely, if there exists a fundamental domain for the action $\Lambda \c \Sigma$ contained in $\Sigma_s$, then by Lemma \ref{lem-pos-stab}, for each $s'\in S$, there exists $\lambda \in \Lambda$ with $m(\lambda \Sigma_s\cap \Sigma_{s'}^s)>0$.
We thus have $\rho(\lambda)(s)=s'$.

Suppose that the action $\Gamma_s\c \Sigma /\Lambda$ is ergodic, and let $X\subset \Sigma$ be a fundamental domain for the action $\Lambda \c \Sigma$ which contains a fundamental domain $X_1$ for the action $\Lambda_s\c \Sigma_s$.
Note that $X_1=X\cap \Sigma_s$, and put $X_2=X\setminus X_1$.
We claim that $m((\Gamma_s\times \Lambda)X_1\triangle X_2)=0$.
Otherwise, there would exist $\gamma \in \Gamma_s$ and $\lambda \in \Lambda$ such that $m((\gamma, \lambda)X_1\triangle X_2)>0$.
It follows from $\gamma X_1\subset \Sigma_s$ that $m(\lambda \Sigma_s\triangle X_2)>0$.
This is a contradiction because $m(\Lambda X_1\triangle \Lambda X_2)=0$ and $\Lambda_sX_1=\Sigma_s$.
The claim implies
\[m((\Gamma_s\times \Lambda)X_1\triangle (\Gamma_s\times \Lambda)X_2)=0.\]
It follows that both $X_1$ and $X_2$ are invariant under the action $\Gamma_s\c \Sigma /\Lambda$ when $X$ is identified with $\Sigma /\Lambda$.
Since the action $\Gamma_s\c \Sigma /\Lambda$ is ergodic, we have $m(X_2)=0$.
We thus have $X=X_1$, which is contained in $\Sigma_s$.
\end{proof}

If the action $A\c \Sigma /\Lambda$ is ergodic, then so is the action $\Gamma_s\c \Sigma /\Lambda$ for any $s\in S$ and any $S\in \{ V_1(T), V_2(T), E(T)\}$.
In this case, for any $e\in E(T)$, a fundamental domain for the action $\Lambda_e\c \Sigma_e$ is in fact a fundamental domain for either of the actions $\Lambda \c \Sigma$ and $\Lambda_v\c \Sigma_v$ for any $v\in V(T)$ with $v\in \partial e$, where $\partial e$ denotes the boundary of the edge $e$.
There exists a fundamental domain for the action $\Gamma \c \Sigma$ contained in $\Sigma_e$ for any $e\in E(T)$ by Lemma \ref{lem-small-coup} (i), which is also a fundamental domain for the action $\Gamma_v\c \Sigma_v$ for any $v\in V(T)$ with $v\in \partial e$.
We therefore obtain the following description of the structure of $\Lambda$ via the Bass-Serre theory.

\begin{cor}\label{cor-factor-me}
On Assumption $(\circ)$, let $v_i\in V_i(T)$, $e\in E(T)$ be the simplices of $T$ corresponding to the cosets containing the neutral element, and put $\Lambda_i=\Lambda_{v_i}$ for $i=1, 2$ and $B=\Lambda_e$.
If the action $A\c \Sigma /\Lambda$ is ergodic, then $\Lambda$ is isomorphic to the amalgamated free product $\Lambda_1\ast_B\Lambda_2$ with $\Gamma_i\sim_{ME}\Lambda_i$ for $i=1, 2$ and $A\sim_{ME}B$. Moreover, their couplings have the same coupling constant.
\end{cor}

The following corollary is obtained by combining Corollaries \ref{cor-coup-tree}, \ref{cor-factor-me} and the argument in Section \ref{subsec-circ}.
It proves assertion (ii) of Theorem \ref{thm-lqn-str}.
Assertion (i) of the theorem follows from Lemma \ref{lem-small-coup}.

\begin{cor}
On Assumption $(\star)$, suppose that we have a discrete group $\Lambda$ and a coupling $\Sigma$ of the group $\Gamma =\Gamma_1\ast_A\Gamma_2$ and $\Lambda$ such that the action $A\c \Sigma /\Lambda$ is ergodic.
Then there exists a subgroup $\Lambda^{+}$ of $\Lambda$ with its index at most two such that $\Lambda^{+}$ is decomposed as an amalgamated free product $\Lambda_1\ast_B\Lambda_2$, where $\Gamma_i\sim_{ME}\Lambda_i$ for $i=1, 2$ and $A\sim_{ME}B$ hold, and their couplings have the same coupling constant.
\end{cor}


\subsection{Quotients of couplings with countable images}\label{subsec-quo}

The following lemma is a brief observation on equivariant quotient maps of couplings with their images countable, which will be applied many times in the sequel.
We say that a positive measure on a standard Borel space $\Omega$ is {\it atomic} if there is a countable subset of $\Omega$ whose complement is of measure zero.

\begin{lem}\label{lem-image-atomic}
Let $\Gamma$ and $\Lambda$ be discrete groups, and let $(\Sigma, m)$ be a coupling of $\Gamma$ and $\Lambda$. 
Suppose that we have
\begin{itemize}
\item a standard Borel group $G$;
\item homomorphisms $\pi \colon \Gamma \rightarrow G$ and $\rho \colon \Lambda \rightarrow G$ with $\ker \pi$ finite; and
\item an almost $(\Gamma \times \Lambda)$-equivariant Borel map $\Phi \colon \Sigma \rightarrow (G, \pi, \rho)$ such that the measure $\Phi_*m$ on $G$ is atomic.
\end{itemize}
Pick $g\in G$ with $\Phi_*m(\{ g\})>0$, and define a homomorphism $\rho_g\colon \Lambda \rightarrow G$ and a Borel map $\Phi_g\colon \Sigma \rightarrow G$ by
\[\rho_g(\lambda)=g\rho(\lambda)g^{-1},\quad \Phi_g(x)=\Phi(x)g^{-1}\]
for $\lambda \in \Lambda$ and $x\in \Sigma$.
Then the following assertions hold:
\begin{enumerate}
\item The map $\Phi_g\colon \Sigma \rightarrow (G, \pi, \rho_g)$ is almost $(\Gamma \times \Lambda)$-equivariant.
\item The support of the measure $(\Phi_g)_*m$ on $G$ contains the neutral element of $G$ and is contained in $\comm_{G}(\pi(\Gamma))$.
\item $\ker \rho =\ker \rho_g$ is finite, and $\pi(\Gamma)$ and $\rho_g(\Lambda)$ are commensurable in $G$, that is, the subgroup $\pi(\Gamma)\cap \rho_g(\Lambda)$ is of finite index in both $\pi(\Gamma)$ and $\rho_g(\Lambda)$.
\end{enumerate}
\end{lem}

\begin{proof}
Put $M=(\Phi_g)_*m$.
Assertion (i) and positivity of $M(\{ e\})$ are clear.
Let $S\subset G$ be the subset consisting of all $h\in G$ with $M(\{ h\})>0$.
It follows from the finiteness of $\ker \pi$ that $0<M(\{ s\})<\infty$ for any $s\in S$ and that $\ker \rho$ is finite.
For each $s\in S$, the subset $G_s=\pi(\Gamma)s\rho_g(\Lambda)$ is invariant for the action $\Gamma \times \Lambda \c (G, \pi, \rho_g)$, and is a coupling of $\pi(\Gamma)$ and $\rho_g(\Lambda)$.
Since the action $\pi(\Gamma)\c G_s/\rho_g(\Lambda)$ defined by left multiplication admits a finite invariant measure, the stabilizer $\pi(\Gamma)\cap s\rho_g(\Lambda)s^{-1}$ of $s$ for the action is a finite index subgroup of $\pi(\Gamma)$.
Similarly, $\rho_g(\Lambda)\cap s^{-1}\pi(\Gamma)s$ is a finite index subgroup of $\rho_g(\Lambda)$.
It follows that $\pi(\Gamma)$ and $s\rho_g(\Lambda)s^{-1}$ are commensurable in $G$. 
Since $S$ contains the neutral element of $G$, $\pi(\Gamma)$ and $\rho_g(\Lambda)$ are commensurable in $G$, and $S$ is contained in $\comm_{G}(\pi(\Gamma))$.
\end{proof}

We recall a fundamental fact on couplings of lattices in higher rank simple Lie groups.
Aassertion (i) follows from Lemma 6.1 in \cite{furman-mer} and Zimmer's argument in \cite{zim-trans} (see Section 7 in \cite{furman-mer} for details).
Assertion (ii) is proved in Lemma 4.6 of \cite{furman-mer}.

\begin{thm}\label{thm-hrl-measure}
Let $G$ be a non-compact connected simple Lie group with its center finite and its real rank at least two, and let $\Gamma$ be a lattice in $G$.
Suppose that we have a discrete group $\Lambda$, a coupling $(\Sigma, m)$ of $\Gamma$ and $\Lambda$, a homomorphism $\rho \colon \Lambda \rightarrow \aut(\ad G)$, and an almost $(\Gamma \times \Lambda)$-equivariant Borel map $\Phi \colon \Sigma \rightarrow (\aut(\ad G), \pi, \rho)$, where $\pi$ is the natural homomorphism from $G$ into $\aut(\ad G)$.
Then the following assertions hold:
\begin{enumerate}
\item The group $\ker \rho$ is finite, and $\rho(\Lambda)$ is a lattice in $\aut(\ad G)$.
\item The measure $\Phi_*m$ is a linear combination of the Haar measure on cosets of $\ad G$ in $\aut(\ad G)$ and atomic measures.
If the action $\Gamma \times \Lambda \c (\Sigma, m)$ is ergodic, then $\Phi_*m$ is either the Haar measure on some cosets of $\ad G$ in $\aut(\ad G)$ or atomic.
\end{enumerate}
\end{thm}

On Assumption $(\circ)$, fix $v\in V(T)$ and suppose that $\Gamma_v$ is isomorphic to a lattice in the Lie group $G$ in Theorem \ref{thm-hrl-measure}.
In the rest of this subsection, we investigate the map $\Phi_v$ from the small coupling $(\Sigma_v, m|_{\Sigma_v})$ of $\Gamma_v$ and $\Lambda_v$ into $\aut(\ad G)$ obtained by Theorems \ref{thm-furman-rep} and \ref{thm-me-hrl}.
We will prove that the measure $(\Phi_v)_*(m|_{\Sigma_v})$ is atomic as an application of the general fact in Theorem \ref{thm-hrl-measure} and of Moore's celebrated theorem on unitary representations of Lie groups \cite{moore}.
This leads to superrigidity of amalgamated free products of higher rank lattices, which is not satisfied by higher rank lattices themselves.

Moore's theorem is applied in the proof of Proposition \ref{prop-moore}, where ``irreducibility'' of the ME cocycle associated with the standard coupling $\ad G$ of two lattices in $\ad G$ is established.
On the other hand, since the coupling $\Sigma_v$ contains the smaller coupling $\Sigma_e$, where $e\in E(T)$ is an edge with $v\in \partial e$, the ME cocycle associated with the coupling $\Sigma_v$ is ``reducible''.
It turns out that $(\Phi_v)_*(m|_{\Sigma_v})$ is not the Haar measure.
This idea will be demonstrated precisely in Corollary \ref{cor-moore}.

\begin{prop}\label{prop-moore}
Let $G$ be a non-compact connected simple Lie group with its center finite, and let $\Gamma$ and $\Lambda$ be lattices in $G$.
Suppose that we have
\begin{itemize}
\item a coupling $(\Sigma, m)$ of $\Gamma$ and $\Lambda$; and
\item an almost $(\Gamma \times \Lambda)$-equivariant Borel map $\Phi \colon \Sigma \rightarrow (\ad G, \pi, \pi)$ with $\Phi_*m$ the Haar measure on $\ad G$,
\end{itemize}
where $\pi$ is the natural homomorphism from $G$ into $\ad G$.
Choose a fundamental domain $X\subset \Sigma$ for the action $\Lambda \c \Sigma$, and let $\alpha \colon \Gamma \times X\rightarrow \Lambda$ be the associated ME cocycle.

Then for any infinite subgroup $A$ of $\Gamma$, any infinite index subgroup $B$ of $\Lambda$ and any Borel subset $Z$ of $X$ of positive measure, it is impossible that the inclusion $\alpha((A\ltimes X)_Z)\subset B$ holds up to null sets.
\end{prop}

\begin{proof}
We may assume that the center of $G$ is trivial and may identify $G$ and $\ad G$. 
Suppose that there are $A$, $B$ and $Z$ as in the statement, satisfying the inclusion $\alpha((A\ltimes X)_Z)\subset B$.
We first prove that we may assume that $Z$ is $A$-invariant by replacing $X$ appropriately.
Take a countable Borel partition $A\cdot Z=Z\sqcup (\bigsqcup_{n\in \mathbb{Z}_{>0}}Z_n)$ of the saturation $A\cdot Z$ such that for each $n\in \mathbb{Z}_{>0}$,
\begin{itemize}
\item there exists $a_n\in A$ with $Z_n\subset a_n\cdot Z$; and
\item the map $\alpha(a_n, \cdot)$ is constant on $a_n^{-1}\cdot Z_n$ with the value $\lambda_n\in \Lambda$. 
\end{itemize}
We set $Z_0=Z$, $a_0=e\in A$ and $\lambda_0=e\in \Lambda$. Define a Borel map $\varphi \colon X\rightarrow \Lambda$ by
\begin{align*}
\varphi(x)=
\begin{cases}
\lambda_n^{-1} & \textrm{if}\ x\in Z_n \ \textrm{and} \ n\in \mathbb{Z}_{\geq 0},\\
e & \textrm{if}\ x\in X\setminus (A\cdot Z),
\end{cases}
\end{align*}
and define a Borel cocycle $\alpha_{\varphi}\colon \Gamma \times X\rightarrow \Lambda$ by $\alpha_{\varphi}(\gamma, x)=\varphi(\gamma \cdot x)\alpha(\gamma, x)\varphi(x)^{-1}$ for $\gamma \in \Gamma$ and $x\in X$.
Pick $a\in A$ and $x\in A\cdot Z$.
We choose $n, m\in \mathbb{Z}_{\geq 0}$ with $x\in Z_n$ and $a\cdot x\in Z_m$.
We then have
\begin{align*}
\alpha_{\varphi}(a, x)&=\varphi(a\cdot x)\alpha(a, x)\varphi(x)^{-1}\\
&=\varphi(a\cdot x)\alpha(a_m, (a_m^{-1}a)\cdot x)\alpha(a_m^{-1}aa_n, a_n^{-1}\cdot x)\alpha(a_n^{-1}, x)\varphi(x)^{-1}\\
&=\lambda_m^{-1}\lambda_m\alpha(a_m^{-1}aa_n, a_n^{-1}\cdot x)\lambda_n^{-1}\lambda_n\\
&=\alpha(a_m^{-1}aa_n, a_n^{-1}\cdot x)\in B.
\end{align*}
The inclusion $\alpha_{\varphi}(A\ltimes (A\cdot Z))\subset B$ thus holds.
We define the Borel subset
\[X_{\varphi}=\{\, \varphi(x)x\in \Sigma \mid x\in X\, \}\]
of $\Sigma$, which is a fundamental domain for the action $\Lambda \c \Sigma$.
The associated ME cocycle $\alpha'\colon \Gamma \times X_{\varphi}\rightarrow \Lambda$ satisfies the equality $\alpha'(\gamma, \varphi(x)x)=\alpha_{\varphi}(\gamma, x)$ for any $\gamma \in \Gamma$ and a.e.\ $x\in X$ because we have
\[(\gamma, \alpha_{\varphi}(\gamma, x))\varphi(x)x=(\gamma, \varphi(\gamma \cdot x)\alpha(\gamma, x))x=\varphi(\gamma \cdot x)(\gamma \cdot x)\in X_{\varphi}\]  
for any $\gamma \in \Gamma$ and a.e.\ $x\in X$.
We may therefore assume that $Z$ is $A$-invariant after replacing $X$ with $X_{\varphi}$.

The map $\Phi \colon \Sigma \rightarrow (G, \pi, \pi)$ induces a $\Gamma$-equivariant Borel map $\Phi_0\colon \Sigma /B\rightarrow G/B$, where $\Gamma$ acts on $G/B$ by left multiplication.
Let $p\colon \Sigma \rightarrow \Sigma /B$ be the canonical projection, and choose a fundamental domain $F$ for the action $B\c \Sigma$ containing $Z$.
We denote by $m_0$ and $m_1$ the restrictions of $m$ to $Z$ and $F$, respectively, and put $M_0=(\Phi_0\circ p)_*m_0$ and $M_1=(\Phi_0\circ p)_*m_1$.
By assumption, $M_1$ is the measure on $G/B$ induced by the Haar measure on $G$.
There exists an $L^{1}$-function $f$ on $G/B$ with respect to $M_1$ such that $0\leq f(x)\leq 1$ for $M_1$-a.e.\ $x\in G/B$ and
\[M_0(S)=\int_S f(x)\,dM_1(x)\]
for any Borel subset $S$ of $G/B$ since $m_0(W)\leq m_1(W)$ for any Borel subset $W$ of $\Sigma$.
The function $f$ is also $L^{2}$ with respect to $M_1$ because of its boundedness. 
Since $p_*m_0$ is $A$-invariant, so is the non-zero vector $f\in L^{2}(G/B, M_1)$. 
On the other hand, Moore's theorem (see \cite{moore} or Theorem 2.2.19 in \cite{zim-book}) tells us that there exists no non-zero $A$-invariant vector for any unitary representation of $G$ without non-zero $G$-invariant vectors.
This is a contradiction.
\end{proof}

\begin{cor}\label{cor-moore}
On Assumption $(\circ)$, fix $i=1, 2$.
We denote by $v_i\in V_i(T)=\Gamma /\Gamma_i$ the vertex corresponding to the coset containing the neutral element, and put $\Sigma_i=\Sigma_{v_i}$ and $\Lambda_i=\Lambda_{v_i}$.
Suppose that
\begin{itemize}
\item $\Gamma_i$ is a lattice in a non-compact connected simple Lie group $G$ with its center finite and its real rank at least two; and
\item $A$ is an infinite and infinite index subgroup of $\Gamma_i$.
\end{itemize}
Let $\rho_i\colon \Lambda_i\rightarrow \aut (\ad G)$ be a homomorphism and $\Phi_i\colon \Sigma_i\rightarrow (\aut(\ad G), \pi, \rho_i)$ an almost $(\Gamma_i \times \Lambda_i)$-equivariant Borel map, where $\pi$ is the natural homomorphism from $G$ into $\aut(\ad G)$.
Then the measure $(\Phi_i)_*(m|_{\Sigma_i})$ on $\aut (\ad G)$ is atomic. 
\end{cor}

\begin{proof}
We first note that $\ker \rho_i$ is finite and $\rho_i(\Lambda_i)$ is a lattice in $\aut(\ad G)$ by Theorem \ref{thm-hrl-measure} (i).
Let $p\colon \Sigma_i\rightarrow C$ be the ergodic decomposition for the action $\Gamma_i\times \Lambda_i\c \Sigma_i$. Put $\Sigma_i^c=p^{-1}(c)$ for $c\in C$ and let $m|_{\Sigma_i}=\int_{C}m^cd\xi(c)$ be the disintegration with respect to $p$.
Theorem \ref{thm-hrl-measure} (ii) shows that for $\xi$-a.e.\ $c\in C$, the measure $(\Phi_i)_*m^c$ on $\aut (\ad G)$ is either atomic or equal to the Haar measure on some cosets of $\ad G$ in $\aut(\ad G)$.
If $(\Phi_i)_*m^c$ is atomic for $\xi$-a.e.\ $c\in C$, then there exists $g\in \aut(\ad G)$ such that $\pi(\Gamma_i)$ and $g\rho_i(\Lambda_i)g^{-1}$ are commensurable in $\aut(\ad G)$ by Lemma \ref{lem-image-atomic}.
It follows that for $\xi$-a.e.\ $c\in C$, the measure $(\Phi_i)_*m^c$ is supported on $\comm_{\aut(\ad G)}(\pi(\Gamma_i))g$.
To prove the corollary, it is therefore enough to prove that $(\Phi_i)_*m^c$ is atomic for $\xi$-a.e.\ $c\in C$.

Let $e\in E(T)=\Gamma /A$ be the edge corresponding to the neutral element and put $\Sigma_e^c=\Sigma_e\cap \Sigma_i^c$ for $c\in C$.
Choose a fundamental domain $X$ for the action $\Lambda_i\c \Sigma_i$ such that $X\cap \Sigma_e$ is a fundamental domain for the action $\Lambda_e\c \Sigma_e$. 
For $\xi$-a.e.\ $c\in C$, the action $\Gamma_i\times \Lambda_i\c (\Sigma_i^c, m^c)$ defines a coupling and the ME cocycle $\alpha^c\colon \Gamma_i\times (X\cap \Sigma_i^c)\rightarrow \Lambda_i$ satisfies the inclusion $\alpha^c(A\times (X\cap \Sigma^c_e))\subset \Lambda_e$ because we have $(A\times \Lambda_e)(X\cap \Sigma_e^c)=\Sigma^c_e$.
Since $A$ is an infinite index subgroup of $\Gamma_i$, $\Lambda_e$ is an infinite index subgroup of $\Lambda_i$.
By Proposition \ref{prop-moore}, for $\xi$-a.e.\ $c\in C$, $(\Phi_i)_*m^c$ is not the Haar measure on some cosets of $\ad G$ in $\aut(\ad G)$.
\end{proof}


\subsection{Finiteness properties of actions on trees}\label{subsec-coco-cri}

The aim of this subsection is to describe sufficient conditions for the action of $\Lambda$ on the tree $T$ to be locally cofinite and to be cocompact in the following setting.

\medskip

\noindent {\bf Assumption $(\bullet)$}\label{ass-bullet}: In the notation in Assumption $(\circ)$, we furthermore suppose the following conditions (a) and (b):
\begin{enumerate}
\item[(a)] For each $i=1, 2$, $\Gamma_i$ is coupling rigid with respect to $(G_i, \pi_i)$, where $G_i$ is a standard Borel group and $\pi_i\colon \Gamma_i\rightarrow G_i$ is a homomorphism such that $\ker \pi_i$ is finite and the action $\pi_i(\Gamma_i)\c G_i$ defined by left multiplication admits a Borel fundamental domain. 
\end{enumerate}
For each $v\in V(T)$, we put
\[\Sigma_v=\Phi^{-1}(\stab(v)),\quad \Gamma_v=\imath^{-1}(\imath(\Gamma)\cap \stab(v)),\quad \Lambda_v=\rho^{-1}(\rho(\Lambda)\cap \stab(v)).\]
If $v$ is in $V_i(T)$, then put $G_v=G_i$.
Let $\pi_v\colon \Gamma_v\rightarrow G_v$ be the homomorphism defined by $\pi_v(\gamma)=\pi_i(\gamma_v\gamma \gamma_v^{-1})$ for $\gamma \in \Gamma_v$, where $\gamma_v\in \Gamma$ is chosen so that $\gamma_v\Gamma_v\gamma_v^{-1}=\Gamma_i$.
By Theorem \ref{thm-furman-rep}, there exist
\begin{itemize}
\item a homomorphism $\rho_v\colon \Lambda_v\rightarrow G_v$ with $\ker \rho_v$ finite; and
\item an almost $(\Gamma_v\times \Lambda_v)$-equivariant Borel map $\Phi_v\colon \Sigma_v\rightarrow (G_v, \pi_v, \rho_v)$.
\end{itemize}
We suppose that
\begin{enumerate}
\item[(b)] for any $v\in V(T)$, the measure $(\Phi_v)_*(m|_{\Sigma_v})$ on $G_v$ is atomic.
\end{enumerate}

\medskip

The following lemma will be used to understand relationship between $\pi_v(\Gamma_e)$ and $\rho_v(\Lambda_e)$ when $v$ is an end point of an edge $e$ of $T$.

\begin{lem}\label{lem-support}
On Assumption $(\bullet)$, pick $v\in V(T)$ and let $e\in E(T)$ be an edge with $v\in \partial e$.
When $s$ is either $v$ or $e$, we denote by $\Phi_v(\Sigma_s)$ the support of the measure $(\Phi_v)_*(m|_{\Sigma_s})$ on $G_v$.
Suppose that $\Phi_v(\Sigma_v)$ contains the neutral element of $G_v$.
Then the following assertions hold:
\begin{enumerate}
\item For any $g\in \Phi_v(\Sigma_e)$, we have the inclusion
\[\Phi_v(\Sigma_e)\subset \comm_{G_v}(\pi_v(\Gamma_e))g\cap \comm_{G_v}(\pi_v(\Gamma_v)).\]
\item Let $i\in \{ 1, 2\}$ be the index with $v\in V_i(T)$.
If the group $A$ is proper in $\Gamma_i$ and satisfies the equality ${\rm LQN}_{\Gamma_i}(A)=A$, then for any $g\in \Phi_v(\Sigma_e)$, the following inclusion holds up to null sets: 
\[\Phi_v^{-1}({\rm LQN}_{G_v}(\pi_v(\Gamma_e))g)\subset \Sigma_e.\]
\item On the assumption in assertion (ii), $\ker \rho_v$ is contained in $\Lambda_e$. 
\end{enumerate}
\end{lem}

\begin{proof}
Assertion (i) follows from Lemma \ref{lem-image-atomic}.
We put
\[\bar{\Gamma}_s=\pi_v(\Gamma_s),\quad \bar{\Lambda}_s=\rho_v(\Lambda_s)\]
for each $s\in \{ v, e\}$.
By Lemma \ref{lem-image-atomic}, to prove assertion (ii), it suffices to show the inclusion $\Phi_v^{-1}({\rm LQN}_{G_v}(\bar{\Gamma}_e))\subset \Sigma_e$ if $\Phi_v(\Sigma_e)$ contains the neutral element of $G_v$.

Pick $h\in {\rm LQN}_{G_v}(\bar{\Gamma}_e)\cap \Phi_v(\Sigma_v)$.
Since $\ker \pi_v$ is finite, $\Phi_v^{-1}(\{ h\})$ has finite measure.
It follows from $h\in {\rm LQN}_{G_v}(\bar{\Gamma}_e)$ that the index $[\bar{\Gamma}_e: \bar{\Gamma}_e\cap h\bar{\Gamma}_eh^{-1}]$ is finite, and thus $\bar{\Gamma}_e$ is virtually contained in $h\bar{\Gamma}_eh^{-1}$.
Since $\Phi_v(\Sigma_e)$ contains the neutral element of $G_v$, $\bar{\Gamma}_e$ and $\bar{\Lambda}_e$ are commensurable in $G_v$.
Some finite index subgroup $\bar{\Delta}$ of $\bar{\Gamma}_e\cap \bar{\Lambda}_e$ is thus contained in $h(\bar{\Gamma}_e\cap \bar{\Lambda}_e)h^{-1}$.
We put $\Delta =\pi_v^{-1}(\bar{\Delta})\cap \Gamma_e$, which is a finite index subgroup of $\Gamma_e$.
For each $\gamma \in \Delta$, there exists $\lambda_{\gamma}\in \Lambda_e$ with $\rho_v(\lambda_{\gamma})=h^{-1}\pi_v(\gamma)h$.
It follows that for any $\gamma \in \Delta$ and a.e.\ $x\in \Phi_v^{-1}(\{ h\})$, the equality
\[\Phi_v((\gamma, \lambda_{\gamma})x)=\pi_v(\gamma)h\rho_v(\lambda_{\gamma})^{-1}=h\]
holds, and thus the equality $(\gamma, \lambda_{\gamma})\Phi_v^{-1}(\{ h\})=\Phi_v^{-1}(\{ h\})$ holds up to null sets for any $\gamma \in \Delta$.

Let $u_0\in V(T)$ be the vertex with $\partial e=\{ v, u_0\}$.
Let ${\rm Lk}(v)$ denote the set of vertices in the link of $v$ in $T$.
For each $u\in {\rm Lk}(v)$, we put
\[E_u=\Phi_v^{-1}(\{ h\})\cap \Phi^{-1}(\{\, f\in \aut(T)\mid f(v)=v,\ f(u_0)=u\, \} ).\]
We then have $(\gamma, \lambda_{\gamma})E_u=E_{\gamma u}$ for any $\gamma \in \Delta$ and $u\in {\rm Lk}(v)$.
Pick $u\in {\rm Lk}(v)$ distinct from $u_0$.
Note that the orbit for the action $\Delta\c {\rm Lk}(v)$ containing $u$ consists of infinitely many elements because we have ${\rm LQN}_{\Gamma_v}(\Gamma_e)=\Gamma_e$.
We have $E_{\gamma u}\subset \Phi_v^{-1}(\{ h\})$ for any $\gamma \in \Delta$ and have $m(E_{u_1}\triangle E_{u_2})=0$ for any distinct $u_1, u_2\in {\rm Lk}(v)$.
Since $\Phi_v^{-1}(\{ h\})$ has finite measure and we have $m(E_{\gamma u})=m(E_u)$ for any $\gamma \in \Delta$, the set $E_u$ is of measure zero.
It follows that $\Phi_v^{-1}(\{ h\})$ coincides with $E_{u_0}$ up to null sets, and thus $\Phi_v^{-1}(\{ h\})\subset \Sigma_e$.
Assertion (ii) is proved.
Assertion (iii) follows from assertion (ii). 
\end{proof}

As indicated in Lemma \ref{lem-trans-erg}, for each $S\in \{ V_1(T), V_2(T), E(T)\}$, there is a close connection between the action $\Lambda \c S$ via $\rho$ and fundamental domains for the action $\Lambda_s\c \Sigma_s$ with $s\in S$.
For each $v\in V(T)$, the coupling constant for the coupling $\Sigma_v$ of $\Gamma_v$ and $\Lambda_v$ is related to the size of subgroups of $G_v$ which are commensurable with $\pi_v(\Gamma_v)$.
The following Condition $(\mathsf{V})$ restricts the size of such subgroups of $G_v$. 
Other conditions introduced below are also related to the coupling constants for some small couplings inside $\Sigma$.

When $\Delta_1$ and $\Delta_2$ are subgroups of a group, we write $\Delta_1\asymp \Delta_2$\label{asymp} if the intersection $\Delta_1\cap \Delta_2$ is a finite index subgroup of both $\Delta_1$ and $\Delta_2$.

\begin{notation}\label{notation-vef}
Let $\Gamma$ be a discrete group and $A$ a subgroup of $\Gamma$.
Suppose that we have a standard Borel group $G$ and a homomorphism $\pi \colon \Gamma \rightarrow G$ with $\ker \pi$ finite.
We introduce Condition $(\mathsf{V})$ for the pair $(\Gamma, (G, \pi))$ and Conditions $(\mathsf{E})$, $(\mathsf{F})$ for the triplet $(\Gamma, A, (G, \pi))$ as follows:
\begin{enumerate}
\item[$(\mathsf{V})$] There exists a positive number $c$ such that for any subgroup $\Delta$ of $G$ with $\Delta \asymp \pi(\Gamma)$, we have
\[\frac{[\Delta :\pi(\Gamma)\cap \Delta]}{[\pi(\Gamma): \pi(\Gamma)\cap \Delta]}\leq c.\] 
\item[$(\mathsf{E})$] There exists a positive number $d$ such that for any subgroup $\Delta$ of $G$ with $\Delta \asymp \pi(\Gamma)$ and for any subgroup $\Delta_0$ of $\Delta$ with $\Delta_0\asymp \pi(A)$, we have
\[\frac{[\Delta_0 :\pi(A)\cap \Delta_0]}{[\pi(A): \pi(A)\cap \Delta_0]}\leq d.\]
\item[$(\mathsf{F})$] For any subgroup $\Delta$ of $G$ with $\Delta \asymp \pi(\Gamma)$ and any subgroup $\Delta_0$ of $\Delta$ with $\Delta_0\asymp \pi(A)$, there exists no non-trivial finite normal subgroup of $\Delta_0$.
\end{enumerate}
\end{notation}

\begin{notation}
On Assumption $(\bullet)$, for each $i=1, 2$, we say that Condition $(\mathsf{V})_i$ is satisfied if Condition $(\mathsf{V})$ is satisfied for the pair $(\Gamma_i, (G_i, \pi_i))$.
For each $i=1, 2$ and each $\mathsf{S}\in \{ \mathsf{E}, \mathsf{F}\}$, we say that Condition $(\mathsf{S})_i$ is satisfied if Condition $(\mathsf{S})$ is satisfied for the triplet $(\Gamma_i, A, (G_i, \pi_i))$.
We introduce Condition $(\mathsf{B})$ for $\Lambda$ as follows:
\begin{enumerate}
\item[$(\mathsf{B})$] There exists a positive number $C$ such that the cardinality of any finite subgroup of $\Lambda$ is at most $C$.
\end{enumerate}
\end{notation}

We now present sufficient conditions for the action $\Lambda \c T$ to have finiteness properties.
Examples of groups with Conditions $(\mathsf{V})$ and $(\mathsf{E})$ are found in Section \ref{subsec-diag}.
For each $v\in V(T)$, let ${\rm Lk}(v)$ denote the set of vertices in the link of $v$ in $T$.

\begin{prop}\label{prop-cocompact}
On Assumption $(\bullet)$, the following assertions hold:
\begin{enumerate}
\item Fix $i=1, 2$.
If Condition $(\mathsf{E})_i$ is satisfied, then for any $v\in V_i(T)$, the number of orbits for the action of $\Lambda_v$ on ${\rm Lk}(v)$ is finite.
\item Fix $i=1, 2$.
If Conditions $(\mathsf{V})_i$ and $(\mathsf{B})$ are satisfied, then the number of orbits for the action $\Lambda \c V_i(T)$ is finite.
\item If all of Conditions $(\mathsf{V})_1$, $(\mathsf{V})_2$, $(\mathsf{E})_1$, $(\mathsf{E})_2$ and $(\mathsf{B})$ are satisfied, then the action $\Lambda \c T$ is cocompact.
\item If for any $i=1, 2$, the group $A$ is proper in $\Gamma_i$ and satisfies the equality ${\rm LQN}_{\Gamma_i}(A)=A$, then in both assertions (ii) and (iii), Condition $(\mathsf{B})$ can be replaced with the condition that both Conditions $(\mathsf{F})_1$ and $(\mathsf{F})_2$ are satisfied.
\end{enumerate}
\end{prop}

\begin{proof}
We prove assertion (i).
Fix $i=1, 2$ and pick $v\in V_i(T)$.
We define $L=\{ u_1, u_2, u_3,\ldots \}$ as a set of representatives chosen from each orbit of the action $\Lambda_v\c {\rm Lk}(v)$. 
Choose $\gamma_j\in \Gamma_v$ with $\gamma_ju_1=u_j$ for each $j$.
For three vertices $w_1, w_2, w_3\in V(T)$, we denote by $\aut(w_1; w_2, w_3)$ the Borel subset of $\aut(T)$ consisting of all $f\in \aut(T)$ such that $f(w_1)=w_1$ and $f(w_2)=w_3$.
We then have
\begin{align*}
\stab(v)&=\bigsqcup_{u\in L}\aut(v; u, u_1)\rho(\Lambda_v)=\bigsqcup_{j=1}^{|L|}\imath(\gamma_j)^{-1}\aut(v; u_j, u_j)\rho(\Lambda_v)\\
&=\bigsqcup_{j=1}^{|L|}\imath(\gamma_j)^{-1}\stab(e_j)\rho(\Lambda_v),
\end{align*}
where $e_j\in E(T)$ is the edge connecting $v$ and $u_j$. The equality
\[\Sigma_v=\bigsqcup_{j=1}^{|L|}(\{ \gamma_j^{-1}\} \times \Lambda_v)\Sigma_{e_j}\]
thus holds up to null sets.
If $X_j\subset \Sigma_{e_j}$ is a fundamental domain for the action $\Lambda_{e_j}\c \Sigma_{e_j}$, then the union
\[X=\bigsqcup_{j=1}^{|L|}\gamma_j^{-1}X_j\subset \Sigma_v\]
is a fundamental domain for the action $\Lambda_v\c \Sigma_v$.
If $Y_j\subset \Sigma_{e_j}$ is a fundamental domain for the action $\Gamma_{e_j}\c \Sigma_{e_j}$, then $Y_j$ is also a fundamental domain for the action $\Gamma \c \Sigma$ by Lemma \ref{lem-small-coup} (i).
The value $m(Y_j)$ is hence independent of $j$.
We will estimate the value $m(X_j)/m(Y_j)$ from below by using Condition $(\mathsf{E})_i$ and prove that $|L|$ is finite.

Let $\pi_v\colon \Gamma_v\rightarrow G_v$ and $\rho_v\colon \Lambda_v\rightarrow G_v$ be the homomorphisms with their kernels finite, and let $\Phi_v\colon \Sigma_v\rightarrow (G_v, \pi_v, \rho_v)$ be the almost $(\Gamma_v\times \Lambda_v)$-equivariant Borel map in Assumption $(\bullet)$.
By replacing $\rho_v$ and $\Phi_v$ as in Lemma \ref{lem-image-atomic}, we may assume that the support of the measure $(\Phi_v)_*(m|_{\Sigma_v})$, denoted by $\Phi_v(\Sigma_v)$, contains the neutral element of $G_v$.
It then follows from Lemma \ref{lem-image-atomic} that $\pi_v(\Gamma_v)=\pi_i(\Gamma_i)$ and $\rho_v(\Lambda_v)$ are commensurable in $G_v$ and that $\Phi_v(\Sigma_v)$ is contained in $\comm_{G_v}(\pi_i(\Gamma_i))$.

Note that $\Phi_v(\Sigma_v)$ is a $(\Gamma_v\times \Lambda_v)$-invariant subset of $(G_v, \pi_v, \rho_v)$.
For each orbit $\cal{O}$ for the action $\Gamma_v\times \Lambda_v\c \Phi_v(\Sigma_v)$, we put $\Sigma_{\cal{O}}=\Phi_v^{-1}(\cal{O})$.
For each $j$, choosing $g\in \cal{O}\cap \Phi_v(\Sigma_{e_j})$, we have
\begin{align*}
\frac{m(X_j\cap \Sigma_{\cal{O}})}{m(Y_j\cap \Sigma_{\cal{O}})}&=\frac{[\pi_v(\Gamma_{e_j}): \pi_v(\Gamma_{e_j})\cap g\rho_v(\Lambda_{e_j})g^{-1}]}{[g\rho_v(\Lambda_{e_j})g^{-1}: \pi_v(\Gamma_{e_j})\cap g\rho_v(\Lambda_{e_j})g^{-1}]}\frac{|\ker \pi_v\cap \Gamma_{e_j}|}{|\ker \rho_v\cap \Lambda_{e_j}|}\\
&\geq (d_i|\ker \rho_v|)^{-1},
\end{align*}
where $d_i$ is the constant in Condition $(\mathsf{E})_i$.
For each $j$, we then have
\[m(X_j)=\sum_{\cal{O}}m(X_j\cap \Sigma_{\cal{O}})\geq (d_i|\ker \rho_v|)^{-1}\sum_{\cal{O}}m(Y_j\cap \Sigma_{\cal{O}})=(d_i|\ker \rho_v|)^{-1}m(Y_j),\]
where the sum is taken over all orbits $\cal{O}$ for the action $\Gamma_v\times \Lambda_v\c \Phi_v(\Sigma_v)$.
It turns out that
\[m(X)=\sum_{j=1}^{|L|}m(X_j)\geq (d_i|\ker \rho_v|)^{-1}\sum_{j=1}^{|L|}m(Y_j)=(d_i|\ker \rho_v|)^{-1}m(Y_1)|L|.\]
Since $m(X)$ is finite, $|L|$ is finite.
Assertion (i) is proved.

We next prove assertion (ii).
Fix $i=1, 2$.
Let $K=\{ w_1, w_2, w_3,\ldots \}$ be a set of representatives chosen from each orbit for the action $\Lambda \c V_i(T)$.
Choose $\gamma_j\in \Gamma$ with $\gamma_jw_1=w_j$ for each $j$.
As in the proof of assertion (i), we obtain the equalities
\[\aut(T)=\bigsqcup_{j=1}^{|K|}\imath(\gamma_j)^{-1}\stab(w_j)\rho(\Lambda),\quad \Sigma =\bigsqcup_{j=1}^{|K|}(\{ \gamma_j^{-1}\} \times \Lambda)\Sigma_{w_j}.\]
Let $X_j\subset \Sigma_{w_j}$ be a fundamental domain for the action $\Lambda_{w_j}\c \Sigma_{w_j}$, and let $Y_j\subset \Sigma_{w_j}$ be a fundamental domain for the action $\Gamma_{w_j}\c \Sigma_{w_j}$.
By Lemma \ref{lem-small-coup} (i), the value $m(Y_j)$ is independent of $j$.
Using Condition $(\mathsf{V})_i$, for each $j$, we obtain the inequality
\[m(X_j)\geq (c_i|\ker \rho_{w_j}|)^{-1}m(Y_j),\]
where $c_i$ is the constant in Condition $(\mathsf{V})_i$.
We then have
\[m(X)=\sum_{j=1}^{|K|}m(X_j)\geq c_i^{-1}m(Y_1)\sum_{j=1}^{|K|}|\ker \rho_{w_j}|^{-1}\geq (c_iC)^{-1}m(Y_1)|K|,\]
where the second inequality follows from Condition $(\mathsf{B})$.
It follows that $|K|$ is finite.
Assertion (ii) is proved.

Assertion (iii) follows from assertions (i) and (ii).
We prove assertion (iv).
In the proof of assertion (ii), we obtain the inequality
\[m(X)\geq c_i^{-1}m(Y_1)\sum_{j=1}^{|K|}|\ker \rho_{w_j}|^{-1}\]
without Condition $(\mathsf{B})$.
Assume $|K|$ to be infinite.
Since $m(X)$ is finite, the sum in the right hand side is convergent.
There thus exists an edge $e\in E(T)$ with $|\ker\rho_u|<|\ker \rho_v|$, where $u, v\in V(T)$ are the vertices with $\partial e=\{ u, v \}$.
Since $\ker \rho_v$ is contained in $\Lambda_e$ and in $\Lambda_u$ by Lemma \ref{lem-support} (iii), $\rho_u(\ker \rho_v)$ is a non-trivial finite normal subgroup of $\rho_u(\Lambda_e)$.
This contradicts either Condition $(\mathsf{F})_1$ or $(\mathsf{F})_2$.  
\end{proof}


\section{Orbit equivalence rigidity}\label{sec-oer}

An ergodic f.f.m.p.\ action $\Gamma \c (X, \mu)$ of a discrete group $\Gamma$ is said to be superrigid if any ergodic f.f.m.p.\ action $\Lambda \c (Y, \nu)$ of a discrete group $\Lambda$ which is (W)OE to it is (virtually) conjugate to it.
In particular, $\Gamma$ and $\Lambda$ are (virtually) isomorphic if this is the case. 
In this section, we provide sufficient conditions for an ergodic f.f.m.p.\ action $\Gamma \c (X, \mu)$ to be superrigid when $\Gamma$ is an amalgamated free product $\Gamma_1\ast_A\Gamma_2$ such that $\Gamma_1$ and $\Gamma_2$ are coupling rigid.
These conditions involve ergodicity of the actions of finite index subgroups of $\Gamma_1$, $\Gamma_2$ and $A$.
Let us collect below the assumption on $\Gamma =\Gamma_1\ast_A\Gamma_2$ for which orbit equivalence rigidity will be proved.

\medskip

\noindent {\bf Assumption $(\dagger)$}\label{ass-dagger}: Let $\Gamma_1$ and $\Gamma_2$ be discrete groups and $A_i$ a subgroup of $\Gamma_i$ for $i=1, 2$. 
Let $\phi \colon A_1\rightarrow A_2$ be an isomorphism and set $\Gamma =\langle\, \Gamma_1, \Gamma_2\mid A_1\simeq_{\phi}A_2\,\rangle$.
We denote by $A$ the subgroup of $\Gamma$ corresponding to $A_1\simeq_{\phi}A_2$.
Let $T$ be the Bass-Serre tree associated with the decomposition of $\Gamma$ and $\imath \colon \Gamma \rightarrow \aut^*(T)$ the homomorphism arising from the action $\Gamma \c T$.
We assume that for each $i=1, 2$,
\begin{enumerate}
\item[(a)] $|A_i|=\infty$, $[\Gamma_i: A_i]=\infty$ and ${\rm LQN}_{\Gamma_i}(A_i)=A_i$;
\item[(b)] $\Gamma_i$ is coupling rigid with respect to $(G_i, \pi_i)$, where $G_i$ is a standard Borel group and $\pi_i\colon \Gamma_i\rightarrow G_i$ is an injective homomorphism;
\item[(c)] either $G_i$ is countable or $\Gamma_i$ is a lattice in a non-compact connected simple Lie group $G^0_i$ with its center trivial and its real rank at least two, we have $G_i=\aut(G_i^0)$, and $\pi_i$ is the inclusion of $\Gamma_i$ into $G_i$; and
\item[(d)] $\Gamma$ is coupling rigid with respect to $(\aut^*(T), \imath)$.
\end{enumerate}

\begin{rem}\label{rem-dagger}
If $\Gamma =\Gamma_1\ast_A\Gamma_2$ is the group in Assumption $(\star)$, then it satisfies conditions (a) and (d) in Assumption $(\dagger)$ by Theorem \ref{thm-coup-tree}.
\end{rem}

We prove superrigidity of certain actions of $\Gamma$ in Assumption $(\dagger)$.
Corollary \ref{cor-oer} below states the result in terms of orbit equivalence, and it implies Theorem \ref{thm-oe-rigidity}.

\begin{thm}\label{thm-oer}
On Assumption $(\dagger)$, let $\Lambda$ be a discrete group and $(\Sigma, m)$ a coupling of $\Gamma$ and $\Lambda$.
Choose fundamental domains $X, Y\subset \Sigma$ for the actions of $\Lambda$ and $\Gamma$ on $\Sigma$, respectively.
Suppose that the two associated actions $\Gamma \c X$ and $\Lambda \c Y$ satisfy the following two conditions:
\begin{enumerate}
\item[(1)] $m(X)\geq m(Y)$; and
\item[(2)] Either the action $A\c X$ is aperiodic or both of the actions $\Gamma_1\c X$ and $\Gamma_2\c X$ are aperiodic, the action $A\c X$ is ergodic, and $A$ is ICC.
\end{enumerate}
Then $m(X)=m(Y)$, and there exist an isomorphism $\rho_0\colon \Lambda \rightarrow \Gamma$ and an almost $(\Gamma \times \Lambda)$-equivariant Borel map $\Phi_0\colon \Sigma \rightarrow (\Gamma, {\rm id}, \rho_0)$.
\end{thm}

\begin{proof}
Theorem \ref{thm-furman-rep} shows that there exist a homomorphism $\rho \colon \Lambda \rightarrow \aut^*(T)$ and an almost $(\Gamma \times \Lambda)$-equivariant Borel map $\Phi \colon \Sigma \rightarrow (\aut^*(T), \imath, \rho)$.
As discussed in Section \ref{subsec-circ}, we may assume that $\Phi^{-1}(\aut(T))$ has positive measure.
We put
\[\Sigma_+=\Phi^{-1}(\aut(T)),\quad \Lambda_+=\rho^{-1}(\rho(\Lambda)\cap \aut(T)).\]
It then follows that $\Lambda \Sigma_+=\Sigma$ and that $\Sigma_+$ is a coupling of $\Gamma$ and $\Lambda_+$.
Let $X\subset \Sigma_+$ be a fundamental domain for the action $\Lambda_+\c \Sigma_+$.
Note that $X$ is also a fundamental domain for the action $\Lambda \c \Sigma$.
Let $Y_+\subset \Sigma_+$ be a fundamental domain for the action $\Gamma \c \Sigma_+$.
If $[\Lambda :\Lambda_+]=2$, then we pick $\lambda \in \Lambda \setminus \Lambda_+$ and put $Y=Y_+\sqcup \lambda Y_+$, which is a fundamental domain for the action $\Gamma \c \Sigma$.
If $\Lambda =\Lambda_+$, then we have $\Sigma =\Sigma_+$ and put $Y=Y_+$.
In Claim \ref{claim-equalities} below, we will show that the equality $\Lambda =\Lambda_+$ always holds.

We denote by $\Phi_+$ the restriction of $\Phi$ to $\Sigma_+$, which is an almost $(\Gamma \times \Lambda_+)$-equivariant Borel map into $(\aut(T), \imath, \rho)$. 
Let $v_i\in V_i(T)=\Gamma /\Gamma_i$ be the vertex corresponding to the coset containing the neutral element for each $i=1, 2$, and let $e\in E(T)$ be the edge connecting $v_1$ and $v_2$.
We put
\[\Sigma_i=\Phi_+^{-1}(\stab(v_i)),\quad \Sigma_e=\Phi_+^{-1}(\stab(e)),\quad \Lambda_i=\rho^{-1}(\rho(\Lambda)\cap \stab(v_i))\]
for each $i=1, 2$.
Recall that $\Sigma_i$ is a coupling of $\Gamma_i$ and $\Lambda_i$ by Lemma \ref{lem-small-coup}.
Let us furthermore put
\[\bar{\Gamma}_i=\pi_i(\Gamma_i),\quad \bar{A}_i=\pi_i(A_i)\]
for each $i=1, 2$.
By condition (b) in Assumption $(\dagger)$, for each $i$, there exist a homomorphism $\rho_i\colon \Lambda_i\rightarrow G_i$ with $\ker \rho_i$ finite and an almost $(\Gamma_i\times \Lambda_i)$-equivariant Borel map $\Phi_i\colon \Sigma_i\rightarrow (G_i, \pi_i, \rho_i)$.
It follows from Corollary \ref{cor-moore} that the support of the measure $(\Phi_i)_*(m|_{\Sigma_i})$, denoted by $\Phi_i(\Sigma_i)$, is countable.
By replacing $\rho_i$ and $\Phi_i$ as in Lemma \ref{lem-image-atomic}, we may assume that for each $i$, the support of the measure $(\Phi_i)_*(m|_{\Sigma_e})$, denoted by $\Phi_i(\Sigma_e)$, contains the neutral element of $G_i$ and that the inclusions $\Phi_i(\Sigma_i)\subset \comm_{G_i}(\bar{\Gamma}_i)$ and $\Phi_i(\Sigma_e)\subset \comm_{G_i}(\bar{A}_i)$ hold.
We prove the following two claims by using conditions (1) and (2).

\begin{claim}\label{claim-equalities}
The equalities
\[\Lambda =\Lambda_+,\quad Y=Y_+,\quad \Sigma =\Sigma_+, \quad m(X)=m(Y)\]
hold (up to null sets).
For each $i=1, 2$, the homomorphism $\rho_i\colon \Lambda_i\rightarrow G_i$ is injective, and we have $\rho_i(\Lambda_i)=\bar{\Gamma}_i$ and $\rho_i(\Lambda_e)=\bar{A}_i$.
For each $i=1, 2$, we also have $\Phi_i(\Sigma_i)=\bar{\Gamma}_i$ and $\Phi_i(\Sigma_e)=\bar{A}_i$.
\end{claim}

\begin{proof}
Fix $i=1, 2$.
Since the action $\Gamma_i\c \Sigma_i/\Lambda_i$ is aperiodic, so is the action $\bar{\Gamma}_i\c \Phi_i(\Sigma_i)/\rho_i(\Lambda_i)$.
We thus have $\bar{\Gamma}_i\subset \rho_i(\Lambda_i)$ and $\Phi_i(\Sigma_i)=\rho_i(\Lambda_i)$ because $\Phi_i(\Sigma_i)$ contains the neutral element of $G_i$. 
By Lemma \ref{lem-trans-erg}, there exists a fundamental domain $X\subset \Sigma$ for the action $\Lambda \c \Sigma$ so that $X\subset \Sigma_i$ because the action $\Gamma_i\c \Sigma /\Lambda \simeq \Sigma_+/\Lambda_+$ is ergodic.
By Lemma \ref{lem-small-coup} (i), we may assume $Y_+\subset \Sigma_i$.
The equality
\[m(X)/m(Y_+)=|\ker \pi_i|/([\rho_i(\Lambda_i): \bar{\Gamma}_i]|\ker \rho_i|)\]
thus holds.
It follows from $m(X)\geq m(Y)\geq m(Y_+)$ and $|\ker \pi_i|=1$ that we have $\rho_i(\Lambda_i)=\bar{\Gamma}_i$ and $m(X)=m(Y)=m(Y_+)$ and that $\rho_i$ is injective.
Since we have $\Phi_i(\Sigma_e)\subset \comm_{\bar{\Gamma}_i}(\bar{A}_i)$ and ${\rm LQN}_{\bar{\Gamma}_i}(\bar{A}_i)=\bar{A}_i=\comm_{\bar{\Gamma}_i}(\bar{A}_i)$, the equality $\Phi_i(\Sigma_e)=\bar{A}_i$ holds, and thus $\rho_i(\Lambda_e)\subset \bar{A}_i$.
By Lemma \ref{lem-support} (ii), we have $\Phi_i^{-1}(\bar{A}_i)\subset \Sigma_e$.
It turns out that for any $a\in \bar{A}_i$, the element $\rho_i^{-1}(a)$ of $\Lambda_i$ fixes the edge $e$.
The equality $\rho_i(\Lambda_e)=\bar{A}_i$ follows.
\end{proof}

\begin{claim}
We define a Borel map $\Psi \colon \Sigma_e\rightarrow \Lambda_e$ by the formula
\[\Psi(x)=\rho_1^{-1}(\Phi_1(x))\rho_2^{-1}(\Phi_2(x))^{-1},\quad x\in \Sigma_e.\]
Then $\Psi$ is essentially constant.
\end{claim}

\begin{proof}
By the definition of $\Psi$, the equalities
\[\Psi(\gamma x)=\rho_1^{-1}\circ \pi_1(\gamma)\Psi(x)\rho_2^{-1}\circ \pi_2(\gamma)^{-1},\qquad \Psi(\lambda x)=\Psi(x)\]
hold for any $\gamma \in \Gamma_e$, $\lambda \in \Lambda_e$ and a.e.\ $x\in \Sigma_e$.
The map $\Psi$ thus induces an essentially $\Gamma_e$-equivariant Borel map from $\Sigma_e/\Lambda_e$ into $\Lambda_e$, where the action $\Gamma_e\c \Lambda_e$ is defined by the formula
\[\gamma \cdot \lambda =\rho_1^{-1}\circ \pi_1(\gamma)\lambda \rho_2^{-1}\circ \pi_2(\gamma)^{-1}\]
for $\gamma \in \Gamma_e$ and $\lambda \in \Lambda_e$.
Choose $g\in \Lambda_e$ such that $\Psi^{-1}(\{ g\})$ has positive measure. 
Since $\Sigma_e/\Lambda_e$ is equipped with a finite $\Gamma_e$-invariant measure, the orbit $\Gamma_e\cdot g$ consists of at most finitely many elements.
It follows from the ergodicity of the action $\Gamma_e\c \Sigma_e/\Lambda_e$ that the image of $\Psi$ is essentially contained in the orbit $\Gamma_e\cdot g$.
Condition (2) implies that either the action $\Gamma_e\c \Sigma_e/\Lambda_e$ is aperiodic or $\Lambda_e$ is ICC.
If the former condition is fulfilled, then $\Psi$ is constant.
Suppose the latter condition.
For any two elements $g_1$, $g_2$ in the orbit $\Gamma_e\cdot g$, there exists a finite index subgroup $\Delta$ of $\Gamma_e$ with $\rho_1^{-1}\circ \pi_1(\gamma)g_t \rho_2^{-1}\circ \pi_2(\gamma)^{-1}=g_t$ for any $\gamma \in \Delta$ and any $t=1, 2$.
We then have
\begin{align*}
g_1g_2^{-1}&=(\rho_1^{-1}\circ \pi_1(\gamma)g_1\rho_2^{-1}\circ \pi_2(\gamma)^{-1})(\rho_1^{-1}\circ \pi_1(\gamma)g_2\rho_2^{-1}\circ \pi_2(\gamma)^{-1})^{-1}\\
&=\rho_1^{-1}\circ \pi_1(\gamma)g_1g_2^{-1}\rho_1^{-1}\circ \pi_1(\gamma)^{-1}
\end{align*}
for any $\gamma \in \Delta$.
Since $\Lambda_e$ is ICC, this equality implies that $g_1g_2^{-1}$ is neutral. 
The map $\Psi$ is therefore constant.
\end{proof}

Let $\lambda_0\in \Lambda_e$ be the value of the constant map $\Psi$.
The equality $\rho_2(\lambda_0)\Phi_2(x)=\rho_2\circ \rho_1^{-1}(\Phi_1(x))$ then holds for a.e.\ $x\in \Sigma_e$.
We define an isomorphism $\tilde{\pi}_2\colon \Gamma_2\rightarrow \bar{\Gamma}_2$ and an almost $(\Gamma_2\times \Lambda_2)$-equivariant Borel map $\tilde{\Phi}_2\colon \Sigma_2\rightarrow (\bar{\Gamma}_2, \tilde{\pi}_2, \rho_2)$ by the formulas
\[\tilde{\pi}_2(\gamma)=\rho_2(\lambda_0)\pi_2(\gamma)\rho_2(\lambda_0)^{-1},\qquad \tilde{\Phi}_2(x)=\rho_2(\lambda_0)\Phi_2(x)\]
for $\gamma \in \Gamma_2$ and $x\in \Sigma_2$.
The equality $\rho_1^{-1}\circ \Phi_1(x)=\rho_2^{-1}\circ \tilde{\Phi}_2(x)$ then holds for a.e.\ $x\in \Sigma_e$.
For each $i=1, 2$, let $e_i$ denote the neutral element of $G_i$.
Since $\Phi_i^{-1}(\{ e_i\})\subset \Sigma_e$ for each $i=1, 2$ by Lemma \ref{lem-support} (ii), the equality $\Phi_1^{-1}(\{ e_1\})=\tilde{\Phi}_2^{-1}(\{ e_2\})$ holds.
We denote by $E$ the subset $\Phi_1^{-1}(\{ e_1\})=\tilde{\Phi}_2^{-1}(\{ e_2\})$ of $\Sigma_e$.
This $E$ is a fundamental domain for the action $\Gamma \c \Sigma$ and is also a fundamental domain for the action $\Lambda \c \Sigma$ by ergodicity of the action $A\c \Sigma /\Lambda$.

The equality $\rho_1^{-1}\circ \pi_1(\gamma)=\rho_2^{-1}\circ \tilde{\pi}_2(\gamma)$ holds for any $\gamma \in \Gamma_e$ because
\begin{align*}
\tilde{\pi}_2(\gamma)\tilde{\Phi}_2(x)&=\tilde{\Phi}_2(\gamma x)=\rho_2\circ \rho_1^{-1}(\Phi_1(\gamma x))=\rho_2\circ \rho_1^{-1}(\pi_1(\gamma)\Phi_1(x))\\
&=\rho_2\circ \rho_1^{-1}\circ \pi_1(\gamma)\rho_2\circ \rho_1^{-1}(\Phi_1(x))=\rho_2\circ \rho_1^{-1}\circ \pi_1(\gamma)\tilde{\Phi}_2(x)
\end{align*}   
for any $\gamma \in \Gamma_e$ and a.e.\ $x\in \Sigma_e$.
We define a homomorphism $F\colon \Gamma \rightarrow \Lambda$ by
\[F(\gamma)=
\begin{cases}
\rho_1^{-1}\circ \pi_1(\gamma) & \textrm{if}\ \gamma \in \Gamma_1,\\
\rho_2^{-1}\circ \tilde{\pi}_2(\gamma) & \textrm{if}\ \gamma \in \Gamma_2.
\end{cases}
\]
The map $F$ is well-defined and is an isomorphism because $\Lambda$ is decomposed as the amalgamated free product $\Lambda_1\ast_{\Lambda_e}\Lambda_2$ by Corollary \ref{cor-factor-me}.
Pick an arbitrary element $\gamma =\gamma_1^1\gamma_2^1\gamma_1^2\gamma_2^2\cdots \gamma_1^n\gamma_2^n$ of $\Gamma$ with $\gamma_i^j\in \Gamma_i$ for each $j$ and each $i=1, 2$.
The equalities
\begin{align*}
(\gamma_1^j, \rho_1^{-1}\circ \pi_1(\gamma_1^j))\Phi_1^{-1}(\{ e_1\})&=\Phi_1^{-1}(\{ e_1\}),\\
(\gamma_2^j, \rho_2^{-1}\circ \tilde{\pi}_2(\gamma_2^j))\tilde{\Phi}_2^{-1}(\{ e_2\})&=\tilde{\Phi}_2^{-1}(\{ e_2\})
\end{align*}
for each $j$ imply the equality $(\gamma, F(\gamma))E=E$ up to null sets.
We then obtain an almost $(\Gamma \times \Lambda)$-equivariant Borel map from $\Sigma$ into $(\Gamma, {\rm id}, \rho_0)$ with $\rho_0=F^{-1}$.
\end{proof}

\begin{cor}\label{cor-oer}
On Assumption $(\dagger)$, let $\Lambda$ be a discrete group and suppose that two ergodic f.f.m.p.\ actions $\Gamma \c (X, \mu)$ and $\Lambda \c (Y, \nu)$ are WOE. 
We assume the following two conditions:
\begin{enumerate}
\item[(1)] Let $E\subset X$ and $F\subset Y$ be Borel subsets of positive measure on which there exists a Borel isomorphism which gives the WOE between the actions $\Gamma \c (X, \mu)$ and $\Lambda \c (Y, \nu)$.
Then $\mu(E)/\mu(X)\leq \nu(F)/\nu(Y)$; and
\item[(2)] Either the action $A\c (X, \mu)$ is aperiodic or both of the actions $\Gamma_1\c (X, \mu)$ and $\Gamma_2\c (X, \mu)$ are aperiodic, the action $A\c (X, \mu)$ is ergodic, and $A$ is ICC.
\end{enumerate}
Then $\mu(E)/\mu(X)=\nu(F)/\nu(Y)$, and the cocycle $\alpha \colon \Gamma \times X\rightarrow \Lambda$ associated with the WOE is cohomologous to the cocycle arising from an isomorphism from $\Gamma$ onto $\Lambda$.
In particular, the two actions $\Gamma \c (X, \mu)$ and $\Lambda \c (Y, \nu)$ are conjugate.
\end{cor}

If some assumptions in condition (2) are dropped, then we can construct an OE between two ergodic f.f.m.p.\ actions of $\Gamma$ such that the associated cocycle is not cohomologous to a constant cocycle, i.e., a cocycle arising from a homomorphism from $\Gamma$ into itself.
We refer to Section \ref{sec-twist} for details.

\begin{rem}
We discuss examples of groups satisfying Assumption $(\dagger)$ in Sections \ref{sec-ex} and \ref{sec-mis-ex}.
All of them are amalgamated free products of higher rank lattices.
In this paper, we do not present examples such that $G_i$ is countable, that is, examples satisfying the former condition in condition (c) of Assumption $(\dagger)$.

As already mentioned in the end of Section \ref{subsec-mecr}, the mapping class group of a non-exceptional compact orientable surface $S$ is coupling rigid with respect to the automorphism group of the complex of curves for $S$.
The latter group is known to be virtually isomorphic to the mapping class group.
In \cite{kida-exama}, we show coupling rigidity of an amalgamated free product of such mapping class groups with respect to the automorphism group of the associated Bass-Serre tree.
We then apply Corollary \ref{cor-oer} and obtain OE rigidity results for such groups.
\end{rem}



\section{Measure equivalence rigidity}\label{sec-mer}

In Section \ref{sec-oer}, we deduce rigidity when ergodicity is imposed on actions of some subgroups of the group $\Gamma =\Gamma_1\ast_A\Gamma_2$.
In this section, assuming a strong restriction to the inclusions $A<\Gamma_1$ and $A<\Gamma_2$ stated below, we deduce rigidity without assuming such ergodicity conditions on actions of any proper subgroup of $\Gamma$.
This allows us to construct ME rigid groups, which are presented in Section \ref{sec-ex}.

\medskip

\noindent {\bf Assumption $(\ddagger)$}\label{ass-ddagger}: Let $\Gamma_1$ and $\Gamma_2$ be discrete groups and $A_i$ a subgroup of $\Gamma_i$ for $i=1, 2$. 
Let $\phi \colon A_1\rightarrow A_2$ be an isomorphism and set $\Gamma =\langle\, \Gamma_1, \Gamma_2\mid A_1\simeq_{\phi}A_2\,\rangle$.
We denote by $A$ the subgroup of $\Gamma$ corresponding to $A_1\simeq_{\phi}A_2$.
Let $T$ be the Bass-Serre tree associated with the decomposition of $\Gamma$ and $\imath \colon \Gamma \rightarrow \aut^*(T)$ the homomorphism arising from the action $\Gamma \c T$.
We assume that
for each $i=1, 2$,
\begin{enumerate}
\item[(a)] $|A_i|=\infty$, $[\Gamma_i: A_i]=\infty$ and $\comm_{\Gamma_i}(A_i)=A_i$;
\item[(b)] $\Gamma_i$ is coupling rigid with respect to $(G_i, \pi_i)$, where $G_i$ is a standard Borel group and $\pi_i\colon \Gamma_i\rightarrow G_i$ is an injective homomorphism;
\item[(c)] either $G_i$ is countable or $\Gamma_i$ is a lattice in a non-compact connected simple Lie group $G^0_i$ with its center trivial and its real rank at least two, we have $G_i=\aut(G_i^0)$, and $\pi_i$ is the inclusion of $\Gamma_i$ into $G_i$; and
\item[(d)] $\Gamma$ is coupling rigid with respect to $(\aut^*(T), \imath)$.
\end{enumerate}
For each $i=1, 2$, we set
\[\bar{\Gamma}_i=\pi_i(\Gamma_i),\quad \bar{A}_i=\pi_i(A_i),\quad C_i=\comm_{G_i}(\bar{\Gamma}_i),\quad C(\bar{A}_i)=\comm_{C_i}(\bar{A}_i).\]
We furthermore suppose the following two conditions:
\begin{enumerate}
\item[(e)] There exists an isomorphism $\bar{\phi}\colon C(\bar{A}_1)\rightarrow C(\bar{A}_2)$ extending the isomorphism $\pi_2\circ \phi \circ \pi^{-1}_1\colon \bar{A}_1\rightarrow \bar{A}_2$.
\item[(f)] For each $i=1, 2$, the group $C_i$ is ICC with respect to $\bar{A}_i$.\end{enumerate}

\medskip

On Assumption $(\ddagger)$, $C_1$ and $C_2$ are countable groups.
We define the amalgamated free product
\[G=\langle\, C_1, C_2\mid C(\bar{A}_1)\simeq_{\bar{\phi}}C(\bar{A}_2)\,\rangle.\]
Let $C(\bar{A})$ denote the identified subgroup $C(\bar{A}_1)\simeq_{\bar{\phi}}C(\bar{A}_2)$ of $G$.
There exists an injective homomorphism $\pi \colon \Gamma \rightarrow G$ with $\pi =\pi_i$ on $\Gamma_i$ for each $i=1, 2$ because we have $\bar{\Gamma}_i\cap C(\bar{A}_i)=\bar{A}_i$ for each $i=1, 2$ by condition (a).

\begin{rem}
In general, if $M$ is a group and $N$ is a subgroup of $M$ with ${\rm LQN}_M(N)=N$, then we have $\comm_M(N)=N$.
It follows that if $\Gamma =\Gamma_1\ast_A\Gamma_2$ is the group in Assumption $(\star)$, then it satisfies conditions (a) in Assumption $(\ddagger)$.
It also satisfies condition (d) in Assumption $(\ddagger)$ by Theorem \ref{thm-coup-tree}.

Condition (f) in Assumption $(\ddagger)$ implies the following two assertions:
\begin{itemize}
\item For each $i=1, 2$, $\bar{\Gamma}_i$ is ICC with respect to $\bar{A}_i$.
In particular, $\bar{A}_i$ is ICC.
\item For each $i=1, 2$, any isomorphism between finite index subgroups of $\bar{\Gamma}_i$ that is the identity on some finite index subgroup of $\bar{A}_i$ is the restriction of the identity.
\end{itemize}
Condition (f) in Assumption $(\ddagger)$ also implies uniqueness of the extension $\bar{\phi}$ of the isomorphism $\pi_2\circ \phi \circ \pi^{-1}_1$.
\end{rem}

On Assumption $(\ddagger)$, it will be shown that $\comm(\Gamma)$ is countable and that $\Gamma$ is coupling rigid with respect to $\comm(\Gamma)$.
As a first step for the proof, we construct an almost $(\Gamma \times \Gamma)$-equivariant Borel map from a self-coupling of $\Gamma$ into the $(\Gamma \times \Gamma)$-space $(G, \pi, \pi)$ when a mild condition is imposed on the coupling.
The following two lemmas are preliminaries for this first step.

\begin{lem}\label{lem-self-image}
Let $\Gamma$ be a discrete group, $G$ a standard Borel group and $\pi \colon \Gamma \rightarrow G$ a homomorphism with $\ker \pi$ finite. 
Suppose that we have a self-coupling $(\Sigma, m)$ of $\Gamma$ and an almost $(\Gamma \times \Gamma)$-equivariant Borel map $\Phi \colon \Sigma \rightarrow (G, \pi, \pi)$ such that the measure $\Phi_*m$ on $G$ is atomic.
Then the support of $\Phi_*m$ is contained in $\comm_G(\pi(\Gamma))$. 
\end{lem}

\begin{proof}
Note that any element of $G$ in the support of $\Phi_*m$ has finite measure with respect to $\Phi_*m$ because $\ker \pi$ is finite.
Pick $g\in G$ in the support of $\Phi_*m$.
Let $\cal{O}$ be the orbit for the action $\Gamma \times \Gamma \c (G, \pi, \pi)$ containing $g$.
This action defines a coupling with respect to the measure $\Phi_*m$.
It follows that the set of all left (resp.\ right) cosets of $\pi(\Gamma)$ contained in $\cal{O}$ consists of at most finitely many elements.
Since $\pi(\Gamma)$ acts on this finite set, there exists a finite index subgroup $\Gamma'$ of $\pi(\Gamma)$ contained in $g\pi(\Gamma)g^{-1}$ and $g^{-1}\pi(\Gamma)g$.
We therefore have $g\in \comm_G(\pi(\Gamma))$.
\end{proof}

\begin{lem}\label{lem-icc-strong}
On Assumption $(\ddagger)$, let $\Gamma_1'$ and $\Gamma_2'$ be finite index subgroups of $\bar{\Gamma}_1$ and $\bar{\Gamma}_2$, respectively.
If $g\in G$ satisfies the equality $\gamma g\gamma^{-1}=g$ for any $\gamma \in \Gamma_1'\cup \Gamma_2'$, then $g$ is neutral.
\end{lem}

\begin{proof}
Assuming that $g$ is non-neutral, we deduce a contradiction.
We write $g$ as a normal form $g=c_1c_2\cdots c_n$ with respect to the decomposition $G=C_1\ast_{C(\bar{A})}C_2$, which satisfies the following four conditions (see IV.2.6 in \cite{ls}):
\begin{itemize}
\item Each $c_i$ belongs to either $C_1$ or $C_2$.
\item Any two successive elements $c_i$ and $c_{i+1}$ belong to distinct factors $C_1$ and $C_2$, respectively.
\item When $n>1$, any $c_i$ does not belong to $C(\bar{A})$.
\item When $n=1$, $c_1$ is non-neutral.
\end{itemize}
Conversely, it is known that any product of elements of $C_1$ and $C_2$ satisfying these four conditions is non-neutral in $G$.
When $n=1$, it follows from condition (b) in Assumption $(\ddagger)$ that $\Gamma_1$ and $\Gamma_2$ are ICC, and thus $g$ is the neutral element.
This is a contradiction.

Suppose that $n$ is bigger than one.
We first consider the case where both $c_1$ and $c_n$ are in $C_1\setminus C(\bar{A})$.
For each $\gamma \in \Gamma_2'\setminus \bar{A}_2$, the product
\[\gamma g\gamma^{-1}g^{-1}=\gamma c_1\cdots c_n\gamma^{-1}c_n^{-1}\cdots c_1^{-1}\]
is then a normal form because we have $\gamma \in C_2\setminus C(\bar{A})$.
This product is thus non-neutral in $G$.
This is a contradiction.
Similarly, we can deduce a contradiction in the case where both $c_1$ and $c_n$ belong to $C_2\setminus C(\bar{A})$.

We next assume $c_1\in C_1\setminus C(\bar{A})$ and $c_n\in C_2\setminus C(\bar{A})$.
For any $\gamma, \gamma'\in \Gamma_1'$, if both $\gamma c_1$ and $\gamma'c_1$ belong to $C(\bar{A})$, then $\gamma'\gamma^{-1}=(\gamma'c_1)(\gamma c_1)^{-1}$ also belongs to $C(\bar{A})$.
The two elements $\gamma$ and $\gamma'$ are thus in the same right coset of $\bar{A}_1\cap \Gamma_1'=C(\bar{A})\cap \Gamma_1'$ in $\Gamma_1'$.
It follows that there exists $\gamma \in \Gamma_1'\setminus \bar{A}_1$ with $\gamma c_1\not\in C(\bar{A})$ because the index of $\bar{A}_1\cap \Gamma_1'$ in $\Gamma_1'$ is infinite.
The product
\[\gamma g\gamma^{-1}g^{-1}=(\gamma c_1)c_2\cdots c_n\gamma^{-1}c_n^{-1}\cdots c_1^{-1}\]
is then a normal form and is non-neutral in $G$.
This is a contradiction.
The proof for the case where $c_1\in C_2\setminus C(\bar{A})$ and $c_n\in C_1\setminus C(\bar{A})$ is obtained analogously.
\end{proof}

\begin{thm}\label{thm-mer-ddagger}
On Assumption $(\ddagger)$, let $(\Sigma, m)$ be a self-coupling of $\Gamma$. 
Suppose that we have an almost $(\Gamma \times \Gamma)$-equivariant Borel map $\Phi \colon \Sigma \rightarrow (\aut(T), \imath, \imath)$.
Then there exists an essentially unique, almost $(\Gamma \times \Gamma)$-equivariant Borel map $\Phi_0\colon \Sigma \rightarrow (G, \pi, \pi)$.
\end{thm}

\begin{proof}
We use the notation employed in Section \ref{subsec-pat}.
Namely, for each $s\in V(T)\cup E(T)$, we set
\[\Sigma_s=\Phi^{-1}(\stab(s)),\quad \Gamma_s=\imath^{-1}(\imath(\Gamma)\cap \stab(s))\]
and denote by $v_i\in V_i(T)=\Gamma /\Gamma_i$ the vertex corresponding to the coset containing the neutral element for $i=1, 2$.
We put $\Sigma_i=\Sigma_{v_i}$ for $i=1, 2$.
Recall that $\Sigma_s$ is a self-coupling of $\Gamma_s$ for each $s\in V(T)\cup E(T)$ by Lemma \ref{lem-small-coup}.
For $e\in E(T)$, we set $\bar{\Gamma}_e=\pi(\Gamma_e)$.
For each $i=1, 2$ and each $v\in V_i(T)$, choosing $\gamma \in \Gamma$ with $\gamma v_i=v$, we set $C_v=\pi(\gamma)C_i\pi(\gamma)^{-1}$.
This definition of the subgroup $C_v$ of $G$ does not depend on the choice of $\gamma$.
The proof of the theorem consists of the following four claims.

\begin{claim}
For each $v\in V(T)$, there exists an essentially unique, almost $(\Gamma_v\times \Gamma_v)$-equivariant Borel map $\Phi_v\colon \Sigma_v\rightarrow (C_v, \pi, \pi)$.
\end{claim}

\begin{proof}
Choose $i\in \{ 1, 2\}$ and $\gamma \in \Gamma$ with $C_v=\pi(\gamma)C_i\pi(\gamma)^{-1}$.
Since $\Gamma_i$ is coupling rigid with respect to $(G_i, \pi_i)$, there exists an essentially unique, almost $(\Gamma_i\times \Gamma_i)$-equivariant Borel map $\Phi_i\colon \Sigma_i\rightarrow (G_i, \pi_i, \pi_i)$.
It follows from Corollary \ref{cor-moore} that the measure $(\Phi_i)_*(m|_{\Sigma_i})$ on $G_i$ is atomic.
By Lemma \ref{lem-self-image}, its support $\Phi_i(\Sigma_i)$ is contained in $C_i$. 
The map $\Phi_v$ is then defined so that the following diagram commutes, where the map from $\Sigma_v$ into $\Sigma_i$ is defined by the action of the element $(\gamma^{-1}, \gamma^{-1})$ of $\Gamma \times \Gamma$:
\[
\begin{CD}
\Sigma_v @>{\Phi_v}>> C_v\\
@V{(\gamma^{-1}, \gamma^{-1})}VV @AA{\ad \pi(\gamma)}A\\
\Sigma_i @>{\Phi_i}>> C_i\\
\end{CD}
\]
The map $\Phi_v$ is then almost $(\Gamma_v\times \Gamma_v)$-equivariant.
Essential uniqueness of $\Phi_v$ follows from Lemma \ref{lem-uni-quo} (i).
\end{proof}

\begin{claim}\label{claim-coincide}
Let $u, v\in V(T)$ be any two adjacent vertices and denote by $e\in E(T)$ the edge connecting $u$ and $v$.
Then the two maps $\Phi_u\colon \Sigma_u\rightarrow C_u$ and $\Phi_v\colon \Sigma_v\rightarrow C_v$ essentially coincide on $\Sigma_e=\Sigma_u\cap \Sigma_v$.
\end{claim}

\begin{proof}
We define a Borel map $\Psi \colon \Sigma_e\rightarrow G$ by $\Psi(x)=\Phi_u(x)^{-1}\Phi_v(x)$ for $x\in \Sigma_e$.
The equality
\[\Psi((\gamma_1, \gamma_2)x)=\pi(\gamma_2)\Psi(x)\pi(\gamma_2)^{-1}\]
holds for any $\gamma_1, \gamma_2\in \Gamma_e$ and a.e.\ $x\in \Sigma_e$.
By Lemma \ref{lem-self-image}, for each $w\in \{ u, v\}$, the support of the measure $(\Phi_{w})_*(m|_{\Sigma_e})$ is contained in $\comm_{C_w}(\bar{\Gamma}_e)$.
By definition, $\comm_{C_u}(\bar{\Gamma}_e)$ and $\comm_{C_v}(\bar{\Gamma}_e)$ are identified in $G$ and are a conjugate of $C(\bar{A})$ in $G$.
The image of $\Psi$ is hence contained in this identified subgroup of $G$.
Condition (f) in Assumption $(\ddagger)$ implies that $\Psi(x)$ is the neutral element for a.e.\ $x\in \Sigma_e$.
\end{proof}

Let us introduce new notation in order to distinguish the two actions of $\Gamma$ on $\Sigma$.
For $\gamma \in \Gamma$, we set
\[L(\gamma)=(\gamma, 1_{\Gamma})\in \Gamma \times \Gamma,\quad R(\gamma)=(1_{\Gamma}, \gamma)\in \Gamma \times \Gamma,\] 
where $1_{\Gamma}$ is the neutral element of $\Gamma$.

\begin{claim}\label{claim-ext}
For each $v\in V(T)$, we can extend the map $\Phi_v\colon \Sigma_v\rightarrow C_v$ to a Borel map $\tilde{\Phi}_v\colon \Sigma \rightarrow G$ so that $\tilde{\Phi}_v(L(\gamma)x)=\pi(\gamma)\tilde{\Phi}_v(x)$ for any $\gamma \in \Gamma$ and a.e.\ $x\in \Sigma$.
This map $\tilde{\Phi}_v$ satisfies the equality $\tilde{\Phi}_v(R(\lambda)x)=\tilde{\Phi}_v(x)\pi(\lambda)^{-1}$ for any $\lambda \in \Gamma_v$ and a.e.\ $x\in \Sigma$.
\end{claim}

\begin{proof}
For any two vertices $u, v\in V(T)$, we put
\[\Sigma_v^u=\Phi^{-1}(\{\, f\in \aut(T)\mid f(v)=u\, \}).\]
For each $v\in V(T)$, the equality $\Sigma =\bigsqcup_{u\in V(T)}\Sigma_v^u$ holds up to null sets.

Fix any vertex $v\in V(T)$.
Choose any $\gamma_1, \gamma_2\in \Gamma$ and any Borel subsets $Z_1, Z_2\subset\Sigma_v$ of positive measure with $L(\gamma_1)Z_1=L(\gamma_2)Z_2\subset \Sigma_v^u$ for some $u\in V(T)$. 
Pick $\gamma \in \Gamma$ with $\gamma v=u$.
Both $\gamma^{-1}\gamma_1$ and $\gamma^{-1}\gamma_2$ then belong to $\Gamma_v$.
Since the map $\Phi_v\colon \Sigma_v\rightarrow (G, \pi, \pi)$ is almost $(\Gamma_v\times \Gamma_v)$-equivariant, the equality
\begin{align*}
\pi(\gamma_1)\Phi_v(x_1)=&\pi(\gamma)\pi(\gamma^{-1}\gamma_1)\Phi_v(x_1)=\pi(\gamma)\Phi_v(L(\gamma^{-1}\gamma_1)x_1)\\
=&\pi(\gamma)\Phi_v(L(\gamma^{-1}\gamma_2)x_2)=\pi(\gamma)\pi(\gamma^{-1}\gamma_2)\Phi_v(x_2)=\pi(\gamma_2)\Phi_v(x_2)
\end{align*} 
holds for a.e.\ $x_1\in Z_1$ and $x_2\in Z_2$ with $L(\gamma_1)x_1=L(\gamma_2)x_2$.
The former assertion of the claim follows.

Choose any $\lambda \in \Gamma_v$ and any Borel subset $W\subset \Sigma$ of positive measure such that $R(\lambda)W$ is contained in $\Sigma_v^u$ for some $u\in V(T)$.
Pick $\gamma \in \Gamma$ with $\gamma v=u$.
We then have $L(\gamma^{-1})R(\lambda)x\in \Sigma_v$ for a.e.\ $x\in W$.
Since $\Phi_v$ is almost $(\Gamma_v\times \Gamma_v)$-equivariant, the equality\begin{align*}
\tilde{\Phi}_v(R(\lambda)x)&=\pi(\gamma)\Phi_v(L(\gamma^{-1})R(\lambda)x)=\pi(\gamma)\Phi_v(L(\gamma^{-1})x)\pi(\lambda)^{-1}\\
&=\tilde{\Phi}_v(x)\pi(\lambda)^{-1}
\end{align*}
holds for a.e.\ $x\in W$.
The latter assertion of the claim follows.
\end{proof}

\begin{claim}
For any two vertices $u, v\in V(T)$, the two maps $\tilde{\Phi}_u, \tilde{\Phi}_v\colon \Sigma \rightarrow G$ are essentially equal.
\end{claim}

\begin{proof}
We may assume that $u$ and $v$ are adjacent vertices.
Let $e\in E(T)$ be the edge connecting $u$ and $v$.
Choose any $\gamma_1, \gamma_2\in \Gamma$ and any Borel subsets $Z_1\subset \Sigma_u$ and $Z_2\subset \Sigma_v$ of positive measure with $L(\gamma_1)Z_1=L(\gamma_2)Z_2$ and $Z_1\subset \Sigma_v^w$ for some $w\in V(T)$, where we use the same notation as in the proof of Claim \ref{claim-ext}.
Since $w$ is in the link of $u$, there exists $\gamma \in \Gamma_u$ with $\gamma v=w$.
It follows from $L(\gamma^{-1})Z_1=L(\gamma^{-1}\gamma_1^{-1}\gamma_2)Z_2\subset \Sigma_e$ that $\gamma^{-1}\gamma_1^{-1}\gamma_2$ belongs to $\Gamma_v$. 
By Claim \ref{claim-coincide}, the equality
\begin{align*}
\pi(\gamma_1)\Phi_u(x_1)&=\pi(\gamma_1\gamma)\Phi_u(L(\gamma^{-1})x_1)=\pi(\gamma_1\gamma)\Phi_v(L(\gamma^{-1})x_1)\\
&=\pi(\gamma_1\gamma)\Phi_v(L(\gamma^{-1}\gamma_1^{-1}\gamma_2)x_2)=\pi(\gamma_1\gamma)\pi(\gamma^{-1}\gamma_1^{-1}\gamma_2)\Phi_v(x_2)\\
&=\pi(\gamma_2)\Phi_v(x_2)
\end{align*}
holds for a.e.\ $x_1\in Z_1$ and $x_2\in Z_2$ with $L(\gamma_1)x_1=L(\gamma_2)x_2$.
The claim follows.
\end{proof}

We denote by $\Phi_0\colon \Sigma \rightarrow G$ the common map $\tilde{\Phi}_v$ for all $v\in V(T)$.
The map $\Phi_0\colon \Sigma \rightarrow (G, \pi, \pi)$ is almost $(\Gamma \times \Gamma)$-equivariant because $\Gamma$ is generated by $\Gamma_u$ and $\Gamma_v$ for any adjacent vertices $u, v\in V(T)$.
Essential uniqueness of $\Phi_0$ follows from Lemmas \ref{lem-uni-quo} and \ref{lem-icc-strong}.
\end{proof}

\begin{lem}\label{lem-comm-count}
On Assumption $(\ddagger)$, we have a subgroup of $\comm(\Gamma)$ of index at most two, denoted by $\comm^+(\Gamma)$, and an injective homomorphism $\bar{\pi}\colon \comm^+(\Gamma) \rightarrow G$ such that for any isomorphism $f\colon \Gamma'\rightarrow \Gamma''$ between finite index subgroups of $\Gamma$ with its equivalence class $[f]$ in $\comm^+(\Gamma)$, we have
\[\pi(f(\gamma))=\bar{\pi}([f])\pi(\gamma)\bar{\pi}([f])^{-1},\quad \forall \gamma \in \Gamma'.\]
In particular, $\comm(\Gamma)$ is countable.
\end{lem}

\begin{proof}
By Lemma \ref{lem-comm-rep}, we have an injective homomorphism $\bar{\imath}\colon \comm(\Gamma)\rightarrow \aut^*(T)$ with the equality
\[\imath(f(\gamma))=\bar{\imath}([f])\imath(\gamma)\bar{\imath}([f])^{-1},\quad \forall \gamma \in \Gamma'\]
for any isomorphism $f\colon \Gamma'\rightarrow \Gamma''$ between finite index subgroups of $\Gamma$.
We set $\comm^+(\Gamma)=\bar{\imath}^{-1}(\aut(T))$.
For each $[f]\in \comm^+(\Gamma)$, we define the self-coupling $\Sigma_f$ of $\Gamma$ as in the proof of Lemma \ref{lem-comm-rep}.
The image of the unique, almost $(\Gamma \times \Gamma)$-equivariant Borel map from $\Sigma_f$ into $(\aut^*(T), \imath, \imath)$ is then contained in $\aut(T)$.
By Theorem \ref{thm-mer-ddagger}, there exists an essentially unique, almost $(\Gamma \times \Gamma)$-equivariant Borel map from $\Sigma_f$ into $(G, \pi, \pi)$.
Following the proof of Lemma \ref{lem-comm-rep}, we obtain a homomorphism $\bar{\pi}\colon \comm^+(\Gamma)\rightarrow G$ satisfying the desired equality. 
Since $\pi \colon \Gamma \rightarrow G$ is injective, so is $\bar{\pi}$.  
\end{proof}

\begin{thm}\label{thm-ddagger-comm-rigid}
On Assumption $(\ddagger)$, the group $\Gamma$ is coupling rigid with respect to $\comm(\Gamma)$ and is thus ME rigid.
\end{thm}

\begin{proof}
We note that $\comm(\Gamma)$ is countable by Lemma \ref{lem-comm-count}.
Let $(\Sigma, m)$ be a self-coupling of $\Gamma$.
By Theorem \ref{thm-coup-tree}, there exists an almost $(\Gamma \times \Gamma)$-equivariant Borel map $\Phi \colon \Sigma \rightarrow (\aut^*(T), \imath, \imath)$.
If $\Phi^{-1}(\aut(T))$ has positive measure, then applying Theorem \ref{thm-mer-ddagger} and Lemma \ref{lem-red-comm}, we obtain an almost $(\Gamma \times \Gamma)$-equivariant Borel map from $\Phi^{-1}(\aut(T))$ into $(\comm(\Gamma), {\bf i}, {\bf i})$.

In what follows, we assume that $\Phi^{-1}(\aut(T))$ is of measure zero.
It is enough to construct an almost $(\Gamma \times \Gamma)$-equivariant Borel map from $\Sigma$ into $(\comm(\Gamma), {\bf i}, {\bf i})$ in this case.
To distinguish the two actions of $\Gamma$ on $\Sigma$, we use the following notation employed in the proof of Theorem \ref{thm-mer-ddagger}:
For each $\gamma \in \Gamma$, we set
\[L(\gamma)=(\gamma, 1_{\Gamma})\in \Gamma \times \Gamma,\quad R(\gamma)=(1_{\Gamma}, \gamma)\in \Gamma \times \Gamma,\] 
where $1_{\Gamma}$ is the neutral element of $\Gamma$.
Let us consider the two actions of $\Gamma \times \Gamma$ on $\Sigma \times \Gamma \times \Sigma$ defined by the formulas
\begin{align*}
(\gamma_1, \gamma_2)(x, \lambda, y)&=(L(\gamma_1)x, \lambda, L(\gamma_2)y),\\
(\lambda_1, \lambda_2)(x, \lambda, y)&=(R(\lambda_1)x, \lambda_1\lambda \lambda_2^{-1}, R(\lambda_2)y),
\end{align*}
respectively, for $\gamma_1, \gamma_2, \lambda, \lambda_1, \lambda_2\in \Gamma$ and $x, y\in \Sigma$.
We denote by $\Omega$ the quotient space of $\Sigma \times \Gamma \times \Sigma$ by the second action of $\Gamma \times \Gamma$.
The first action of $\Gamma \times \Gamma$ induces an action $\Gamma \times \Gamma \c \Omega$, which defines a self-coupling of $\Gamma$.
The map
\[\Sigma \times \Gamma \times \Sigma \ni (x, \lambda, y)\mapsto \Phi(x)\imath(\lambda)\Phi(y)^{-1}\in \aut^*(T)\]
induces an almost $(\Gamma \times \Gamma)$-equivariant Borel map from $\Omega$ into $(\aut^*(T), \imath, \imath)$, whose image is contained in $\aut(T)$ because for a.e.\ $z\in \Sigma$, $\Phi(z)$ belongs to $\aut^*(T)\setminus \aut(T)$.
By Theorem \ref{thm-mer-ddagger}, there exists an almost $(\Gamma \times \Gamma)$-equivariant Borel map from $\Omega$ into $(G, \pi, \pi)$.
The proof of Theorem \ref{thm-furman-rep} shows that there exist a homomorphism $\rho \colon \Gamma \rightarrow G$ with $\ker \rho$ finite and an almost $(\Gamma \times \Gamma)$-equivariant Borel map from $\Sigma$ into $(G, \pi, \rho)$.
Since $\Gamma$ is ICC by Lemma \ref{lem-icc-strong}, the kernel of $\rho$ is trivial.
It follows from Lemma \ref{lem-red-comm} that there exists an almost $(\Gamma \times \Gamma)$-equivariant Borel map from $\Sigma$ into $(\comm(\Gamma), {\bf i}, {\bf i})$.
\end{proof}

\begin{rem}\label{rem-comm}
On Assumption $(\star)$, let $\Lambda$ be a discrete group and $(\Sigma, m)$ a coupling of $\Gamma$ and $\Lambda$.
Suppose that we have a countable Borel space $C$ on which $\Gamma \times \Lambda$ acts so that each action of $\Gamma$ and $\Lambda$ is free, and an almost $(\Gamma \times \Lambda)$-equivariant Borel map $\Phi_0\colon \Sigma \rightarrow C$.
By Corollary \ref{cor-coup-tree}, there exist a homomorphism $\rho \colon \Lambda \rightarrow \aut^*(T)$ and an almost $(\Gamma \times \Lambda)$-equivariant Borel map $\Phi \colon \Sigma \rightarrow (\aut^*(T), \imath, \rho)$. 
We claim that the measure $\Phi_*m$ on $\aut^*(T)$ is atomic. 
If $\ker \iota$ is furthermore finite, then the claim implies that $\ker \rho$ is finite and that $\imath(\Gamma)$ and some conjugate of $\rho(\Lambda)$ in $\aut^*(T)$ are commensurable.

Choose $c\in C$ such that $\Phi_0^{-1}(\{ c\})$ has positive measure.
We put $X_c=\Phi_0^{-1}(\{ c\})$.
If $F$ is a fundamental domain for the action $\Lambda \c C$ containing $c$, then $\Phi_0^{-1}(F)$ is a fundamental domain for the action $\Lambda \c \Sigma$.
It follows that $X_c$ is invariant under the stabilizer $\Gamma_c$ of $c$ for the action $\Gamma \c F$, which is a finite index subgroup of $\Gamma$.
We define a Borel map $\Phi_c\colon X_c\times X_c\rightarrow \aut^*(T)$ by $\Phi_c(x, y)=\Phi(x)\Phi(y)^{-1}$ for $(x, y)\in X_c\times X_c$.
For any $\gamma \in \Gamma_c$ and a.e.\ $(x, y)\in X_c\times X_c$, the equality $\Phi_c(\gamma \cdot x, \gamma \cdot y)=\imath(\gamma)\Phi_c(x, y)\imath(\gamma)^{-1}$ holds.
By Lemma \ref{lem-conj-inv}, $\Phi_c$ is essentially constant, and its value is the neutral element of $\aut^*(T)$.
The map $\Phi$ is therefore essentially constant on $X_c$.
It follows that the measure $\Phi_*m$ on $\aut^*(T)$ is atomic.
Our claim in the last paragraph is therefore obtained.
\end{rem}



\section{Twisted actions of amalgamated free products}\label{sec-twist}

The universality of amalgamated free products enables us to find examples of them or their actions which do not satisfy a variety of rigidity discussed in Sections \ref{sec-oer} and \ref{sec-mer}.
In this section, we give a method to construct such examples by twisting the action of one factor subgroup when an ergodic f.f.m.p.\ action $\Gamma \c (X, \mu)$ of an amalgamated free product $\Gamma =\Gamma_1\ast_A\Gamma_2$ is given.
More precisely, we twist the action $\Gamma_2\c (X, \mu)$ by inner conjugation in a sense of orbit equivalence, keeping the action $\Gamma_1\c (X, \mu)$ fixed.
The original action and the new twisted action of $\Gamma$ are then OE with respect to the identity map on $X$.

\subsection{Twisted actions}\label{subsec-twist}

Let $\Gamma =\Gamma_1\ast_A\Gamma_2$ be an amalgamated free product of discrete groups.
Suppose that we have a f.f.m.p.\ action $\Gamma \c (X, \mu)$ on a standard probability space $(X, \mu)$ and a Borel map $\varphi \colon X\rightarrow A$ satisfying the following two conditions:
\begin{enumerate}
\item[(a)] The map $X\ni x\mapsto \varphi(x)^{-1}x\in X$ defines a Borel isomorphism between conull Borel subsets of $X$.
We denote by $f_{\varphi}$ this isomorphism.
\item[(b)] The following equality holds:
\[\varphi(ax)a\varphi(x)^{-1}=a,\quad \forall a\in A,\ \textrm{a.e.\ }x\in X.\]
\end{enumerate}
Let $(Y, \nu)$ be a copy of $(X, \mu)$.
We construct a f.f.m.p.\ action of $\Gamma$ on $(Y, \nu)$ by twisting the action of $\Gamma_2$ as follows: Let $I\colon X\rightarrow Y$ be the identity map.
Define an action of $\Gamma$ on $Y$ by the formulas
\[\gamma_1y=I(\gamma_1I^{-1}(y)),\quad \gamma_2y=I\circ f_{\varphi}(\gamma_2(I\circ f_{\varphi})^{-1}(y))\]
for $\gamma_1\in \Gamma_1$, $\gamma_2\in \Gamma_2$ and $y\in Y$. 
Since the equality
\[af_{\varphi}(x)=a\varphi(x)^{-1}x=\varphi(ax)^{-1}ax=f_{\varphi}(ax)\]
holds for any $a\in A$ and a.e.\ $x\in X$, this action $\Gamma \c (Y, \nu)$ is well-defined.
This action is essentially free because $\varphi$ is valued in $A$.

The map $I\colon X\rightarrow Y$ gives OE between the two actions $\Gamma \c X$ and $\Gamma \c Y$.
The associated OE cocycle $\alpha \colon \Gamma \times X\rightarrow \Gamma$ is given by
\[\alpha(\gamma_1, x)=\gamma_1,\quad \alpha(\gamma_2, x)=\varphi(\gamma_2x)\gamma_2\varphi(x)^{-1}\] 
for any $\gamma_1\in \Gamma_1$, $\gamma_2\in \Gamma_2$ and a.e.\ $x\in X$.

\subsection{Groups which are not coupling rigid}

We keep the same notation as in Section \ref{subsec-twist}.
Let $\bar{\alpha}\colon \Gamma \times X\rightarrow \comm(\Gamma)$ be the composition of $\alpha$ with the natural homomorphism ${\bf i}\colon \Gamma \rightarrow \comm(\Gamma)$.

\begin{lem}\label{lem-varphi-const}
Suppose that $\comm(\Gamma)$ is countable, $\Gamma$ is ICC with respect to $\Gamma_1$, and the action $\Gamma_2\c X$ is aperiodic.
If there exists a Borel map $\psi \colon X\rightarrow \comm(\Gamma)$ with the equality
\[\psi(\gamma x)\bar{\alpha}(\gamma, x)\psi(x)^{-1}=\gamma,\quad \forall \gamma \in \Gamma,\ \textrm{a.e.}\ x\in X,\]
then $\varphi$ is essentially constant.
\end{lem}

\begin{proof}
Since the action $\Gamma \c X$ is aperiodic, the equality $\gamma \psi(x)=\psi(\gamma x)\bar{\alpha}(\gamma, x)$ for any $\gamma \in \Gamma$ and a.e.\ $x\in X$ implies that there exists $g_0\in \comm(\Gamma)$ with $\psi(X)\subset g_0\Gamma$ for a.e.\ $x\in X$, where $\psi(X)$ denotes the support of the measure $\psi_*\mu$ on $\comm(\Gamma)$.
We may assume $g_0\in \psi(X)$. 

We claim that the equality $\psi(X)=\{ g_0\}$ holds.
Pick $\gamma \in \Gamma$ with $g_0\gamma \subset \psi(X)$.
Since the equality $\psi(\gamma_1x)=\gamma_1\psi(x)\gamma_1^{-1}$ holds for any $\gamma_1\in \Gamma_1$ and a.e.\ $x\in X$, each $g\in \psi(X)$ centralizes a finite index subgroup of $\Gamma_1$.
We thus have a finite index subgroup $\Delta$ of $\Gamma_1$ contained in the centralizers of $g_0$ and $g_0\gamma$.
For each $\gamma_1\in \Delta$, the equality $g_0\gamma =\gamma_1g_0\gamma \gamma_1^{-1}=g_0\gamma_1\gamma \gamma_1^{-1}$ implies that $\Delta$ is contained in the centralizer of $\gamma$.
Since $\Gamma$ is ICC with respect to $\Gamma_1$, the element $\gamma$ is neutral.
The claim follows.

We therefore have the equality $g_0\varphi(\gamma_2x)=\gamma_2g_0\varphi(x)\gamma_2^{-1}$ for any $\gamma_2\in \Gamma_2$ and a.e.\ $x\in X$.
Aperiodicity of the action $\Gamma_2\c X$ implies that $\varphi$ is essentially constant.
\end{proof}

We give a sufficient condition for $\Gamma$ to satisfy the ICC assumption in Lemma \ref{lem-varphi-const}.

\begin{lem}\label{lem-ama-icc}
Let $\Gamma =\Gamma_1\ast_A\Gamma_2$ be an amalgamated free product of discrete groups such that $[\Gamma_1:A]=\infty$ and $\Gamma_1$ is ICC.
Then $\Gamma$ is ICC with respect to $\Gamma_1$.
\end{lem}

\begin{proof}
Let $T$ be the Bass-Serre tree associated with the decomposition $\Gamma =\Gamma_1\ast_A\Gamma_2$.
We denote by $V_1(T)=\Gamma /\Gamma_1$ the set of vertices corresponding to left cosets of $\Gamma_1$ in $\Gamma$.
Recall that $\Gamma_1$ acts on $T$ so that $\Gamma_1$ fixes the vertex $v_1\in V_1(T)$ corresponding to the coset containing the neutral element and acts transitively on ${\rm Lk}(v_1)$, the set of vertices in the link of $v_1$ in $T$.
Since $A$ is of infinite index in $\Gamma_1$, ${\rm Lk}(v_1)$ consists of infinitely many elements.
Each orbit for the action $\Gamma_1\c V_1(T)$ hence consists of either the single vertex $v_1$ or infinitely many elements. 

If $\mu$ is a probability measure on $\Gamma$ which is invariant under conjugation by any element of $\Gamma_1$, then consider the natural quotient map $q\colon \Gamma \rightarrow \Gamma /\Gamma_1=V_1(T)$.
For any $g\in \Gamma$ and any $\gamma \in \Gamma_1$, we have $q(\gamma g\gamma^{-1})=\gamma gv_1$.
The observation on orbits for the action $\Gamma_1\c V_1(T)$ implies that the support of $\mu$ is contained in $\Gamma_1$.
Since $\Gamma_1$ is ICC, $\mu$ is the Dirac measure on the neutral element.
\end{proof}

If $A$ is of infinite index in $\Gamma_2$ and is not ICC, then using co-induced actions introduced in \cite{gab-survey}, we can construct an ergodic f.f.m.p action $\Gamma \c X$ and a Borel map $\varphi \colon X\rightarrow A$ such that
\begin{itemize}
\item the action $\Gamma_2\c X$ is aperiodic; and
\item $\varphi$ fulfills conditions (a) and (b) stated in Section \ref{subsec-twist} and is not essentially constant.
\end{itemize}
Lemma \ref{lem-varphi-const} then implies the following:

\begin{prop}\label{prop-not-coup}
Let $\Gamma =\Gamma_1\ast_A\Gamma_2$ be an amalgamated free product of discrete groups such that $[\Gamma_1: A]=\infty$, $[\Gamma_2: A]=\infty$, $\Gamma_1$ is ICC, and $A$ is not ICC.
Then $\Gamma$ is not coupling rigid with respect to $\comm(\Gamma)$.
\end{prop}

\section{Examples of ME rigid groups}\label{sec-ex}

We provide two kinds of subgroups $A$ of $SL(n, \mathbb{Z})$ with $n\geq 3$ such that (the quotient by the center of) the amalgamated free product $SL(n, \mathbb{Z})\ast_ASL(n, \mathbb{Z})$ fulfills all conditions in Assumption $(\ddagger)$, and thus it is ME rigid.
The simplest examples of them are presented in Theorem \ref{thm-mer-ex}.

The following classical theorem describes all automorphisms of $(P)SL(n, \mathbb{R})$ and will be applied repeatedly throughout this section.
The reader should consult \cite{dieudonne}, \cite{hua} and references therein for details and more general results.

\begin{thm}\label{thm-dieudonne}
Let $n$ be a positive integer.
Then the following assertions hold:
\begin{enumerate}
\item For any automorphism $f$ of $SL(n, \mathbb{R})$, there exists $g_0\in GL(n, \mathbb{R})$ such that either
\begin{itemize}
\item $f(g)=g_0gg_0^{-1}$ for any $g\in SL(n, \mathbb{R})$; or
\item $f(g)=g_0({}^t\! g)^{-1}g_0^{-1}$ for any $g\in SL(n, \mathbb{R})$,
\end{itemize}
where ${}^t\! g$ denotes the transpose of a matrix $g\in SL(n, \mathbb{R})$.
\item Any automorphism of $PSL(n, \mathbb{R})$ is induced by a unique automorphism of $SL(n, \mathbb{R})$.
\end{enumerate}
\end{thm}

We note that any automorphism of $SL(2, \mathbb{R})$ is the conjugation by an element of $GL(2, \mathbb{R})$ because we have
\[\left(
\begin{array}{cc}
0 & 1\\
-1 & 0\\
\end{array}\right)a\left(
\begin{array}{cc}
0 & 1\\
-1 & 0\\
\end{array}\right)^{-1}=({}^t\! a)^{-1},\quad \forall a\in SL(2, \mathbb{R}).
\]


\subsection{Tensor products}\label{subsec-ten}

In this subsection, we mean by a vector space a finite-dimensional one defined over the field $\mathbb{R}$ of real numbers unless otherwise stated.
When $V$ is a finite-dimensional vector space defined over $\mathbb{R}$, we denote by $GL(V)$ the group of all linear automorphisms of $V$ over $\mathbb{R}$, and by $SL(V)$ the subgroup of $GL(V)$ consisting of all linear automorphisms whose determinants are equal to one.
We denote by $I_V$ the identity map on $V$.

Let $V_1,\ldots, V_n$ be vector spaces and put $V=V_1\otimes \cdots \otimes V_n$.
We define a homomorphism $T\colon GL(V_1)\times \cdots \times GL(V_n)\rightarrow GL(V)$ by the formula
\[(g_1,\ldots, g_n)\mapsto g_1\otimes \cdots \otimes g_n\]
for $(g_1,\ldots, g_n)\in GL(V_1)\times \cdots \times GL(V_n)$, where $g_1\otimes \cdots \otimes g_n$ acts on $V$ by
\[(g_1\otimes \cdots \otimes g_n)(v_1\otimes \cdots \otimes v_n)=g_1v_1\otimes \cdots \otimes g_nv_n\]
for $v_1\otimes \cdots \otimes v_n\in V$.
Given a subgroup $G_i$ of $GL(V_i)$ for each $i$, let $G_1\otimes \cdots \otimes G_n$ denote the image $T(G_1\times \cdots \times G_n)$.
Note that the image of $SL(V_1)\otimes \cdots \otimes SL(V_n)$ in $PSL(V)$ is isomorphic to $PSL(V_1)\times \cdots \times PSL(V_n)$.
The following three lemmas will be used to compute the centralizer and normalizer in $GL(V)$ of the subgroup $SL(V_1)\otimes \cdots \otimes SL(V_n)$.

\begin{lem}\label{lem-tensor-cent}
Let $V_1$ and $V_2$ be vector spaces and put $V=V_1\otimes V_2$.
If $G$ is a subgroup of $SL(V_2)$ and if $g_0\in GL(V)$ commutes any element of $SL(V_1)\otimes G$, then $g_0$ is written as $g_0=I_{V_1}\otimes z$ for some $z \in {\rm Z}_{GL(V_2)}(G)$.
\end{lem}

\begin{proof}
Fix a basis for $V_1$ and decompose $g_0$ as the matrix $g_0=(g_{ij})_{i, j=1}^n$ with respect to the basis, where $g_{ij}\in {\rm End}(V_2)$ and $n=\dim V_1$.
Let $E_{ij}\in {\rm End}(V_1)$ be the elementary matrix whose $(i, j)$-entry is one and any other entries are zero with respect to the basis.
It then follows that
\begin{itemize}
\item the matrix $g_0((I_{V_1}+E_{ij})\otimes I_{V_2})$ is obtained by adding the $i$-th column of $g_0$ to the $j$-th column of $g_0$; and 
\item the matrix $((I_{V_1}+E_{ij})\otimes I_{V_2})g_0$ is obtained by adding the $j$-th row of $g_0$ to the $i$-th row of $g_0$.
\end{itemize}
By assumption, the two matrices $g_0((I_{V_1}+E_{ij})\otimes I_{V_2})$ and $((I_{V_1}+E_{ij})\otimes I_{V_2})g_0$ coincide if $i$ and $j$ are distinct.
The matrix $g_0$ is thus the diagonal matrix whose diagonal entries are all the same and lie in ${\rm Z}_{GL(V_2)}(G)$.
\end{proof}

\begin{lem}\label{lem-trans-inv}
Let $V_1$ and $V_2$ be vector spaces and put $V=V_1\otimes V_2$.
Pick $g_0\in GL(V)$.
If we have the equality
\[g_0(SL(V_1)\otimes I_{V_2})g_0^{-1}=SL(V_1)\otimes I_{V_2},\] 
then the automorphism of $SL(V_1)$ induced by $g_0$ is the conjugation by an element of $GL(V_1)$. 
\end{lem}

\begin{proof}
We put $n=\dim V_1$ and may assume $n\geq 2$.
Since any automorphism of $SL(2, \mathbb{R})$ is the conjugation by an element of $GL(2, \mathbb{R})$, we may assume $n\geq 3$.
Fix a basis for $V_1$.
By Theorem \ref{thm-dieudonne} (i), if the conclusion of the lemma were not true, then we would obtain the equality
\[g_0(a\otimes I_{V_2})g_0^{-1}=({}^t\! a)^{-1}\otimes I_{V_2},\quad \forall a\in SL(V_1)\]  
after multiplying $g_0$ by an appropriate element of $GL(V_1)\otimes I_{V_2}$, where ${}^t\! a$ denotes the transpose of a matrix $a\in SL(V_1)$ with respect to the chosen basis.
Let $a\in SL(V_1)$ be the matrix whose entries and transpose-inverse are given by 
\[a=
\left(
\begin{array}{ccccc}
1 & a_{12} & a_{13} & \cdots & a_{1n}\\
0 & 1 & 0 & \cdots & 0\\
0 & 0 & 1 & \cdots & 0\\
\vdots & \vdots & \vdots & \ddots & \vdots \\
0 & 0 & 0 & \cdots & 1\\
\end{array}
\right),\quad
({}^t\!a)^{-1}=
\left(
\begin{array}{ccccc}
1 & 0 & 0 & \cdots & 0\\
-a_{12} & 1 & 0 & \cdots & 0\\
-a_{13} & 0 & 1 & \cdots & 0 \\
\vdots & \vdots & \vdots & \ddots & \vdots \\
-a_{1n} & 0 & 0 & \cdots & 1\\
\end{array}
\right),
\]
where $a_{1j}\in \mathbb{R}$ is an arbitrary element.
Decomposing $g_0$ as the matrix $g_0=(g_{ij})_{i, j=1}^n$ with $g_{ij}\in {\rm End}(V_2)$, we obtain
\[g_0(a\otimes I_{V_2})=
\left(
\begin{array}{cccc}
g_{11} & a_{12}g_{11}+g_{12} & \cdots & a_{1n}g_{11}+g_{1n}\\
g_{21} & a_{12}g_{21}+g_{22} & \cdots & a_{1n}g_{21}+g_{2n}\\
\vdots & \vdots & \vdots & \vdots \\
\end{array}
\right),
\]
\[(({}^t\!a)^{-1}\otimes I_{V_2})g_0=
\left(
\begin{array}{cccc}
g_{11} & g_{12} & \cdots & g_{1n}\\
-a_{12}g_{11}+g_{21} & -a_{12}g_{12}+g_{22} & \cdots & -a_{12}g_{1n}+g_{2n}\\
\vdots & \vdots & \vdots & \vdots \\
\end{array}
\right).
\]
Since these two products coincide for any $a_{1j}\in \mathbb{R}$, we deduce the equalities
\[g_{11}=0,\quad g_{12}=-g_{21},\quad g_{21}=g_{13}=\cdots =g_{1n}=0\]
by comparing the $(1, j)$-entries for $j\geq 2$, the $(2, 2)$-entry and the $(2, j)$-entries for $j\geq 3$, respectively.
We therefore have $g_{1j}=0$ for each $j=1,\ldots, n$.
This is a contradiction because $g_0$ belongs to $GL(V)$.
\end{proof}

\begin{lem}\label{lem-tensor-nor}
Let $V_1,\ldots, V_n$ be vector spaces and put $V=V_1\otimes \cdots \otimes V_n$.
We define the subgroups $H$, $H_i$ of $SL(V)$ as
\[H=SL(V_1)\otimes \cdots \otimes SL(V_n),\]
\[H_i=I_{V_1}\otimes \cdots I_{V_{i-1}}\otimes SL(V_i)\otimes I_{V_{i+1}}\otimes \cdots \otimes I_{V_n}\]
for each $i=1,\ldots, n$.
Then the equality
\[\bigcap_{i=1}^n{\rm N}_{GL(V)}(H_i)=GL(V_1)\otimes \cdots \otimes GL(V_n)\]
holds, and the normalizer ${\rm N}_{SL(V)}(H)$ contains $H$ as a finite index subgroup.
\end{lem}

\begin{proof}
We denote by $L$ and $R$ the left and right hand sides of the equality in the lemma, respectively.
It is obvious that $R$ is contained in $L$.
Pick any $g_0\in L$. 
By Lemma \ref{lem-trans-inv}, we may assume that $g_0$ commutes any element of $H$ by multiplying an appropriate element of $R$ to $g_0$.
By Lemma \ref{lem-tensor-cent}, $g_0$ belongs to the center of $GL(V)$, which is contained in $R$.
The latter assertion of the lemma follows from the equality $L=R$ and the fact that $H_1,\ldots, H_n$ are the whole collection of connected simple normal subgroups of $H$ that are non-trivial and proper if $n\geq 2$ and $\dim V_i\geq 2$ for all $i$.
\end{proof}

Let $V_1,\ldots, V_n$ be finite-dimensional vector spaces over $\mathbb{R}$ and put $V=V_1\otimes \cdots \otimes V_n$.
If a basis $\Delta_i$ for $V_i$ is chosen for each $i=1,\ldots, n$, then the set
\[\Delta =\{ \, e_1\otimes \cdots \otimes e_n \mid e_i\in \Delta_i,\ i=1,\ldots, n \, \}\]
is a basis for $V$.
In the rest of this subsection, we fix these bases for $V_i$ and $V$.

\begin{prop}\label{prop-tensor-ddagger}
Let $V_1,\ldots, V_n$ be finite-dimensional vector spaces over $\mathbb{R}$ with $\dim V_i\geq 2$ for each $i=1,\ldots, n$ and with $\dim V_j\geq 3$ for some $j=1,\ldots, n$.
Fix bases $\Delta_i$ for $V_i$ for each $i$ and $\Delta$ for $V=V_1\otimes \cdots \otimes V_n$ as above.
We set
\[G=SL(V),\quad H=SL(V_1)\otimes \cdots \otimes SL(V_n).\]
For a subgroup $K$ of $GL(V)$, we denote by $K_{\mathbb{Z}}$ the group consisting of all elements of $K$ that preserve the set of integral points $\sum_{e\in \Delta}\mathbb{Z}e$.
We set $\Gamma =G_{\mathbb{Z}}$ and $A={\rm N}_G(H)_{\mathbb{Z}}$.
Then the following assertions hold:
\begin{enumerate}
\item The equality ${\rm LQN}_{\Gamma}(A)=A$ holds.
\item For any finite index subgroup $A'$ of $A$, the identity is the only automorphism of $G$ fixing all elements of $A'$. 
\end{enumerate}
\end{prop}

\begin{proof}
Pick $\gamma \in {\rm LQN}_{\Gamma}(A)$.
By definition, a finite index subgroup of $A$ is contained in $\gamma A\gamma^{-1}$. 
Since $H$ is a connected semisimple Lie group without compact factors and since $H_{\mathbb{Z}}$ is a finite index subgroup of $A$ by Lemma \ref{lem-tensor-nor}, the argument of Zariski closures show the inclusion $H<\gamma H\gamma^{-1}$.
Comparing the dimension, we obtain the equality $H=\gamma H\gamma^{-1}$ and thus $\gamma \in A$.
Assertion (i) is proved.

Let $f$ be an automorphism of $G$ fixing all elements of a finite index subgroup of $A$.
By Theorem \ref{thm-dieudonne}, $f$ is given by either
\begin{itemize}
\item $f(g)=g_0gg_0^{-1}$ for any $g\in G$; or
\item $f(g)=g_0({}^t\! g)^{-1}g_0^{-1}$ for any $g\in G$
\end{itemize}
for some $g_0\in GL(V)$, where ${}^t\! g$ denotes the transpose of a matrix $g\in G$ with respect to the basis $\Delta$.
The argument of Zariski closures shows that $f$ fixes all elements of $H$.
We thus have $g_0\in {\rm N}_{GL(V)}(H_i)$ for each $i=1,\ldots, n$, where $H_i$ is the subgroup of $H$ defined by
\[H_i=I_{V_1}\otimes \cdots I_{V_{i-1}}\otimes SL(V_i)\otimes I_{V_{i+1}}\otimes \cdots \otimes I_{V_n}.\]
It follows from Lemma \ref{lem-trans-inv} that for each $i=1,\ldots, n$, there exists $g_i\in GL(V_i)$ such that $(g_1\otimes \cdots \otimes g_n)g_0$ commutes any element of $H$.
By Lemma \ref{lem-tensor-cent}, $(g_1\otimes \cdots \otimes g_n)g_0$ lies in the center of $GL(V)$.
We may thus assume that $g_0$ is equal to $g_1^{-1}\otimes \cdots \otimes g_n^{-1}$.

If $f(g)=g_0gg_0^{-1}$ for any $g\in G$, then $g_i$ lies in the center of $GL(V_i)$ for each $i$ because $f$ is the identity on $H$.
It follows that $f$ is the identity on $G$.
We assume that $f(g)=g_0({}^t\! g)^{-1}g_0^{-1}$ for any $g\in G$.
Note that for any $h_i\in H_i$ with $i=1,\ldots, n$, the equality
\[{}^t\! (h_1\otimes \cdots \otimes h_n)={}^t\! h_1\otimes \cdots \otimes {}^t\! h_n\]
holds.
Let $j\in \{ 1,\ldots, n\}$ be the index with $\dim V_j\geq 3$.
We then have $g_j({}^t\! h)^{-1}g_j^{-1}=h$ for any $h\in H_j$.
Along argument in the proof of Lemma \ref{lem-trans-inv}, we can deduce a contradiction.
\end{proof}

\begin{rem}
Let $V_1,\ldots, V_n$ be vector spaces over $\mathbb{R}$ with $\dim V_i=2$ for each $i=1,\ldots, n$.
Define $G$, $H$, $\Gamma$ and $A$ as in Proposition \ref{prop-tensor-ddagger}.
We set
\[J=\left(
\begin{array}{cc}
0 & 1\\
-1 & 0\\
\end{array}\right),\quad J_n=J\otimes \cdots \otimes J\in H.\]
The automorphism of $G$ defined by $g\mapsto J_n^{-1}({}^t\! g)^{-1}J_n$ is then the identity on $H$, but is not the identity on $G$ if $n\geq 2$.
\end{rem}

\begin{thm}\label{thm-mer-tensor}
In the notation in Proposition \ref{prop-tensor-ddagger}, we put $\bar{\Gamma}=p(\Gamma)$ and $\bar{A}=p(A)$, where $p\colon SL(V)\rightarrow PSL(V)$ is the natural quotient map.
If $n\geq 2$, then the amalgamated free product $\Gamma_0=\bar{\Gamma}\ast_{\bar{A}}\bar{\Gamma}$ is coupling rigid with respect to $\comm(\Gamma_0)$. 
In particular, $\Gamma_0$ is ME rigid.
\end{thm}

\begin{proof}
By Proposition \ref{prop-tensor-ddagger} (i), we have the equality ${\rm LQN}_{\bar{\Gamma}}(\bar{A})=\bar{A}$.
By Theorem \ref{thm-dieudonne} (ii) and Proposition \ref{prop-tensor-ddagger} (ii), the centralizer of any finite index subgroup of $\bar{A}$ in $\aut(PSL(V))$ is trivial.
Since $\bar{\Gamma}$ is coupling rigid with respect to $\aut(PSL(V))$ by Theorem \ref{thm-me-hrl}, the group $\Gamma_0=\bar{\Gamma}\ast_{\bar{A}}\bar{\Gamma}$ fulfills all conditions in Assumption $(\ddagger)$.
Theorem \ref{thm-ddagger-comm-rigid} then implies the theorem.
\end{proof}

We note that the map $p\colon SL(V)\rightarrow PSL(V)$ induces the homomorphism from $\Gamma \ast_A\Gamma$ onto $\bar{\Gamma}\ast_{\bar{A}}\bar{\Gamma}$ whose kernel is equal to the center of $\Gamma \ast_A\Gamma$ and is either trivial or isomorphic to $\mathbb{Z}/2\mathbb{Z}$.
It follows that $\Gamma \ast_A\Gamma$ is also ME rigid.

In the notation in Proposition \ref{prop-tensor-ddagger}, let us describe $A={\rm N}_G(H)_{\mathbb{Z}}$ explicitly when $n=2$.
We leave the reader to consider the case where $n\geq 3$.
Theorem \ref{thm-mer-ex} (a) is a consequence of Theorem \ref{thm-mer-tensor} and the following:

\begin{lem}
Let $V_1$ and $V_2$ be finite-dimensional vector spaces over $\mathbb{R}$ with $\dim V_i\geq 2$ for each $i=1, 2$.
Define $V$, $G$, $H$ and $A$ as in Proposition \ref{prop-tensor-ddagger}.
We set
\[M=(GL(V_1)\otimes GL(V_2))\cap G\]
and $d_i=\dim V_i$ for $i=1, 2$.
Then the following assertions hold:
\begin{enumerate}
\item We have $M_{\mathbb{Z}}<A$ and $M_{\mathbb{Z}}=(GL(V_1)_{\mathbb{Z}}\otimes GL(V_2)_{\mathbb{Z}})\cap G$.
\item If either $d_1\neq d_2$ or $d_1=d_2$ and this number is congruent to $2$ modulo $4$, then $A=M_{\mathbb{Z}}$.
Otherwise $[A: M_{\mathbb{Z}}]=2$.
\item If at least one of $d_1$ and $d_2$ is odd, then we have the equality $M_{\mathbb{Z}}=H_{\mathbb{Z}}=SL(V_1)_{\mathbb{Z}}\otimes SL(V_2)_{\mathbb{Z}}$.
\end{enumerate}
\end{lem}

\begin{proof}
The inclusion of assertion (i) is obvious.
Choose bases $\Delta_1$, $\Delta_2$ and $\Delta$ as in Proposition \ref{prop-tensor-ddagger}.
We note that if $g\in GL(V_1)$ and $h\in GL(V_2)$ are represented as matrices $g=(g_{ij})$ and $h=(h_{kl})$ with respect to the bases $\Delta_1$ and $\Delta_2$, respectively, then $g\otimes h\in SL(V)$ is represented as the matrix the set of whose entries consists of the products $g_{ij}h_{kl}$ for all $i$, $j$, $k$ and $l$, with respect to the basis $\Delta$.
It follows that the right hand side of the equality in assertion (i) is contained in $M_{\mathbb{Z}}$.

Pick $g\in GL(V_1)$ and $h\in GL(V_2)$ with $g\otimes h\in M_{\mathbb{Z}}$. 
We claim that if both $\det g$ and $\det h$ lie in $\{ \pm 1\}$, then $g\in GL(V_1)_{\mathbb{Z}}$ and $h\in GL(V_2)_{\mathbb{Z}}$.
This claim implies the equality in assertion (i).
Write $g=(g_{ij})$ and $h=(h_{kl})$ as in the last paragraph.
For any $i$, $j$, $k$ and $l$, the product $g_{ij}h_{kl}$ is then an integer.
For any $i$ and $j$, the determinant of the matrix $(g_{ij}h_{kl})_{k, l}$ is an integer and is equal to $(g_{ij})^{d_2}$ or $-(g_{ij})^{d_2}$.
It follows that if $g_{ij}\neq 0$, then $g_{ij}$ is uniquely written as $g_{ij}=s_{ij}(t_{ij})^{1/d_2}$, where $s_{ij}$ is a non-zero integer and $t_{ij}$ is a positive integer such that the power of each prime in the prime factorization of $t_{ij}$ is less than $d_2$.
On the other hand, we have $g_{i_1j_1}/g_{i_2j_2}\in \mathbb{Q}$ for any $i_1$, $j_1$, $i_2$ and $j_2$ with $g_{i_2j_2}\neq 0$.
All $t_{ij}$ therefore coincide, and we denote the number by $t$.
The number $\det g$ is then equal to the product of $t^{d_1/d_2}$ and an integer. 
Since $\det g$ lies in $\{ \pm 1\}$, $t$ is equal to one.
We thus have $g\in GL(V_1)_{\mathbb{Z}}$.
Similarly, we have $h\in GL(V_2)_{\mathbb{Z}}$.

We next prove assertion (ii).
We set
\[H_1=SL(V_1)\otimes I_{V_2},\quad H_2=I_{V_1}\otimes SL(V_2).\]
If $d_1$ and $d_2$ are distinct, then $A=M_{\mathbb{Z}}$ by Lemma \ref{lem-tensor-nor}.
Suppose $d=d_1=d_2\geq 2$.
Let $f\colon V_1\rightarrow V_2$ be the isomorphism defined by $f(e^1_j)=e^2_j$ for each $j=1,\ldots, d$, where we put $\Delta_i=\{ e_1^i,\ldots, e_{d}^i\}$ for $i=1, 2$.
The automorphism $\varphi$ of $V$ defined by
\[\varphi(x\otimes y)=f^{-1}(y)\otimes f(x),\quad x\in V_1,\ y\in V_2\]
then belongs to $GL(V)_{\mathbb{Z}}$.
We then have $\det \varphi =(-1)^{d(d-1)/2}$, $\varphi H_1\varphi^{-1}=H_2$ and $\varphi H_2\varphi^{-1}=H_1$.
If $d$ is congruent to either $0$ or $1$ modulo $4$, then $\det \varphi =1$ and $\varphi$ is in $A\setminus M_{\mathbb{Z}}$.
We thus have $[A: M_{\mathbb{Z}}]=2$.
If $d$ is congruent to $3$ modulo $4$, then $\det \varphi =-1$.
Pick $a\in GL(V_2)_{\mathbb{Z}}$ with $\det a=-1$.
The product $(I_{V_1}\otimes a)\varphi$ belongs to $A\setminus M_{\mathbb{Z}}$ because $\det (I_{V_1}\otimes a)=(-1)^d=-1$.
We thus have $[A: M_{\mathbb{Z}}]=2$.

If $d$ is congruent to $2$ modulo $4$, then $\det \varphi =-1$.
We claim that the equality ${\rm N}_G(H)=M$ holds.
Assuming that there exists $g\in G$ with $gH_1g^{-1}=H_2$ and $gH_2g^{-1}=H_1$, we deduce a contradiction.
The product $\varphi g$ belongs to ${\rm N}_{GL(V)}(H_1)$ and ${\rm N}_{GL(V)}(H_2)$.
By Lemma \ref{lem-trans-inv}, there exists $h_i\in GL(V_i)$ with $\det h_i\in \{ \pm 1\}$ for $i=1, 2$ such that the product $(h_1\otimes h_2)\varphi g$ commutes any element of $H$.
It follows from Lemma \ref{lem-tensor-cent} that $(h_1\otimes h_2)\varphi g$ is in the center of $GL(V)$.
Since the determinant of any element in the center of $GL(V)$ is positive, this contradicts $\det((h_1\otimes h_2)\varphi g)=-1$.
The claim is proved.
The equality $A=M_{\mathbb{Z}}$ therefore holds.

We prove assertion (iii).
We set $S=SL(V_1)_{\mathbb{Z}}\otimes SL(V_2)_{\mathbb{Z}}$.
It suffices to show the inclusion $M_{\mathbb{Z}}< S$.
Pick $g_1\in GL(V_1)_{\mathbb{Z}}$ and $g_2\in GL(V_2)_{\mathbb{Z}}$ with $g=g_1\otimes g_2\in G$.
Suppose that both $d_1$ and $d_2$ are odd.
If $\det g_1=1$, then $\det g_2=1$ and thus $g\in S$.
If $\det g_1=-1$, then $\det g_2=-1$, and we have $g=(-g_1)\otimes (-g_2)$ with $\det(-g_1)=\det(-g_2)=1$.
We thus have $g\in S$.
Suppose that $d_1$ is odd and $d_2$ is even.
Since $\det g=(\det g_1)^{d_2}(\det g_2)^{d_1}$, we have $\det g_2=1$.
If $\det g_1=-1$, then we have $g=(-g_1)\otimes (-g_2)$ with $\det(-g_1)=\det(-g_2)=1$.
We thus have $g\in S$.
The same arugment can be applied to the case where $d_1$ is even and $d_2$ is odd.
\end{proof}


\subsection{Upper block triangular matrices}\label{subsec-but}

Throughout this subsection, $k$ stands for either the field $\mathbb{R}$ of real numbers or the field $\mathbb{C}$ of complex numbers. Let $V$ be a finite-dimensional vector space over $k$.
We mean by a {\it flag} of $V$ a strictly increasing sequence of subspaces of $V$ from $0$ to $V$.
The flag $0\subsetneq V_1\subsetneq \cdots \subsetneq V_l\subsetneq V$ of $V$ is said to be of {\it type} $(d_1,\ldots, d_l)$ if $d_i=\dim V_i$ for each $i=1,\ldots, l$.

\begin{notation}\label{notation-but}
Let $\Delta =(n_1,\ldots, n_l)$ be a finite sequence of positive integers, and put $n=n_1+\cdots +n_l$.
We denote by $\{ e_1,\ldots, e_n\}$ the standard basis for the vector space $k^n$ over $k$.
For each $i=1,\ldots, l$, let $k^{(i)}$ be the subspace of $k^n$ spanned by $e_1,\ldots, e_{d_i}$, where $d_i=n_1+\cdots +n_i$.
We denote by $F(\Delta, k)$ the flag corresponding to the sequence of the subspaces
\[0\subsetneq k^{(1)}\subsetneq k^{(2)}\subsetneq \cdots \subsetneq k^{(l-1)}\subsetneq k^n.\] 
We define the subgroup $P(\Delta, k)$ of $SL(n, k)$ as the subgroup consisting of all matrices in $SL(n, k)$ of the form
\[
\left(
\begin{array}{ccccc}
a_1 & \ast & \ast & \cdots & \ast \\
0 & a_2 & \ast & \cdots & \ast \\
0 & 0 & a_{3} & \cdots & \ast \\ 
\vdots & \vdots & \vdots & \ddots & \vdots \\
0 & 0 & 0 & \cdots & a_l \\
\end{array}
\right),
\]
where $a_i\in SL(n_i, k)$ for each $i=1,\ldots, l$.
The group $P(\Delta, k)$ fixes the flag $F(\Delta, k)$.
\end{notation}

\begin{lem}\label{lem-flag}
Let $\Delta =(n_1,\ldots, n_l)$ be a finite sequence of positive integers.
In Notation \ref{notation-but}, $F(\Delta, k)$ is the only flag of type $(d_1,\ldots, d_{l-1})$ fixed by $P(\Delta, k)$.  
\end{lem}

\begin{proof}
Let $U$ be a subspace of $k^n$ of dimension $d_{l-1}$ fixed by $P(\Delta, k)$.
We show the equality $U=k^{(l-1)}$.
Assume that there exists $u=\sum_i\lambda_ie_i\in U$ with $\lambda_i\in k$ such that $\lambda_{i_0}\neq 0$ for some $i_0$ with $d_{l-1}<i_0\leq n$.
Applying the matrices of the form
\[I+\sum_{i=1}^{d_{l-1}}t_iE_{ii_0}\in P(\Delta, k),\quad t_i\in k\]
to $u$, we see that $U$ contains $k^{(l-1)}$, where $I$ is the identity $n$-by-$n$ matrix and $E_{ij}$ is the elementary $n$-by-$n$ matrix whose $(i, j)$-entry is one and any other entries are zero.
We thus have $\dim U>d_{l-1}$.
This is a contradiction.
It follows that $U$ is contained in $k^{(l-1)}$, and thus $U=k^{(l-1)}$.

Similarly, we can show that for each $i=1,\ldots, l-1$, $k^{(i)}$ is the only $d_i$-dimensional subspace of $k^{(i+1)}$ fixed by $P(\Delta, k)$.
The lemma follows. 
\end{proof}

\begin{prop}\label{prop-but-ddagger}
Let $\Delta =(n_1,\ldots, n_l)$ be a finite sequence of positive integers with either $n_1\geq 2$ or $n_l\geq 2$.
We put $n=n_1+\cdots +n_l$ and
\[G=SL(n, \mathbb{R}),\quad \Gamma =SL(n, \mathbb{Z}).\]
Define the subgroup $A$ of $\Gamma$ to be the stabilizer of the flag $F(\Delta, \mathbb{R})$ in $\Gamma$.
Then the following assertions hold:
\begin{enumerate}
\item The equality ${\rm LQN}_{\Gamma}(A)=A$ holds.
\item For any finite index subgroup $A'$ of $A$, the identity is the only automorphism of $G$ fixing all elements of $A'$.
\end{enumerate}
\end{prop}

\begin{proof}
Pick $\gamma \in {\rm LQN}_{\Gamma}(A)$.
There exists a finite index subgroup of $A$ contained in $\gamma A\gamma^{-1}$.
Since the identity component of the Zariski closure of $A$ in $SL(n, \mathbb{C})$ is equal to $P(\Delta, \mathbb{C})$, $\gamma$ normalizes $P(\Delta, \mathbb{C})$. 
Lemma \ref{lem-flag} implies that $\gamma$ fixes $F(\Delta, \mathbb{C})$, and thus $\gamma \in A$.
Assertion (i) follows.

We next prove assertion (ii).
Theorem \ref{thm-dieudonne} describes all automorphisms of $G$.
Suppose that $g_0\in GL(n, \mathbb{R})$ commutes any element of a finite index subgroup $A'$ of $A$.
It then follows that $g_0$ commutes any element in the identity component $P(\Delta, \mathbb{C})$ of the Zariski closure of $A'$.
By Lemma \ref{lem-flag}, $g_0$ fixes the flag $F(\Delta, \mathbb{R})$ and is of the form
\[g_0=
\left(
\begin{array}{ccccc}
\lambda_1I_{n_1} & \ast & \ast & \cdots & \ast \\
0 & \lambda_2I_{n_2} & \ast & \cdots & \ast \\
0 & 0 & \lambda_{3}I_{n_{3}} & \cdots & \ast \\ 
\vdots & \vdots & \vdots & \ddots & \vdots \\
0 & 0 & 0 & \cdots & \lambda_lI_{n_l} \\
\end{array}
\right)
\]
for some $\lambda_i\in \mathbb{R}$, where $I_{n_i}$ denotes the identity $n_i$-by-$n_i$ matrix.
We denote by $I$ the identity $n$-by-$n$ matrix, and denote by $E_{ij}$ the elementary $n$-by-$n$ matrix whose $(i, j)$-entry is one and any other entries are zero.
Recall the following elementary facts:
\begin{itemize}
\item The matrix $g_0(I+E_{ij})$ is obtained by adding the $i$-th column of $g_0$ to the $j$-th column of $g_0$.
\item The matrix $(I+E_{ij})g_0$ is obtained by adding the $j$-th row of $g_0$ to the $i$-th row of $g_0$.
\end{itemize}
Since $g_0(I+E_{ij})=(I+E_{ij})g_0$ for any $i<j$, we have $g_0=\lambda I+tE_{1n}$ for some $\lambda \in \mathbb{R}\setminus \{ 0\}$ and $t\in \mathbb{R}$.
Our assumption that either $n_1\geq 2$ or $n_l\geq 2$ implies $t=0$.
Therefore, $g_0$ lies in the center of $GL(n, \mathbb{R})$.
We have shown that any automorphism of $G$ that is the conjugation by an element of $GL(n, \mathbb{R})$ and fixes any element of $A'$ is the identity.
Assertion (ii) in the case $n=2$ follows because any automorphism of $SL(2, \mathbb{R})$ is the conjugation by an element of $GL(2, \mathbb{R})$.

Assume $n\geq 3$.
Pick $g_0\in GL(n, \mathbb{R})$ and suppose that the automorphism of $G$ defined by $g\mapsto g_0({}^t\!g)^{-1}g_0^{-1}$ fixes any element of $A'$.
The equality $g_0({}^t\!a)^{-1}g_0^{-1}=a$ then holds for any $a\in A'$.
Along argument in the proof of Lemma \ref{lem-trans-inv}, we can deduce a contradiction.
Assertion (ii) is proved.
\end{proof}

As a consequence of Proposition \ref{prop-but-ddagger}, the following theorem is obtained in the same way as in the proof of Theorem \ref{thm-mer-tensor}.
It proves Theorem \ref{thm-mer-ex} (b).

\begin{thm}\label{thm-mer-but}
In the notation in Proposition \ref{prop-but-ddagger}, we assume $l\geq 2$ and put $\bar{\Gamma}=p(\Gamma)$ and $\bar{A}=p(A)$, where $p\colon SL(V)\rightarrow PSL(V)$ is the natural quotient map.
Then the amalgamated free product $\Gamma_0=\bar{\Gamma}\ast_{\bar{A}}\bar{\Gamma}$ is coupling rigid with respect to $\comm(\Gamma_0)$.
In particular, $\Gamma_0$ is ME rigid.
\end{thm}


\section{Miscellaneous examples}\label{sec-mis-ex}

We provide a variety of examples of groups to which discussion so far can be applied.
In Section \ref{subsec-int}, we find subgroups of arithmetic lattices which are left-quasi-normalized by themselves, relying on \cite{burger-harpe}.
In Section \ref{subsec-diag}, we focus on the subgroups $A$ of $\Gamma =SL(n, \mathbb{Z})$ consisting of diagonal blocks (see Notation \ref{notation-diag} for a precise definition).
Given a discrete group $\Lambda$ which is ME to the amalgamated free product $\Gamma\ast_A\Gamma$, we observe behavior of the action of $\Lambda$ on the Bass-Serre tree associated with $\Gamma \ast_A\Gamma$.
We note that these $\Gamma \ast_A\Gamma$ do not fulfill all conditions in Assumption $(\ddagger)$ (see Remark \ref{rem-not-ddagger}).

\subsection{Integral points of algebraic subgroups}\label{subsec-int}

Let us collect the notation and conventions employed in this subsection.
We assume all algebraic groups to be linear.
We mean by a $k$-group an algebraic group defined over a field $k$.
The identity component of a $k$-group $G$ is denoted by $G^0$. If $G$ is defined over the field $\mathbb{Q}$ of rational numbers, then $G_{\mathbb{Z}}$ stands for the set of integral points of $G$.
For a subset $S$ of $G$, we denote by $S^-$ its closure with respect to the Zariski topology.

\begin{prop}\label{prop-tori}
Let $G$ be a connected reductive $\mathbb{R}$-group and $\Gamma$ a discrete subgroup of $G_{\mathbb{R}}$.
If $T$ is a maximal $\mathbb{R}$-split torus of $G$ such that $T_{\mathbb{R}}/(T_{\mathbb{R}}\cap \Gamma)$ is compact, then we have the equality
\[{\rm LQN}_{\Gamma}({\rm N}_G(T)\cap \Gamma)={\rm N}_G(T)\cap \Gamma.\]  
\end{prop}

\begin{proof}
This is essentially proved in Proposition 5.1 in \cite{burger-harpe}.
We give a proof for the reader's convenience.
Note that the group $({\rm N}_G(T)\cap \Gamma)/(T\cap \Gamma)$ is finite because there is a natural injective homomorphism from it into the group ${\rm N}_G(T)/T$, which is finite by 8.10 and 13.17 in \cite{borel-book}.
Pick any $\gamma \in {\rm LQN}_{\Gamma}({\rm N}_G(T)\cap \Gamma)$.
There then exists a finite index subgroup $\Delta$ of $T\cap \Gamma$ contained in $\gamma(T\cap \Gamma)\gamma^{-1}$.
Passing to the Zariski closure, we see that $T$ is a subgroup of $\gamma T\gamma^{-1}$.
Comparing the dimension, we obtain $\gamma \in {\rm N}_G(T)$.
\end{proof}

\begin{ex}
Let $G$ be a connected semisimple $\mathbb{R}$-group and $\Gamma$ a lattice in $G_{\mathbb{R}}$.
By \cite{prasad-rag}, there exists a maximal $\mathbb{R}$-split torus $T$ in $G$ with $T_{\mathbb{R}}/(T_{\mathbb{R}}\cap \Gamma)$ compact.
We refer to Section 5 in \cite{burger-harpe} for examples of maximal tori satisfying the assumption in Proposition \ref{prop-tori}.
\end{ex}

For a connected algebraic $\mathbb{Q}$-group $G$, we define $\mathscr{S}(G)$ as the set of all connected $\mathbb{Q}$-subgroups $H$ of $G$ such that $(H_{\mathbb{Z}})^-=H$ and $[(\textrm{N}_G(H)_{\mathbb{Z}})^-: H]< \infty$, where the closure is taken with respect to the Zariski topology.

\begin{lem}\label{lem-s-lqn}
Let $G$ be a connected algebraic $\mathbb{Q}$-group.
Then the following assertions hold:
\begin{enumerate}
\item For each $H\in \mathscr{S}(G)$, the equality ${\rm LQN}_{G_{\mathbb{Z}}}({\rm N}_G(H)_{\mathbb{Z}})={\rm N}_G(H)_{\mathbb{Z}}$ holds.
\item Let $H$ be a connected $\mathbb{Q}$-subgroup of $G$.
Suppose that $H$ is a semisimple group such that the equality ${\rm N}_G(H)^0=H$ holds and the semisimple Lie group $H_{\mathbb{R}}$ has no compact factor.
Then $H\in \mathscr{S}(G)$.
\end{enumerate}
\end{lem}

\begin{proof}
We owe this lemma to argument in Section 6 of \cite{burger-harpe}.
Pick $H\in \mathscr{S}(G)$ and $\gamma \in {\rm LQN}_{G_{\mathbb{Z}}}({\rm N}_G(H)_{\mathbb{Z}})$.
There then exists a finite index subgroup of ${\rm N}_G(H)_{\mathbb{Z}}$ contained in $\gamma {\rm N}_G(H)_{\mathbb{Z}}\gamma^{-1}$.
Taking the closure, we obtain the inclusion $H<\gamma H\gamma^{-1}$.
We thus have $H=\gamma H\gamma^{-1}$ and $\gamma \in {\rm N}_G(H)_{\mathbb{Z}}$. 
The equality in assertion (i) is proved.

In the notation in assertion (ii), it is known that $H_{\mathbb{Z}}$ is a lattice in $H_{\mathbb{R}}$ and that any lattice in $H_{\mathbb{R}}$ is Zariski dense in $H$. 
These two facts are due to \cite{b-hc} and \cite{borel-density}, respectively.
The equality $(H_{\mathbb{Z}})^-=H$ thus holds.
\end{proof}

\begin{ex}
Let $G$ be a connected algebraic $\mathbb{Q}$-group $G$ which is $\mathbb{Q}$-simple, that is, there is no connected normal $\mathbb{Q}$-subgroup of $G$ other than $\{ e\}$ and $G$.
Let $H$ be a maximal connected $\mathbb{Q}$-subgroup of $G$.
We suppose that $H$ is semisimple and that the semisimple Lie group $H_{\mathbb{R}}$ has no compact factor.
The maximality of $H$ implies that ${\rm N}_G(H)^0$ equals either $G$ or $H$.
Since $G$ is $\mathbb{Q}$-simple, ${\rm N}_G(H)^0$ equals $H$.
We thus have $H\in \mathscr{S}(G)$.

Maximal subgroups of the classical groups are classified by Dynkin \cite{dynkin}.
If $n$ is an even positive integer, then the symplectic group $Sp(n, \mathbb{C})$ is a maximal connected closed subgroup of $SL(n, \mathbb{C})$. 
\end{ex}

\subsection{Diagonal subgroups}\label{subsec-diag}

We first introduce the notation employed throughout this subsection.

\begin{notation}\label{notation-diag}
We fix positive integers $n_1$, $n_2$ with $n=n_1+n_2\geq 3$.
We set
\[G=SL(n, \mathbb{R}),\quad G_1=SL(n_1, \mathbb{R}),\quad G_2=SL(n_2, \mathbb{R}).\]
For $g_i\in GL(n_i, \mathbb{R})$ with $i=1, 2$, we define $\diag(g_1, g_2)$ as the $n$-by-$n$ matrix
\[
\left(
\begin{array}{cc}
g_1 & 0\\
0 & g_2\\
\end{array}
\right)\in GL(n, \mathbb{R}).
\]
For subgroups $H_1<GL(n_1, \mathbb{R})$ and $H_2<GL(n_2, \mathbb{R})$, we denote by $\diag(H_1, H_2)$ the subgroup of $GL(n, \mathbb{R})$ consisting of all elements of the form $\diag(g_1, g_2)$
with $g_i\in H_i$ for each $i=1, 2$.
For a subgroup $K$ of $G$, we set $K_{\mathbb{Z}}=K\cap SL(n, \mathbb{Z})$. Let $p\colon SL(n, \mathbb{R})\rightarrow PSL(n, \mathbb{R})$ denote the natural quotient map.
We put
\[H=\diag(G_1, G_2),\quad \Gamma =SL(n, \mathbb{Z}),\quad A={\rm N}_G(H)_{\mathbb{Z}}\]
and put $\bar{\Gamma}=p(\Gamma)$ and $\bar{A}=p(A)$.
Let $I$, $I_1$ and $I_2$ denote the neutral elements of $G$, $G_1$ and $G_2$, respectively.
\end{notation}

We investigate the amalgamated free product $\Gamma \ast_A\Gamma$ and $\bar{\Gamma}\ast_{\bar{A}}\bar{\Gamma}$.
The latter is isomorphic to the quotient of $\Gamma \ast_A\Gamma$ by its center.

\begin{rem}\label{rem-not-ddagger}
In Notation \ref{notation-diag}, the centralizer of $A$ in $G$ contains all matrices of the form $\diag(\lambda^{n_2}I_1, \lambda^{-n_1}I_2)$ for any non-zero $\lambda \in \mathbb{R}$.
The images of these matrices via the natural projection $p$ are non-trivial if $\lambda \not\in \{ \pm 1\}$.
It follows that $\bar{\Gamma}\ast_{\bar{A}}\bar{\Gamma}$ does not satisfy condition (f) in Assumption $(\ddagger)$.

If either $n_1$ or $n_2$ is even, then the center of $\bar{A}$ is non-trivial and hence $\bar{\Gamma}\ast_{\bar{A}}\bar{\Gamma}$ is not coupling rigid with respect to the commensurator by Proposition \ref{prop-not-coup}.
\end{rem}

The following theorem due to Kazhdan-Margulis states the existence of a lower bound of the covolumes of lattices in semisimple Lie groups.
This lower bound is directly linked with the size of groups which are commensurable with a lattice in $G$.
As a result, it is shown that Conditions $(\mathsf{V})$ and $(\mathsf{E})$ introduced in Notation \ref{notation-vef} are fulfilled by $\Gamma$ and $A$ in Notation \ref{notation-diag}.

\begin{thm}[\ci{Corollary 11.9}{rag-book}]\label{thm-kazhdan-margulis}
Let $G$ be a connected semisimple Lie group without compact factors and $m$ a Haar measure on $G$.
Then there exists a positive constant $M$ satisfying the following:
For any discrete subgroup $\Gamma$ of $G$, the total volume of the homogeneous space $G/\Gamma$ with respect to the measure on it induced by $m$ is at least $M$.
\end{thm}

\begin{prop}\label{prop-kazhdan-margulis}
Let $G$ be a connected semisimple Lie group without compact factors.
We denote by $\imath \colon G\rightarrow \aut (\ad G)$ the natural homomorphism.
Let $\Gamma$ be a lattice in $G$.
Then there exists a positive constant $c$ with the following property:
For any subgroup $\Delta$ of $\aut(\ad G)$ with $\Delta \asymp \imath(\Gamma)$, we have the inequality
\[\frac{[\Delta :\imath(\Gamma)\cap \Delta]}{[\imath(\Gamma): \imath(\Gamma)\cap \Delta]}\leq c.\]
\end{prop}

\begin{proof}
Let $m$ be the Haar measure on $\ad G$.
For a discrete subgroup $\Lambda$ of $\ad G$, we denote by $m(\ad G/\Lambda)$ the total volume of $\ad G/\Lambda$ with respect to the measure on it induced by $m$.
Let $\Delta$ be any subgroup of $\aut(\ad G)$ with $\Delta \asymp \imath(\Gamma)$ and put $\Delta_1=\Delta \cap \ad G$.
We then have
\begin{align*}
\frac{[\Delta :\imath(\Gamma)\cap \Delta]}{[\imath(\Gamma): \imath(\Gamma)\cap \Delta]}&\leq \frac{[\aut(\ad G): \ad G][\Delta_1: \imath(\Gamma)\cap \Delta_1]}{[\imath(\Gamma): \imath(\Gamma)\cap \Delta_1]}\\
&\leq [\aut(\ad G): \ad G]m(\ad G/\imath(\Gamma))m(\ad G/\Delta_1)^{-1}\\
&\leq [\aut(\ad G): \ad G]m(\ad G/\imath(\Gamma))M^{-1},
\end{align*}
where $M$ is a positive constant obtained by applying Theorem \ref{thm-kazhdan-margulis} to $\ad G$.
Since the index $[\aut(\ad G): \ad G]$ is finite, the proposition follows.
\end{proof}

\begin{lem}\label{lem-diag-vef}
In Notation \ref{notation-diag}, the following assertions hold:
\begin{enumerate}
\item We have $[A: H_{\mathbb{Z}}]=2$ if $n_1\neq n_2$, and $[A: H_{\mathbb{Z}}]=4$ if $n_1=n_2$.
\item The equality ${\rm LQN}_{\Gamma}(A)=A$ holds.
\item Let $\imath$ denote the natural homomorphism from $G$ into $\aut(\ad G)$. 
Then the pair $(\Gamma, (\aut(\ad G), \imath))$ satisfies Condition $(\mathsf{V})$.
\item The triplet $\mathscr{T}=(\Gamma, A, (\aut(\ad G), \imath))$ satisfies Condition $(\mathsf{E})$.
\end{enumerate}
\end{lem}

\begin{proof}
Let $\{ e_1,\ldots, e_n\}$ denote the standard basis for $\mathbb{R}^n$.
Any element of ${\rm N}_G(H)$ preserves the decomposition of the vector space $\mathbb{R}^n$ into the two subspaces spanned by $e_1,\ldots, e_{n_1}$ and spanned by $e_{n_1+1},\ldots, e_n$.
It follows that the subgroup
\[G_0=\diag(GL(n_1, \mathbb{R}), GL(n_2, \mathbb{R}))\cap SL(n, \mathbb{R})\]
of ${\rm N}_G(H)$ has index one if $n_1\neq n_2$, and has index two if $n_1=n_2$. 
Assertion (i) then follows.
Assertions (ii) and (iii) follow from Lemma \ref{lem-s-lqn} and Proposition \ref{prop-kazhdan-margulis}, respectively.

We next prove assertion (iv).
Choose a subgroup $\Delta$ of $\aut (\ad G)$ and a subgroup $\Delta_0$ of $\Delta$ with $\Delta \asymp \imath(\Gamma)$ and $\Delta_0\asymp \imath(A)$.
The group $\imath^{-1}(\Delta_0)$ is contained in ${\rm N}_G(H)$ because it is contained in $\comm_G(A)$.
There thus exists a subgroup of $\imath^{-1}(\Delta_0)$ of index at most two contained in $G_0$.
If there were $g\in \imath^{-1}(\Delta_0)$ of the form $g=\diag(g_1, g_2)$ with $\det g_1\neq \pm 1$, then $\imath^{-1}(\Delta_0)$ and $A$ would not be commensurable in $G$ because for any integer $n$, we have $g^n\in A$ if and only if $n=0$.
There thus exists a subgroup of $\imath^{-1}(\Delta_0)$ of index at most four contained in $\diag(G_1, G_2)$.
Assertion (iv) then follows from Proposition \ref{prop-kazhdan-margulis} because $\diag(G_1, G_2)$ is a connected semisimple Lie group without compact factors.
\end{proof}

\begin{rem}
The triplet $\mathscr{T}=(\Gamma, A, (\aut(\ad G), \imath))$ in Lemma \ref{lem-diag-vef} does not satisfy Condition $(\mathsf{F})$ in Notation \ref{notation-vef}.
For if either $n_1$ or $n_2$ is even, then $\diag(-I_1, I_2)$ or $\diag(I_1, -I_2)$ lies in $G$.
Each of them commutes any element of $H$.
It follows that Condition $(\mathsf{F})$ is not satisfied.
If both $n_1$ and $n_2$ are odd, then we can show that Condition $(\mathsf{F})$ is not satisfied by using the element of $\aut(\ad G)$ defined by the conjugation by $\diag(-I_1, I_2)$.
\end{rem}

By Corollary \ref{cor-coup-tree}, if $\Lambda$ is a discrete group ME to $\bar{\Gamma}\ast_{\bar{A}}\bar{\Gamma}$, then $\Lambda$ acts on the Bass-Serre tree $T$ for $\bar{\Gamma}\ast_{\bar{A}}\bar{\Gamma}$.
The following theorem shows finiteness properties of this action of $\Lambda$ on $T$.

\begin{thm}\label{thm-diag-cocompact}
In Notation \ref{notation-diag}, we define the amalgamated free product $\Gamma_0=\bar{\Gamma}\ast_{\bar{A}}\bar{\Gamma}$ and denote by $T$ the associated Bass-Serre tree.
If $\Lambda$ is a discrete group ME to $\Gamma_0$, then there exist a subgroup $\Lambda_+$ of $\Lambda$ of index at most two and a homomorphism $\rho \colon \Lambda_+\rightarrow \aut(T)$ satisfying the following assertions:
\begin{enumerate}
\item For each $s\in V(T)\cup E(T)$, we denote by $\Lambda_s$ the stabilizer of $s$ in $\Lambda_+$.
If $v\in V(T)$, then $\Lambda_v$ is virtually isomorphic to $\bar{\Gamma}$.
If $e\in E(T)$, then $\Lambda_e$ is virtually isomorphic to $\bar{A}$.
\item For any $v\in V(T)$, the number of orbits for the action of $\Lambda_v$ on the set of vertices in the link of $v$ in $T$ is finite.
\item If there exists a uniform upper bound for the cardinalities of finite subgroups of $\Lambda$, then the action of $\Lambda_+$ on $T$ is cocompact.
\end{enumerate}
\end{thm}

\begin{proof}
Since $n\geq 3$, the group $\bar{\Gamma}$ satisfies property (T).
By Lemma \ref{lem-diag-vef} (ii), the group $\Gamma_0=\bar{\Gamma}\ast_{\bar{A}}\bar{\Gamma}$ fulfills all conditions in Assumption $(\star)$.
Let $(\Sigma, m)$ be a coupling of $\Gamma_0$ and $\Lambda$.
By Corollary \ref{cor-coup-tree}, there exist a homomorphism $\rho \colon \Lambda \rightarrow \aut^*(T)$ and an almost $(\Gamma_0\times \Lambda)$-equivariant Borel map $\Phi \colon \Sigma \rightarrow (\aut^*(T), \imath, \rho)$, where $\imath \colon \Gamma_0\rightarrow \aut^*(T)$ is the homomorphism associated with the action of $\Gamma_0$ on $T$.
As discussed in Section \ref{subsec-circ}, replacing $\rho$ and $\Phi$ appropriately, we may assume that $\Phi^{-1}(\aut(T))$ has positive measure.
We set
\[\Lambda_+=\rho^{-1}(\rho(\Lambda)\cap \aut(T)),\quad \Sigma_+=\Phi^{-1}(\aut(T)).\]
Note that $\Sigma_+$ is a coupling of $\Gamma_0$ and $\Lambda_+$.
Applying Lemma \ref{lem-small-coup} and Corollary \ref{cor-moore} to $\Sigma_+$, we obtain assertion (i).
Proposition \ref{prop-cocompact} and Lemma \ref{lem-diag-vef} (iii), (iv) then show assertions (ii) and (iii).
\end{proof}


\section{Additional results}\label{sec-add}

We present a few consequences of coupling rigidity of groups with respect to their commensurators and brief observations on groups ME to free products, without a precise proof.
These results can be deduced by known techniques.

\subsection{Lattice embeddings}\label{subsec-lat-emb}

Let $\Gamma$ be a discrete group.
We denote by ${\rm vZ}(\Gamma)$ the virtual center of $\Gamma$, i.e., the subgroup of $\Gamma$ consisting of all elements that centralize a finite index subgroup of $\Gamma$.
Note that $\Gamma$ is ICC if and only if ${\rm vZ}(\Gamma)$ is trivial. 
We set $\bar{\Gamma}=\Gamma /{\rm vZ}(\Gamma)$ and denote by $q\colon \Gamma \rightarrow \bar{\Gamma}$ the quotient map.
If ${\rm vZ}(\Gamma)$ is finite, then $\bar{\Gamma}$ is ICC.

We determine locally compact second countable groups containing a lattice isomorphic to $\Gamma$ on the assumption that ${\rm vZ}(\Gamma)$ is finite and $\bar{\Gamma}$ is coupling rigid with respect to $\comm(\bar{\Gamma})$.
We note that on this assumption, $\Gamma$ is coupling rigid with respect to $(\comm(\bar{\Gamma}), {\bf i}\circ q)$, where ${\bf i}\colon \bar{\Gamma}\rightarrow \comm(\bar{\Gamma})$ is the natural homomorphism.
This application of coupling rigidity is originally introduced by Furman \cite{furman-lat}. 
The following theorem is proved along the argument in Section 2 of \cite{furman-lat}.

\begin{thm}
Let $\Gamma$ be a discrete group such that ${\rm vZ}(\Gamma)$ is finite and $\comm(\bar{\Gamma})$ is countable.
Suppose that $\bar{\Gamma}$ is coupling rigid with respect to $\comm(\bar{\Gamma})$.
Let $H$ be a locally compact second countable group and $\tau \colon \Gamma \rightarrow H$ an injective homomorphism such that $\tau(\Gamma)$ is a lattice in $H$.
Then there exists a continuous homomorphism $\Phi_0\colon H\rightarrow \comm(\bar{\Gamma})$ satisfying the following:
\begin{itemize} 
\item The group $K=\ker \Phi_0$ is compact.
\item The equality $\Phi_0(\tau(\gamma))={\bf i}\circ q(\gamma)$ holds for any $\gamma \in \Gamma$.
\item The group $H_0=\Phi_0^{-1}({\bf i}(\bar{\Gamma}))$ is a finite index subgroup of $H$.
\item The group $C=\tau(\Gamma)\cap K$ is equal to $\tau({\rm vZ}(\Gamma))$.
\item Let $\tau(\Gamma)\ltimes K$ be the semi-direct product defined by the action of $\tau(\Gamma)$ on $K$ by conjugation.
We then have the short exact sequence
\[1\rightarrow C\stackrel{j}{\rightarrow}\tau(\Gamma)\ltimes K\stackrel{p}{\rightarrow}H_0\rightarrow 1\]
of groups, where the homomorphisms $j$, $p$ are defined by $j(c)=(c, c^{-1})$ and $p(\tau(\gamma), k)=\tau(\gamma)k$, respectively, for $c\in C$, $\gamma \in \Gamma$ and $k\in K$. 
\end{itemize}
In particular, if ${\rm vZ}(\Gamma)$ is trivial, then $H_0$ is isomorphic to $\tau(\Gamma)\ltimes K$.
\end{thm}


\subsection{Outer automorphism groups of equivalence relations}

Let $\Gamma$ be an ICC discrete group.
Assuming that $\Gamma$ is coupling rigid with respect to $\comm(\Gamma)$, we compute the outer automorphism groups of the equivalence relations for certain generalized Bernoulli actions of $\Gamma$.
We refer to \cite{kida-out} for the notation used below without a precise definition.
Let $\Gamma \c (X, \mu)$ be an ergodic f.f.m.p.\ action, called $\alpha$, and $\mathcal{R}$ the discrete measured equivalence relation associated with the action. 
We denote by $\varepsilon \colon \aut(\mathcal{R})\rightarrow {\rm Out}(\mathcal{R})$ the quotient map.
Let $\aut^*(\alpha)$ be the subgroup of $\aut(\mathcal{R})$ consisting of all $f\in \aut(X, \mu)$ such that there exists $\pi \in \aut(\Gamma)$ with
\[f(\gamma x)=\pi(\gamma)f(x),\quad \forall \gamma \in \Gamma,\ \textrm{a.e.\ }x\in X.\]
We define the subgroup ${\sf A}^*(\alpha)$ of ${\rm Out}(\mathcal{R})$ as $\varepsilon(\aut^*(\alpha))$.
The following theorem tells us that in some cases, the computation of ${\rm Out}(\mathcal{R})$ is reduced to the computation of $\aut^*(\alpha)$ or ${\sf A}^*(\alpha)$.

\begin{thm}\label{thm-out-ineq}
Let $\Gamma$ be an ICC discrete group such that $\comm(\Gamma)$ is countable and $\Gamma$ is coupling rigid with respect to $\comm(\Gamma)$.
We define $S(\Gamma)$ as the set of all conjugacy classes of ${\bf i}(\Gamma)$ in $\comm(\Gamma)$.
Let $\Gamma \c (X, \mu)$ be an ergodic f.f.m.p.\ action and call it $\alpha$.
Then we have the inequality
\[[{\rm Out}(\mathcal{R}): {\sf A}^*(\alpha)]\leq |S(\Gamma)|.\]
Moreover, the equality can be attained for some ergodic f.f.m.p.\ action $\alpha$ if $|S(\Gamma)|$ is finite.
If the action $\alpha$ is aperiodic, then ${\rm Out}(\mathcal{R})={\sf A}^*(\alpha)$.
\end{thm}

This theorem is proved along argument in the proof of Theorems 1.1 and 1.2 in \cite{kida-out}, where the theorem is proved when $\Gamma$ is the mapping class group.
Note that if ${\bf i}(\Gamma)$ is a finite index subgroup of $\comm(\Gamma)$, then $|S(\Gamma)|$ is finite.

The following theorem computes the group ${\rm Out}(\mathcal{R})$ if $\mathcal{R}$ is associated with a certain generalized Bernoulli action.
The proof is obtained along argument in the proof of Proposition 5.3 in \cite{kida-out}.

\begin{thm}
Let $\Gamma$ be the group in Theorem \ref{thm-out-ineq}, $A$ an infinite proper subgroup of $\Gamma$ with ${\rm LQN}_{\Gamma}(A)=A$, and $(X, \mu)$ a non-trivial standard probability space, where a standard probability space is said to be non-trivial if there is no point in it whose measure is one.
If the generalized Bernoulli action $\Gamma \c (X, \mu)^{\Gamma /A}$ is essentially free, then the outer automorphism group of the associated equivalence relation is isomorphic to the group
\[\frac{\aut(A<\Gamma)}{\ad(\Gamma)\cap \aut(A<\Gamma)}\times \aut(X, \mu),\]
where $\aut(A<\Gamma)$ is the subgroup of $\aut(\Gamma)$ consisting of all automorphisms of $\Gamma$ preserving $A$, and $\ad(\Gamma)$ is the group of inner automorphisms of $\Gamma$.
\end{thm}

Let $\Gamma=\Gamma_1\ast_A\Gamma_2$ be an amalgamated free product of discrete groups.
For each $i=1, 2$, if $[\Gamma_i: A]=\infty$, then the equality ${\rm LQN}_{\Gamma}(\Gamma_i)=\Gamma_i$ holds.


\subsection{Groups ME to free products}

Let $\Gamma_1$ and $\Gamma_2$ be infinite discrete groups satisfying property (T).
Let $\Gamma =\Gamma_1\ast \Gamma_2$ be the free product of them.
We denote by $T$ the associated Bass-Serre tree and denote by $V(T)=\Gamma /\Gamma_1\sqcup \Gamma /\Gamma_2$ the set of vertices of $T$.
We define $B$ as the group of bijections on $V(T)$ equipped with the standard Borel structure associated with the pointwise convergence topology.
Let $\imath \colon \Gamma \rightarrow B$ denote the homomorphism arising from the action of $\Gamma$ on $V(T)$.
Using Theorem \ref{thm-adams-spa} as in the proof of Theorem \ref{thm-general-coup-tree}, we can prove the following:

\begin{prop}
If $\Gamma_1$ and $\Gamma_2$ are infinite discrete groups satisfying property (T), then $\Gamma =\Gamma_1\ast \Gamma_2$ is coupling rigid with respect to $(B, \imath)$.
\end{prop}

The group $\Gamma$ is not generally coupling rigid with respect to the automorphism group $\aut^*(T)$ of $T$ because $\comm(\Gamma)$ is not naturally embedded in $\aut^*(T)$.

Let $\Lambda$ be a discrete group and $\Sigma$ a coupling of $\Gamma$ and $\Lambda$.
Applying Theorem \ref{thm-furman-rep}, we obtain a homomorphism $\rho \colon \Lambda \rightarrow B$ with $\ker \rho$ finite and an almost $(\Gamma \times \Lambda)$-equivariant Borel map $\Phi \colon \Sigma \rightarrow (B, \imath, \rho)$. 
Along argument of the same kind as in the proof of Lemma \ref{lem-small-coup}, we can show the following:

\begin{lem}
In the above notation, we pick $v\in V(T)$.
Define $\stab(v)$ as the stabilizer of $v$ in $B$ and set
\[\Sigma_v=\Phi^{-1}(\stab(v)),\quad \Gamma_v=\imath^{-1}(\imath(\Gamma)\cap \stab(v)),\quad \Lambda_v=\rho^{-1}(\rho(\Lambda)\cap \stab(v)).\]
If $\rho$ and $\Phi$ are replaced appropriately as in Section \ref{subsec-circ}, then $\Sigma_v$ has positive measure and is a coupling of $\Gamma_v$ and $\Lambda_v$.
In particular, there exist subgroups $\Lambda_1$ and $\Lambda_2$ of $\Lambda$ with $\Gamma_1\sim_{\rm ME}\Lambda_1$ and $\Gamma_2\sim_{\rm ME}\Lambda_2$.
\end{lem}

For any two distinct vertices $v_1, v_2\in V(T)$, we can prove that $\Lambda_{v_1}\cap \Lambda_{v_2}$ is finite.
It might be interesting to ask whether the relation between $\Lambda_{v_1}$ and $\Lambda_{v_2}$ is nearly free or not.
We note that if $\Gamma_1$ and $\Gamma_2$ are infinite discrete groups containing a common finite subgroup $F$, then $\Gamma_1\ast \Gamma_2$ and $(\Gamma_1\ast_F\Gamma_2)\ast F$ are ME.
This is proved by using twisted actions studied in Section \ref{sec-twist}.



\end{document}